\documentclass[12pt,reqno]{amsart}
\usepackage{fullpage}

\newtheorem{theorem}{Theorem}[section]
\newtheorem{lemma}[theorem]{Lemma}

\newtheorem{cor}[theorem]{Corollary}
\newtheorem{conj}[theorem]{Conjecture}

\usepackage{graphicx}
\usepackage{color}
\usepackage[dvipsnames]{xcolor}
\usepackage{subfigure}
\usepackage{amssymb}
\usepackage{amsmath,mathrsfs}
\usepackage{colonequals}
\usepackage{upquote}
\usepackage[backref=page]{hyperref}
\theoremstyle{definition}
\newtheorem{definition}[theorem]{Definition}

\newtheorem{remark}[theorem]{Remark}

\renewcommand{\subset}{\subseteq}

\renewcommand{\epsilon}{\varepsilon}

\newcommand{\abs}[1]{\left|#1\right|}                   % Absolute value notation
\newcommand{\vnorm}[1]{\left\|#1\right\|}    % norm notation
\newcommand{\vnormf}[1]{\|#1\|}                         % norm notation, forced to be small
\newcommand{\vnormt}[1]{\left\|#1\right\|}    % norm notation
\newcommand{\Z}{\mathbb{Z}}                             % Blackboard notation
\newcommand{\N}{\mathbb{N}}
\newcommand{\E}{\mathbb{E}}

\newcommand{\R}{\mathbb{R}}

\renewcommand{\d}{\mathrm{d}}
\newcommand{\embolden}[1]{\textbf {#1}}

\newcommand{\sdimn}{n}
\newcommand{\adimn}{n}

\newcommand{\rhocurp}{.1082}
\newcommand{\rhocurn}{-.0234}
\newcommand{\rhocurnabs}{.0234}
\newcommand{\aprx}{.98937}

\begin{document}

\title{Three Candidate Plurality is Stablest\\ for Correlations at most 1/10}

\author{Steven Heilman}
\address{Department of Mathematics, University of Southern California, Los Angeles, CA 90089-2532}
\email{stevenmheilman@gmail.com}
\date{\today}
%\thanks{S. H. is Supported by NSF Grant CCF 1911216}
%60E15, 60G15, 53A10, 58E30
%\subjclass[2010]{60E15, 53A10, 60G15, 58E30}

\begin{abstract}
We prove the three candidate Plurality is Stablest Conjecture of Khot-Kindler-Mossel-O'Donnell from 2005 for correlations $\rho$ satisfying $-1/43<\rho<1/10$: the Plurality function is the most noise stable three candidate election method with small influences, when the corrupted votes have correlation $-1/43<\rho<1/10$ with the original votes.  The previous best result of this type only achieved positive correlations at most $10^{-10^{10}}$.  Our result follows by solving the three set Standard Simplex Conjecture of Isaksson-Mossel from 2011 for all correlations $-1/43<\rho<1/10$.

The Gaussian Double Bubble Theorem corresponds to the case $\rho\to1^{-}$, so in some sense, our result is a generalization of the Gaussian Double Bubble Theorem.  Our result is also notable since it is the first result for any $\rho<0$, which is the only relevant case for computational hardness of MAX-3-CUT.  In fact, assuming the Unique Games Conjecture, we show that MAX-3-CUT is NP-hard to approximate within a multiplicative factor of $.98937$, which improves on the known (unconditional) NP-hardness of approximation within a factor of $1-(1/102)$, proven in 1997.  As an additional corollary, we conclude that three candidate Borda Count is stablest for all $-1/43<\rho<1/10$.
\end{abstract}
\maketitle
\setcounter{tocdepth}{1}
\tableofcontents
%
%
% arxiv subjects: math.PR, math.DG, cs.CC?
%
%  MSC:    60E15, 60G15, 53A10, 58E30
%
%  keywords: convex, symmetric, Gaussian, minimal surface, calculus of variations

\section{Introduction}\label{secintro}

Motivated by a result of Bourgain \cite{bourgain02} in discrete Fourier analysis and by computational hardness results for MAX-CUT \cite{khot07}, Mossel, O'Donnell and Oleszkiewicz proved the Majority is Stablest Theorem \ref{thm0} below \cite{mossel10}.  This Theorem says that the majority function is the most noise stable voting method, such that each candidate has an equal chance of winning the election, and such that each voter has a small influence on the election's outcome.  Here we interpret a function $f\colon\{-1,1\}^{n}\to\{-1,1\}$ as a voting method with two candidates, denoted as $1$ and $-1$.  That is, the input of $f$ are votes $x=(x_{1},\ldots,x_{n})\in\{-1,1\}^{n}$, where the $i^{th}$ person's vote is for candidate $x_{i}\in\{-1,1\}$, and the winner of the election is $f(x)$ when the votes are $x$.

The Majority is Stablest Theorem can also be motivated by social choice theory, which seeks ``optimal'' voting methods.  For a mathematical discussion of social choice theory, see e.g. \cite{odonnell14,mossel10,isaksson11,kalai02}.  Noise stability has also been investigated in percolation, e.g. in \cite{benjamini99}.

Another main motivation of Theorem \ref{thm0} was \cite{khot07} computational hardness of MAX-CUT .  The MAX-CUT problem asks for the partition of the vertices of a finite undirected graph into two sets that maximizes the number of edges going between the two sets.  The MAX-CUT problem is NP-hard, so if P$\neq$NP, no polynomial time algorithm can solve it.  Moreover, it was shown that it is NP-hard to find a partition of a graph that is about 94\% as much as the largest partition value \cite{hastad01} \cite{trevisan00} (for every $\epsilon>0$, a multiplicative approximation of $16/17+\epsilon$ is NP-hard).  However, this NP-hardness result does not quite match the best known polynomial time algorithm for MAX-CUT.  For any $\epsilon>0$, the semidefinite programming algorithm of \cite{goemans95} approximates the MAX-CUT problem in polynomial time within a multiplicative factor of $\alpha_{2}-\epsilon$, where
$$\alpha_{2}
\colonequals
\inf_{-1\leq\rho\leq 1}\frac{2}{1}\frac{\frac{1}{\pi}\arccos(\rho)}{1-\rho}\\
\approx.87856720578.
$$
That is, it is possible to obtain in polynomial time a partition of the graph where the number of edges going between the two partition elements is at least 87.8\% as much as the largest partition value.

Assuming the Unique Games Conjecture, the gap between the $16/17$ hardness and the $.87856\ldots$ approximation can be closed exactly.

\begin{theorem}[\embolden{Sharp Hardness for MAX-CUT}, {\cite[Theorem 1]{khot07}}]\label{thm8}
Assume that the Unique Games Conjecture is true \cite{khot02,khot10a,khot18}.  Then, for any $\epsilon>0$, it is NP-hard to approximate MAX-CUT within a multiplicative factor of $\alpha_{2}+\epsilon$.
\end{theorem}

Theorem \ref{thm8} is called a sharp hardness result since the quantity $\alpha_{2}$ gives an exact barrier between tractability (i.e. the polynomial time algorithm of \cite{goemans95}) and intractability (as stated in Theorem \ref{thm8}).

For a statement of the Unique Games Conjecture, see \cite{khot02,khot10a}.  For some recent progress demonstrating that the conjecture is ``half way proven,'' see \cite{khot18}.

One main ingredient in the proof of Theorem \ref{thm8} was Theorem \ref{thm0}.

\begin{theorem}[\embolden{Majority is Stablest}, {\cite{mossel10}}]\label{thm0}
Let $\epsilon>0$ and let $0<\rho<1$.  Then there exists a $\tau>0$ such that the following holds.  Let $n$ be a positive integer.  Let $f\colon\{-1,1\}^{n}\to\{-1,1\}$ satisfy $\sum_{x\in\{-1,1\}^{n}}f(x)=0$ and $\max_{1\leq i\leq n}\mathrm{Inf}_{i}(f)\leq\tau$.  Then $$S_{\rho}(f)\leq1-\frac{2}{\pi}\arccos(\rho)+\epsilon.$$
\end{theorem}
For each $1\leq i\leq n$, we defined the $i^{th}$ \textbf{influence} of $f$ to be
$$\mathrm{Inf}_{i}(f)\colonequals 2^{-n}\sum_{x\in\{-1,1\}^{n}}\frac{1}{2}\abs{f(x)-f(x_{1},\ldots,x_{i-1},-x_{i},x_{i+1},\ldots,x_{n})},$$
and we defined the \textbf{noise stability} of $f$ with parameter $\rho\in(-1,1)$ to be
$$S_{\rho}(f)\colonequals2^{-n}\sum_{x\in\{-1,1\}^{n}}f(x)
\sum_{y\in\{-1,1\}^{n}}((1-\rho)/2)^{\frac{\vnorm{x-y}_{1}}{2}}
((1+\rho)/2)^{n-\frac{\vnorm{x-y}_{1}}{2}}f(y).$$
Equivalently, $S_{\rho}(f)$ is the expected value $\E(f(X)f(Y))$, where $X$ is a uniformly random element of $\{-1,1\}^{n}$, and $Y=(Y_{1},\ldots,Y_{n})$ has independent $pm1$ entries with $\E X_{i}Y_{i}=\rho$ for all $1\leq i\leq n$.  We also used $\vnorm{x}_{1}\colonequals\sum_{i=1}^{n}\abs{x_{i}}$ for all $x=(x_{1},\ldots,x_{n})\in\R^{n}$.

In Theorem \ref{thm0}, the quantity $1-\frac{2}{\pi}\arccos(\rho)$ can also be written as
\begin{equation}\label{zero0}
\lim_{n\to\infty}S_{\rho}(\mathrm{Maj}_{n}),
\end{equation}
where $\mathrm{Maj}_{n}(x)\colonequals\mathrm{sign}(x_{1}+\cdots+x_{n})$, for all $x=(x_{1},\ldots,x_{n})\in\{-1,1\}^{n}$, is the majority function.  (The value of the limit in \eqref{zero0} does not depend on the definition of the sign of zero.)

Note that the name ``Majority is Stablest'' for Theorem \ref{thm0} is a slight misnomer, since for any particular $n$, the majority function $\mathrm{Maj}_{n}$ is not guaranteed to exactly maximize noise stability, due to the $+\epsilon$ term in the inequality of Theorem \ref{thm0}.  Moreover, $S_{\rho}(-\mathrm{Maj}_{n})=S_{\rho}(\mathrm{Maj}_{n})$, i.e. anti-majority has the same noise stability as the majority function.

The full Majority is Stablest Theorem from \cite{mossel10} is actually more general than Theorem \ref{thm0}, since the constraint $2^{-n}\sum_{x\in\{-1,1\}^{n}}f(x)=0$ can be replaced with the following constraint: for some fixed $-1<a<1$, we have $2^{-n}\sum_{x\in\{-1,1\}^{n}}f(x)=a$.  Under this assumption, a shifted majority function is most noise stable.

\subsection{From Two Sets to Three}

Theorem \ref{thm8} addresses the MAX-CUT problem, where a graph is split into two pieces.  One can consider a similar problem for splitting a graph into $k\geq3$ pieces.  This problem is called the MAX-$k$-CUT problem.

\begin{definition}[\embolden{MAX-$k$-CUT}]\label{mkcdef}
Let $w\colon\{1,\ldots,n\}^{2}\to[0,\infty)$ be an $n\times n$ symmetric matrix with $w(i,i)=0$ for all $1\leq i\leq n$.  Fix $k\geq2$.  The goal of the MAX-$k$-CUT problem is to find a partition $P_{1},\ldots,P_{k}$ of $\{1,\ldots,n\}$ that maximizes the quantity $\sum_{1\leq i<j\leq k}\sum_{\ell\in P_{i},m\in P_{j}} w(\ell,m)$.
\end{definition}
Choosing $k=2$ shows MAX-CUT is the same as MAX-2-CUT, when $w$ is the adjacency matrix of an $n$-vertex graph.

It was conjectured in \cite{khot07} that an analogue of Theorem \ref{thm8} should be true for MAX-$k$-CUT.  However, this result has been open since it was formulated.

Theorem \ref{thm0} was proven as a Corollary of the so-called invariance principle \cite{mossel10} (see also \cite{chat06,rotar79}).  That is, the main insight used to prove Theorem \ref{thm0} was that Theorem \ref{thm0} is equivalent to a continuous inequality for Euclidean sets equipped with the Gaussian measure.  Such a continuous inequality was proven already in 1985 \cite{borell85}.  This inequality says that the noise stability of a set of fixed Gaussian measure is maximized by half spaces.  (We define the noise stability of a Euclidean set in Definition \ref{nsdef} below.) The equivalence between the discrete problem of Theorem \ref{thm0} and Borell's inequality in Euclidean space \cite{borell85} was proven via a nonlinear generalization of the Berry-Esse\'{e}n Central Limit Theorem, called an invariance principle.

If we try to generalize Theorem \ref{thm0} to product domains e.g. to functions $f\colon\{1,2,3\}^{n}\to\{1,2,3\}$, then we obtain the so-called Plurality is Stablest Conjecture \cite{khot07,isaksson11}.  As in \cite{mossel10}, this new conjecture is equivalent to a continuous problem in Euclidean space \cite{isaksson11}, which is a $3$-set generalization of the inequality of Borell \cite{borell85}.

 In \cite{isaksson11}, the arguments of \cite{khot07,mossel10} were adapted to show that sharp computational hardness for the MAX-$k$-CUT problem would follow from a $k$-set generalization of the inequality of Borell \cite{borell85}.  Despite this equivalence, the $k$-set analogue of Borell's inequality (known as the Standard Simplex Conjecture) was not known to be true or false.  (The $k=3$ case of this conjecture is stated in Theorem \ref{thm1} below.)  Moreover, previously studied proofs of Borell's inequality did not generalize to $k\geq3$ sets, despite the development of many different proofs of Borell's inequality \cite{borell85,ledoux94,ledoux96,bobkov97,burchard01,borell03,mossel15,mossel12,eldan13}.

Recently, we have been developing calculus of variations methods in order to prove the Conjecture of \cite{isaksson11}, i.e. a generalization of Theorem \ref{thm0} to other finite product domains.  Such methods originated in \cite{colding12a} (itself inspired by e.g. \cite{simons68,perelman02}), where a monotone quantity for the mean curvature flow was created and its stable critical points were investigated.  This monotone quantity was a supremum of Gaussian surface area.  It was unclear if the methods of \cite{colding12a} could apply to the noise stability functional in Definition \ref{nsdef} until \cite{heilman20d}, where these methods were adapted to show that the $k$-set Standard Simplex Conjecture in $\R^{n}$ reduces to the same conjecture in $\R^{k-1}$, for any $n\geq k-1$.  This proof method can also prove Borell's original inequality \cite{heilman20d,heilman21}.  The result of \cite{heilman20d} also circumvented a difficulty identified in \cite{heilman14}.  Although Theorem \ref{thm0} holds when the average value of $f$ is fixed to be some number $-1<a<1$, the three candidate analogue of the Majority is Stablest Theorem can \textit{only} be true when the voting method takes all of its values with equal probability.  Likewise, the arguments used to prove the $k=3$ set case of the Standard Simplex Conjecture must \textit{only} work when each of the Euclidean sets in question has Gaussian measure $1/3$.  For other measure constraints, it may be impossible to easily describe the optimal sets, as e.g. observed in attempts to optimize Krivine's functional as it relates to Grothendieck's constant \cite{braverman11}.

In this paper, we prove the $k=3$ set analogue of Borell's inequality (i.e. the $k=3$ set Standard Simplex Conjecture) for all correlation parameters $\rhocurn\leq\rho\leq\rhocurp$, thereby proving the $3$ candidate case of the Plurality is Stablest Conjecture, i.e. the $3$ candidate analogue of Theorem \ref{thm0} for all correlation parameters $\rhocurn\leq\rho\leq\rhocurp$.  These conjectures have remained open since 2006 \cite{khot07}.  The current paper and the previous results \cite{heilman20d,heilman12} seem to be the only works verifying cases of the Plurality is Stablest conjecture.

\subsection{Some Notation}

For any $x=(x_{1},\ldots,x_{\sdimn}),y=(y_{1},\ldots,y_{\sdimn})\in\R^{\sdimn}$, denote
$$\langle x,y\rangle\colonequals\sum_{i=1}^{n}x_{i}y_{i},\qquad\vnorm{x}\colonequals\langle x,x\rangle^{1/2}.$$
Denote the $(\sdimn-1)$-dimensional sphere in $\R^{\sdimn}$ as
$$S^{\sdimn-1}\colonequals\{x\in\R^{\sdimn}\colon \vnorm{x}=1\}.$$
%%%%$$B^{\sdimn}\colonequals\{x\in\R^{\sdimn}\colon \vnorm{x}\leq1\}.$$
For any integer $k\geq2$, denote the standard simplex in $\R^{k}$ as
$$\Delta_{k}\colonequals\{x\in\R^{k}\colon \sum_{i=1}^{k}x_{i}=1,\, x_{i}\geq0,\,\forall\,1\leq i\leq k\}.$$
Define the Gaussian density function in $\R^{\adimn}$ as
$$\gamma_{\adimn}(x)\colonequals(2\pi)^{-\adimn/2}e^{-\vnorm{x}^{2}/2},\qquad\forall\,x\in\R^{\adimn}.$$
When $A\subset\R^{\adimn}$ is a measurable set, denote the Gaussian measure of $A$ as $\gamma_{\adimn}(A)\colonequals\int_{A}\gamma_{\adimn}(x)\,\d x$.

\begin{definition}[\embolden{Ornstein-Uhlenbeck Operator}]\label{oudef}
Let $-1<\rho<1$.  Let $f\colon\R^{\adimn}\to[0,1]$ be measurable.  Define the \textbf{Ornstein-Uhlenbeck} operator applied to $f$ by
$$T_{\rho}f(x)\colonequals \int_{\R^{\adimn}}f(\rho x+y\sqrt{1-\rho^{2}})\gamma_{\adimn}(x)\,\d x,\qquad\forall\,x\in\R^{\adimn}.$$
\end{definition}

\subsection{Main Results}

Below, we say that $\Omega_{1},\Omega_{2},\Omega_{3}\subset\R^{\adimn}$ is a \textbf{partition} of $\R^{\adimn}$ if $\Omega_{1},\Omega_{2},\Omega_{3}$ are measurable, $\Omega_{i}\cap\Omega_{j}=\emptyset$ for all $1\leq i<j\leq 3$, and $\cup_{i=1}^{3}\Omega_{i}=\R^{\adimn}$.

Below, we also let $\Theta_{1},\Theta_{2},\Theta_{3}\subset\R^{2}$ be a partition of $\R^{2}$ into three disjoint sectors (cones) each with cone angle $2\pi/3$, centered at the origin.

\begin{theorem}[\embolden{Standard Simplex Conjecture, $3$ Sets}, {\cite[Conjecture 1.4]{isaksson11}}]\label{thm1}
Let $\adimn\geq2$.  Let $\Omega_{1},\Omega_{2},\Omega_{3}$ be a partition of $\R^{\adimn}$.

If $0<\rho\leq\rhocurp$, and if $\gamma_{\adimn}(\Omega_{i})=1/3$ for all $1\leq i\leq 3$, then
$$\sum_{i=1}^{3}\int_{\R^{\adimn}}1_{\Omega_{i}}(x)T_{\rho}1_{\Omega_{i}}(x)\,\gamma_{\adimn}(x)\,\d x
\leq \sum_{i=1}^{3}\int_{\R^{2}}1_{\Theta_{i}}(x)T_{\rho}1_{\Theta_{i}}(x)\,\gamma_{2}(x)\,\d x.$$

If $\rhocurn\leq\rho<0$, (with no restriction on the measures $\gamma_{\adimn}(\Omega_{i})$), then
$$\sum_{i=1}^{3}\int_{\R^{\adimn}}1_{\Omega_{i}}(x)T_{\rho}1_{\Omega_{i}}(x)\,\gamma_{\adimn}(x)\,\d x
\geq \sum_{i=1}^{3}\int_{\R^{2}}1_{\Theta_{i}}(x)T_{\rho}1_{\Theta_{i}}(x)\,\gamma_{2}(x)\,\d x.$$
\end{theorem}
Theorem \ref{thm1} should hold for all $-1<\rho<1$ \cite{isaksson11}.  The two main inequalities in Theorem \ref{thm1} are equalities only when $\Omega_{i}=\Theta_{i}\times\R^{n-1}$ for all $1\leq i\leq 3$ (up to measure zero changes and rotations applied to the sets).  That is, we actually prove a stronger ``stability'' version of Theorem \ref{thm1} in \eqref{two10} and \eqref{eight1} when $\sdimn=2$:

If $0<\rho\leq\rhocurp$, and if $\gamma_{\adimn}(\Omega_{i})=1/3$ for all $1\leq i\leq 3$, then
\begin{flalign*}
&\sum_{i=1}^{3}\int_{\R^{\adimn}}1_{\Omega_{i}}(x)T_{\rho}1_{\Omega_{i}}(x)\,\gamma_{\adimn}(x)\,\d x
-\sum_{i=1}^{3}\int_{\R^{2}}1_{\Theta_{i}}(x)T_{\rho}1_{\Theta_{i}}(x)\,\gamma_{2}(x)\,\d x\\
&\qquad\leq
(-.3+3(\rho+\rho^{2}))\int_{r=0}^{\infty}r(1-e^{-\rho r/2})e^{-r^{2}/2}\sum_{i=1}^{3}(\sigma(\Omega_{i}\cap r S^{1})-1/3)^{2}\frac{3}{2}\,\d r.
\end{flalign*}

If $\rhocurn\leq\rho<0$, (with no restriction on the measures $\gamma_{\adimn}(\Omega_{i})$), then
\begin{flalign*}
&\sum_{i=1}^{3}\int_{\R^{\adimn}}1_{\Omega_{i}}(x)T_{\rho}1_{\Omega_{i}}(x)\,\gamma_{\adimn}(x)\,\d x
- \sum_{i=1}^{3}\int_{\R^{2}}1_{\Theta_{i}}(x)T_{\rho}1_{\Theta_{i}}(x)\,\gamma_{2}(x)\,\d x\\
&\quad\geq
(2\cdot.75/13.2 -(3\sqrt{\pi/2}\rho+8\rho^{2}))\int_{r=0}^{\infty}r(1-e^{-\rho r/2})e^{-r^{2}/2}\sum_{i=1}^{3}(\sigma(\Omega_{i}\cap r S^{1})-1/3)^{2}\frac{3}{2}\,\d r.
\end{flalign*}
Here $\sigma$ denotes normalized (Haar) probability measure on the sphere $S^{1}$.  Note that the ``penalty'' terms we wrote above only check how far the measure of each set is from $1/3$ when restricted to sphere of radius $r$, though in the course of the proof we find that such restricted sets must be circular arcs.

As demonstrated in \cite{isaksson11}, Theorem \ref{thm1} has a discrete analogue, known as the Plurality is Stablest Conjecture.  This Conjecture was first formulated in \cite{khot07}.

\begin{theorem}[\embolden{Plurality is Stablest, $3$ Candidates, Informal Statement}, {\cite[page 9]{khot07}}, {\cite[Conjecture 1.9]{isaksson11}}]\label{thm2}
The Plurality function is the most noise stable three-candidate voting method with small influences, for all correlations $\rho$ satisfying $\rhocurn\leq\rho\leq\rhocurp$
\end{theorem}
More formally, Conjecture \ref{conj1} below holds when $m=3$ and $\rhocurn\leq\rho\leq\rhocurp$.

Theorem \ref{thm2} should hold for all $-1/2\leq \rho<1$ \cite{khot07,isaksson11}.  This is called the three candidate Plurality is Stablest Conjecture.  If this conjecture holds, then we would be able to conclude sharp hardness of approximation for the MAX-3-CUT problem, assuming the Unique Games Conjecture.

\begin{conj}[\embolden{Sharp Hardness for MAX-3-CUT}, {\cite{isaksson11}}]\label{thm3}
Assume that the Unique Games Conjecture is true \cite{khot02,khot10a,khot18}.  Assume the Plurality is Stablest Conjecture holds (for three candidates) \cite{isaksson11}.  Then, for any $\epsilon>0$, it is NP-hard to approximate MAX-3-CUT within a multiplicative factor of $\alpha_{3}+\epsilon$.
\end{conj}

As shown in \cite{isaksson11} using the formula from \cite{klerk04}, we have
\begin{flalign*}
\alpha_{3}
&\colonequals
\inf_{-\frac{1}{2}\leq\rho\leq 1}\frac{3}{2}\frac{1-\sum_{i=1}^{3}\int_{\R^{2}}1_{\Theta_{i}}(x)T_{\rho}1_{\Theta_{i}}(x)\,\gamma_{2}(x)\,\d x}{1-\rho}\\
&=\inf_{-\frac{1}{2}\leq\rho\leq 1}\frac{3}{2}\frac{1-3\Big(\frac{1}{9}+\frac{[\arccos(-\rho)]^{2} - [\arccos(\rho/2)]^{2}}{4\pi^{2}}\Big)}{1-\rho}\\
&=\frac{3}{2}\frac{1-3\Big(\frac{1}{9}+\frac{[\arccos(1/2)]^{2} - [\arccos(-1/4)]^{2}}{4\pi^{2}}\Big)}{1+1/2}
\approx.83600811464.
\end{flalign*}

%%%%% function to minimize
%rho=linspace(-1/2,.9,1000);
%plot(rho, (3/2)*(1 - 3*(  (1/9)+ (((acos(-rho)).^2 - (acos(rho/2)).^2) / (4* pi^2))) )./(1-rho) );
%%
%%   1- 3* (1/9 + ( (acos(1/2))^2  - (acos(-1/4))^2 )/(4* pi^2))
%
The semidefinite programming algorithm of \cite{frieze95} shows that Conjecture \ref{thm3} is sharp, since the polynomial time algorithm of \cite{frieze95} approximates the MAX-3-CUT problem with a multiplicative factor of $\alpha_{3}-\epsilon$, for any $\epsilon>0$.

A corollary of our main result Theorem \ref{thm1} is a weaker version of Conjecture \ref{thm3}.

\begin{theorem}[\embolden{Unique Games Hardness for MAX-3-CUT}]\label{thm4}
Assume that the Unique Games Conjecture is true \cite{khot02,khot10a,khot18}.  Then it is NP-hard to approximate MAX-3-CUT within a multiplicative factor of $\aprx$.
\end{theorem}

Although Theorem \ref{thm4} does not achieve the $.836\ldots$ hardness of approximation of Conjecture \ref{thm3}, the constant in Theorem \ref{thm4} does improve upon (to the author's knowledge) the best (unconditional) NP-hardness result for MAX-3-CUT \cite{kann97}, which says that MAX-3-CUT is NP-hard to approximate within a multiplicative factor of $1-1/(34(3))\approx .990\ldots$.

Another consequence of Theorem \ref{thm2} and the main result of \cite{heilman22c} is that Borda Count is the most noise stable three-candidate ranked choice voting method with small influences satisfying the Condorcet Loser Criterion, for all correlations $\rho$ satisfying $\rhocurn\leq\rho\leq\rhocurp$

\subsection{Sketch of Proof of Main Theorem}

Here we sketch the proof of Theorem \ref{thm1} when $\rho>0$.  First, the main result of \cite{heilman20d} (or \cite{heilman22d} for negative correlations) reduces Theorem \ref{thm1} to the case $n=2$.  Next, we use a spherical harmonic decomposition, as suggested as an approach to the conjectured vector-valued Borell inequality of \cite{hwang21} (though note that their inequality is presently unproven).  That is, we write the noise stability of a set $\Omega=\Omega_{i}\subset\R^{2}$ as an average over spheres.  Using Definition \ref{oudef} below, we have
\begin{flalign*}
&\sum_{i=1}^{3}\int_{\R^{2}}1_{\Omega_{i}}(x)T_{\rho}1_{\Omega_{i}}(x)\gamma_{2}(x)\,\d x\\
&=\sum_{i=1}^{3}\frac{1}{2\pi(1-\rho^{2})}\int_{r=0}^{\infty}\int_{s=0}^{\infty}rse^{\frac{-r^{2}-s^{2}}{2(1-\rho^{2})}}
\int_{u\in S^{1}}\int_{v\in S^{1}}
1_{rS^{1}\cap\Omega_{i}}(ru)1_{sS^{1}\cap\Omega_{i}}(sv)
e^{\frac{\rho rs\langle u,v\rangle}{1-\rho^{2}}}\d v\d u \d s \d r.
\end{flalign*}

One might try to maximize the ``spherical noise stability'' by fixing $r,s>0$ and then maximizing the two inner integrals in $u,v$ over choices of $\Omega_{1},\Omega_{2},\Omega_{3}$ constrained to the circles $rS^{1}$ and $sS^{1}$.  However, such a maximization would ruin the measure constraint in Theorem \ref{thm1}, since the two inner integrals are maximized when $rS^{1}\cap \Omega_{1}= rS^{1}$ and when $sS^{1}\cap \Omega_{1}=sS^{1}$ (with $\Omega_{2}=\Omega_{3}=\emptyset$).  In order to ameliorate this issue, it is natural to add and subtract the mean value $c_{i}(r)\colonequals \int_{rS^{1}\cap \Omega_{i}}\d v / 2\pi$ of $1_{\Omega_{i}}$ on the circle $r S^{1}$, by writing:

\begin{equation}\label{in1}
\begin{aligned}
&\sum_{i=1}^{3}\int_{\R^{2}}1_{\Omega_{i}}(x)T_{\rho}1_{\Omega_{i}}(x)\gamma_{2}(x)\,\d x\\
&=\sum_{i=1}^{3}\frac{1}{2\pi(1-\rho^{2})}\int_{r=0}^{\infty}\int_{s=0}^{\infty}rse^{\frac{-r^{2}-s^{2}}{2(1-\rho^{2})}}
\Big(c_{i}(r)c_{i}(s)\\
&\qquad\qquad\qquad+\int_{u\in S^{1}}\int_{v\in S^{1}}
[1_{rS^{1}\cap\Omega_{i}}(ru)- c_{i}(r)][1_{sS^{1}\cap\Omega_{i}}(sv) - c_{i}(s)]
e^{\frac{\rho rs\langle u,v\rangle}{1-\rho^{2}}}\d v\d u\Big)\d s \d r.
\end{aligned}
\end{equation}

Now, when $r,s>0$ are fixed, the two inner integrals in $u,v$ have a product of mean subtracted terms.  These terms together can be considered a ``mean subtracted spherical noise stability.''  Crucially, this quantity is no longer maximized when $rS^{1}\cap \Omega_{1}= rS^{1}$ and when $sS^{1}\cap\Omega_{1}= sS^{1}$.  Such a mean subtraction step was also used in the earlier work \cite{heilman12}, though using Hermite polynomials as a Fourier basis instead of spherical harmonics.  (It might be possible to use the methods of \cite{heilman12} to prove Theorem \ref{thm1}, but we have not tried to do so.)

The decomposition \eqref{in1} is natural since the mean subtracted spherical noise stability is exactly maximized when $\Omega_{1},\Omega_{2},\Omega_{3}$ intersected with $rS^{1}$ are three disjoint spherical arcs of angle $2\pi/3$, for all $r>0$.  This statement holds without any measure constraint on the sets.  Again, such a realization was crucial in \cite{heilman12}, in the context of Hermite polynomials, since we can temporarily ignore the measure constraint that $\gamma_{2}(\Omega_{i})=1/3$ for all $1\leq i\leq 3$.

However, this observation and \eqref{in1} are insufficient to prove Theorem \ref{thm1}, due to the first term (i.e. the product of the means $c_{i}(r)c_{i}(s)$) in \eqref{in1}.  This term is instead maximized when $rS^{1}\cap \Omega_{i}= rS^{1}$ for some $1\leq i\leq 3$, for each $r>0$.  (In fact, even if we constrain $\gamma_{2}(\Omega_{i})=1/3$ for all $1\leq i\leq 3$, the mean term in \eqref{in1} seems to be maximized when $\Omega_{1},\Omega_{2},\Omega_{3}\subset\R^{2}$ are disjoint annuli.)  Since the two terms of \eqref{in1} have different maximizers, a na\"{i}ve implementation of \eqref{in1} does not prove Theorem \ref{thm1}.  However, it turns out that we can upper bound the first $c_{i}(r)c_{i}(s)$ term in \eqref{in1} by the second mean subtracted noise stability term in \eqref{in1}, at least when the correlation $\rho>0$ is close to zero.  (See Lemma \ref{lemma28}, which uses both spherical and Hermite Fourier analysis, combined with Corollary \ref{cor1}, which uses a derivative estimate in the angular direction.)  Therefore, a straightforward spherical rearrangement (see e.g. Lemma \ref{lemma3}) shows that \eqref{in1} really is maximized when the intersection of each $\Omega_{i}$ with a circle centered at the origin is a circular arc of angle $2\pi/3$.  However, showing that the first term is \eqref{in1} is bounded by the second term relies on some estimates that only seem to hold when $\rho>0$ is small, so this is one main bottleneck in trying to solve the full problem for all correlations $\rho$.  Interestingly, some of these estimates (such as Lemma \ref{lemma28}) hold in $\R^{n}$ for any $n\geq2$, but not when $n=1$, since Lemma \ref{lemma28} uses integrability of the radial function $r\mapsto 1/r$ near $r=0$ in $\R^{n}$ when $n=2$.

Our above discussion focused on the case $\rho>0$.  The above strategy almost works in the case $\rho<0$.  However, the case $\rho<0$ presents an unexpected difficulty.  When $r\cdot s$ is small enough (depending on $\rho<0$), the last term in \eqref{in1} is minimized when $\Omega_{1},\Omega_{2},\Omega_{3}$ intersected with $rS^{1}$ are three disjoint spherical arcs of angle $2\pi/3$.  However, when $r\cdot s$ is large, this is no longer true!  When $r\cdot s$ is large, the last term in \eqref{in1} is minimized when $\Omega_{1},\Omega_{2}$ intersected with $rS^{1}$ are two disjoint spherical arcs of angle $\pi$ (and the intersection with $\Omega_{3}$ is empty).  That is, a straightforward spherical rearrangement can no longer apply when $\rho<0$, since when $r\cdot s$ is large, the optimizing sets do not agree with what they should be in Theorem \ref{thm1}.  Nevertheless, we can further split the last term of \eqref{in1} into two parts, one of which is maximized by three circular arcs with angle $2\pi/3$ (for all $r,s>0$), and the remaining term which is ``smaller'' in an $L_{2}$ sense, leading again to a spherical rearrangement argument.  Lastly, for technical reasons (such as proving that the optimal sets are ``low-dimensional''), when $\rho<0$ we need to deal with a bilinear version of noise stability, rather than the quadratic version used e.g. in \eqref{in1}.  (For similar technical reasons, the bilinear noise stability was used e.g. in \cite{heilman20d,heilman22d}.)

\subsection{On the Difficulty of Four or More Sets}

One might wonder if the results of this paper could apply to $k\geq4$ sets, since our result only holds for $k=3$ sets.  At present, some difficulties remain for $k\geq4$.  The key spherical rearrangement Lemma \ref{lemma3} is false when $k\geq4$.  This can be seen as a consequence of the Propeller Conjecture in $\R^{3}$ \cite{heilman11}.  The derivative of the noise stability of a partition at $\rho=0$ in $\R^{3}$ is maximized for three congruent flat cones with cone angle $2\pi/3$, rather than for four regular tetrahedral cones, as observed in e.g. \cite{heilman11}.  Consequently, Lemma \ref{lemma3} cannot hold for $k=4$ sets.  Put another way, if we start with a partition of Euclidean space into $4$ sets of Gaussian measure $1/4$ each, and we then restrict the sets to spheres of various radii (as in \eqref{in1}), then maximizing noise stability restricted to the spheres will result in three spherical simplices rather than four.  That is, the analogue of the final inequality \eqref{three1} is false when $k\geq4$.  Inequality \eqref{three1} shows that the noise stability (minus the measures of the sets) is larger when comparing three congruent cones to two half space of measure $1/2$ each.  But the analogous statement comparing four regular tetrahedral cones to three regular sectors is false.

Our strategy for $k=3$ sets relies on maximizing a mean subtracted noise stability on circles (i.e. the last term in \eqref{in1}), while ignoring any measure constraints on the sets.  It may be possible (in fact, it seems necessary) to use the measure constraints on the sets in order to consider at least $k\geq4$ sets $\Omega_{1},\ldots,\Omega_{k}$.  Likewise, incorporating first or second variation arguments (as in \cite{heilman20d,heilman22d}) might further constrain the sets under consideration when $k\geq4$, and possibly lead to progress when $k\geq4$.  Our proof of Theorem \ref{thm1} does not use any direct measure constraints or first/second variation arguments.

One main obstacle to proving Theorem \ref{thm1}, as identified in \cite{heilman14}, is that Theorem \ref{thm1} cannot hold for $\rho>0$ if the sets satisfy $(\gamma_{n}(\Omega_{1}),\gamma_{n}(\Omega_{2}),\gamma_{n}(\Omega_{3}))\neq(1/3,1/3,1/3)$.  More specifically, the sets optimizing noise stability are not the affine of simplicial cones with flat facets, unless $(\gamma_{n}(\Omega_{1}),\gamma_{n}(\Omega_{2}),\gamma_{n}(\Omega_{3}))=(1/3,1/3,1/3)$.  Therefore, a proof of Theorem \ref{thm1} must somehow \textit{only} work in the case $(\gamma_{n}(\Omega_{1}),\gamma_{n}(\Omega_{2}),\gamma_{n}(\Omega_{3}))=(1/3,1/3,1/3)$.  And indeed, our proof accomplishes this task.

In this work we have also identified a new obstacle to proving Theorem \ref{thm1} for $\rho<0$ which does not seem to have been identified before, as described above.  Namely, the spherical version of Theorem \ref{thm1} is false for spheres of large radii when $\rho<0$.

\subsection{Formal Statement of Plurality is Stablest}

Here we provide a formal statement of the Plurality is Stablest Conjecture, as written e.g. in \cite{heilman22d}.

Let $k\geq2$, $k\in\Z$.  If $g\colon\{1,\ldots,k\}^{\sdimn}\to\R$ and $1\leq i\leq\sdimn$, we denote
$$ \E(g)\colonequals k^{-\sdimn}\sum_{\omega\in\{1,\ldots,k\}^{\sdimn}} g(\omega)$$
$$\E_{i}(g)(\omega_{1},\ldots,\omega_{i-1},\omega_{i+1},\ldots,\omega_{\sdimn})\colonequals k^{-1}\sum_{\omega_{i}\in\{1,\ldots,k\}} g(\omega_{1},\ldots,\omega_{n})$$
$$\qquad\qquad\qquad\qquad\qquad\qquad\qquad\qquad\qquad\forall\,(\omega_{1},\ldots,\omega_{i-1},\omega_{i+1},\ldots,\omega_{\sdimn})\in\{1,\ldots,k\}^{n-1}.$$
Define also the $i^{th}$ \textbf{influence} of $g$, i.e. the influence of the $i^{th}$ voter of $g$, as
\begin{equation}\label{infdef}
\mathrm{Inf}_{i}(g)\colonequals \E [(g-\E_{i}g)^{2}].
\end{equation}

If $f\colon\{1,\ldots,k\}^{\sdimn}\to\Delta_{k}$, we denote the coordinates of $f$ as $f=(f_{1},\ldots,f_{k})$.  For any $\omega\in\Z^{\sdimn}$, we denote $\vnormt{\omega}_{0}$ as the number of nonzero coordinates of $\omega$.  The \textbf{noise stability} of $g\colon\{1,\ldots,k\}^{\sdimn}\to\R$ with parameter $\rho\in(-1/(k-1),1)$ is
\begin{flalign*}
S_{\rho} g
&\colonequals k^{-\sdimn}\sum_{\omega\in\{1,\ldots,k\}^{\sdimn}} g(\omega)\E_{\rho} g(\delta)\\
&=k^{-\sdimn}\sum_{\omega\in\{1,\ldots,k\}^{\sdimn}} g(\omega)\sum_{\sigma\in\{1,\ldots,k\}^{\sdimn}}\left(\frac{1+(k-1)\rho}{k}\right)^{\sdimn-\vnormt{\sigma-\omega}_{0}}
\left(\frac{1-\rho}{k}\right)^{\vnormt{\sigma-\omega}_{0}} g(\sigma).
\end{flalign*}
Equivalently, conditional on $\omega$, $\E_{\rho}g(\delta)$ is defined so that for all $1\leq i\leq\sdimn$, $\delta_{i}=\omega_{i}$ with probability $\frac{1+(k-1)\rho}{k}$, and $\delta_{i}$ is equal to any of the other $(k-1)$ elements of $\{1,\ldots,k\}$ each with probability $\frac{1-\rho}{k}$, and so that $\delta_{1},\ldots,\delta_{\sdimn}$ are independent.

\begin{remark}\label{rhobdrk}
The Plurality is Stablest Conjecture is only stated for $-1/(k-1)\leq\rho\leq 1$, whereas the Standard Simplex Conjecture is stated for all $-1<\rho<1$, since the discrete noise stability $S_{\rho}g$ only corresponds to an expected value when $\rho\geq -1/(k-1)$.
\end{remark}

The \textbf{noise stability} of $f\colon\{1,\ldots,k\}^{\sdimn}\to\Delta_{k}$ with parameter $\rho\in(-1/(k-1),1)$ is
$$S_{\rho}f\colonequals\sum_{i=1}^{k}S_{\rho}f_{i}.$$

For each $j\in\{1,\ldots,k\}$, let $e_{j}=(0,\ldots,0,1,0,\ldots,0)\in\R^{k}$ be the $j^{th}$ unit coordinate vector.  Define the \textbf{plurality} function $\mathrm{PLUR}_{k,\sdimn}\colon\{1,\ldots,k\}^{\sdimn}\to\Delta_{k}$ for $k$ candidates and $\sdimn$ voters such that for all $\omega\in\{1,\ldots,k\}^{\sdimn}$.
$$\mathrm{PLUR}_{k,\sdimn}(\omega)
\colonequals\begin{cases}
e_{j}&,\mbox{if }\abs{\{i\in\{1,\ldots,k\}\colon\omega_{i}=j\}}>\abs{\{i\in\{1,\ldots,k\}\colon\omega_{i}=r\}},\\
&\qquad\qquad\qquad\qquad\forall\,r\in\{1,\ldots,k\}\setminus\{j\}\\
\frac{1}{k}\sum_{i=1}^{k}e_{i}&,\mbox{otherwise}.
\end{cases}
$$

We can now state the more formal version of the Plurality is Stablest Conjecture from \cite[page 9]{khot07}, \cite[Conjecture 1.9]{isaksson11}.

\begin{conj}[\embolden{Plurality is Stablest, Discrete Version}]\label{conj1}
For any $k\geq2$, $\rho\in[-1/(m-1),1]$, $\epsilon>0$, there exists $\tau>0$ such that if $f\colon\{1,\ldots,k\}^{\sdimn}\to\Delta_{k}$ satisfies $\mathrm{Inf}_{i}(f_{j})\leq\tau$ for all $1\leq i\leq\sdimn$ and for all $1\leq j\leq k$, then
\begin{itemize}
\item If $\rho\geq0$ and if $\E f=\frac{1}{k}\sum_{i=1}^{k}e_{i}$, then
$$
S_{\rho}f\leq \lim_{\sdimn\to\infty}S_{\rho}\mathrm{PLUR}_{k,\sdimn}+\epsilon.
$$
\item If $-\frac{1}{m-1}\leq\rho<0$, then
$$
S_{\rho}f\geq \lim_{\sdimn\to\infty}S_{\rho}\mathrm{PLUR}_{k,\sdimn}-\epsilon.
$$
\end{itemize}
\end{conj}

We now state some remaining definitions.

\begin{definition}[\embolden{Correlated Gaussians}]\label{gausdef}
Let $-1<\rho<1$.  Let $G_{\rho}(x,y)$ denote the joint probability density function on $\R^{\sdimn}\times\R^{\sdimn}$ such that
\begin{equation}\label{gdef}
G_{\rho}(x,y)\colonequals\frac{1}{(2\pi)^{\sdimn}(1-\rho^{2})^{\sdimn/2}}e^{\frac{-\vnorm{x}^{2}-\vnorm{y}^{2}+2\rho \langle x,y\rangle}{2(1-\rho^{2})}}\qquad\forall\,x,y\in \R^{\sdimn}.
\end{equation}

We denote $X\sim_{\rho} Y$ when $(X,Y)\in\R^{n}\times \R^{n}$ have joint probability density function $G_{\rho}$.
\end{definition}

\begin{definition}[\embolden{Correlated Random Variables on the Sphere}]\label{corsph}
Let $G_{\rho}^{r,s}(u,v)$ denote the probability density function on $S^{\sdimn-1}\times S^{\sdimn-1}$ such that the first variable is uniform on $S^{\sdimn-1}$ and such that the second variable conditioned on the first has conditional density
$$G_{\rho}^{r,s}(v|u)\colonequals\frac{1}{z_{\rho,r,s}}e^{\frac{\rho r s\langle u,v\rangle}{1-\rho^{2}}},\qquad\forall\,v\in S^{\sdimn-1}.$$
Here $z_{\rho,r,s}$ is a normalizing constant, chosen so that $\int_{S^{\sdimn-1}}G_{\rho}^{r,s}(v|u)\d\sigma(v)=1$, where $\sigma$ denotes the uniform probability (Haar) measure on $S^{\sdimn-1}$.

We let $N_{\rho}^{r,s}$ denote the above distribution on $S^{\sdimn-1}\times S^{\sdimn-1}$ and we denote $(U,V)\sim N_{\rho}^{r,s}$ when $(U,V)\in S^{\sdimn-1}\times S^{\sdimn-1}$ have the distribution $N_{\rho}^{r,s}$.
\end{definition}

\begin{definition}[\embolden{Noise Stability}]\label{nsdef}
Let $-1<\rho<1$.  Let $\Omega\subset\R^{\adimn}$ be measurable.  Define the \textbf{noise stability} of $\Omega$ with correlation $\rho$, to be%, denoted $S_{\rho}(\Omega)$, to be
$$\int_{\R^{\adimn}}1_{\Omega}(x)T_{\rho}1_{\Omega}(x)\,\gamma_{\adimn}(x)\,\d x
\stackrel{\eqref{gdef}}{=}\int_{\R^{\adimn}}\int_{\R^{\adimn}}1_{\Omega}(x)1_{\Omega}(y)G_{\rho}(x,y)\,\d x \d y.$$
More generally, for any measurable $f\colon\R^{\adimn}\to\Delta_{k}$, define its noise stability with correlation $\rho$, to be%, denoted $S_{\rho}(f)$, to be
$$\int_{\R^{\adimn}}\langle f(x), T_{\rho}f(x)\rangle\gamma_{\adimn}(x)\,\d x.$$
\end{definition}

\subsection{Expected Value Notation}\label{enote}

\begin{itemize}
\item $\E$ with no subscript denotes expected value on a sphere with respect to the uniform (Haar) probability measure.
\item $\E_{(U,V)\sim N_{\rho}^{r,s}}$ denotes expected value with respect to $(U,V)$ from Definition \ref{corsph}.
\item $\underset{X\sim_{\rho}Y}{\E}$ denotes expected value with respect to $(X,Y)$ from Definition \ref{gausdef}.
\item $\E_{R,S}$ denotes expected value with respect to $R,S$ where $R=\vnorm{X},S=\vnorm{Y}$, and $X,Y$ are two standard $\rho$-correlated Gaussians, as in Definition \ref{gausdef}.
\item $\E_{\gamma}$ denotes expected value with respect to the Gaussian density $\gamma_{n}$.
\end{itemize}

\section{Fourier analysis of Spherical Noise Stability}

In this section, we derive some properties of the noise stability restricted to sets in a sphere.  Fix $r,s>0$ and let $0<\rho<1$.  Define $g\colon[-1,1]\to\R$ by
\begin{equation}\label{one0z}
g(t)=g_{\rho,r,s}(t)\colonequals\sqrt{\pi}\frac{\Gamma((n-1)/2)}{\Gamma(n/2)}\frac{e^{\frac{\rho rst}{1-\rho^{2}}}}{\int_{-1}^{1}(1-a^{2})^{\frac{\sdimn}{2}-\frac{3}{2}}e^{\frac{\rho rsa}{1-\rho^{2}}}\d a},\qquad\forall\,t\in[-1,1].
\end{equation}
Recall that, if $h\colon\R\to\R$ is continuous, then
\begin{flalign*}
\frac{1}{\mathrm{Vol}(S^{n-1})}\int_{S^{n-1}}h(y_{1})\d y
&=\frac{\mathrm{Vol}(S^{n-2})}{\mathrm{Vol}(S^{n-1})}\int_{-1}^{1}(1-t^{2})^{\frac{\sdimn}{2}-\frac{3}{2}}h(t)\d t\\
&=\frac{2\pi^{(n-1)/2}/\Gamma((n-1)/2)}{2\pi^{n/2}/\Gamma(n/2)}\int_{-1}^{1}(1-t^{2})^{\frac{\sdimn}{2}-\frac{3}{2}}h(t)\d t\\
&=\frac{1}{\sqrt{\pi}}\frac{\Gamma(n/2)}{\Gamma((n-1)/2)}\int_{-1}^{1}(1-t^{2})^{\frac{\sdimn}{2}-\frac{3}{2}}h(t)\d t.
\end{flalign*}
We have chosen the constants so that $1=\frac{1}{\mathrm{Vol}(S^{n-1})}\int_{S^{n-1}}h(y_{1})\d y$, when $h=g$.  When $h\colonequals1$, we have $\int_{-1}^{1}(1-t^{2})^{\frac{n}{2}-\frac{3}{2}}\d t=\sqrt{\pi}\frac{\Gamma((n-1)/2)}{\Gamma(n/2)}$, so %=c_{n}\frac{(n-3)!!}{(n-2)!!}$ where $c_{n}=\pi$ when $n$ is even, and $c_{n}=2$ when $n$ is odd,
\begin{equation}\label{one0p}
g(t)=g_{\rho,r,s}(t)\stackrel{\eqref{one0z}}{=}e^{\frac{\rho rst}{1-\rho^{2}}}\cdot \frac{\int_{-1}^{1}(1-a^{2})^{\frac{n}{2}-\frac{3}{2}}\d a}{\int_{-1}^{1}(1-a^{2})^{\frac{\sdimn}{2}-\frac{3}{2}}e^{\frac{\rho rsa}{1-\rho^{2}}}\d a},\qquad\forall\,t\in[-1,1].
\end{equation}

\begin{definition}[\embolden{Spherical Noise Stability}]\label{spnsdef}
Let $\rho\in(-1,1)$, $r,s>0$.  Let $f\colon S^{\sdimn-1}\to[0,1]$ be measurable.  Define $g=g_{\rho,r,s}\colon[-1,1]\to\R$ by \eqref{one0p}.  Define the smoothing operator $U_{g}$ applied to $f$ by
$$U_{g}f(x)\colonequals \int_{S^{\sdimn-1}}g(\langle x,y\rangle)f(y)\,\d\sigma(y),\qquad\forall x\in S^{\sdimn}.$$
Here $\sigma$ denotes the (normalized) Haar probability measure on $S^{\sdimn-1}$.  The \textbf{spherical noise stability} of a set $\Omega\subset S^{\sdimn-1}$ with parameters $\rho,r,s$ is
$$\int_{S^{\sdimn-1}}1_{\Omega}(x)U_{g}1_{\Omega}(x)\,\d\sigma(x).$$
\end{definition}

When $\sdimn=2$, we have, for any $-1<t<1$, using the substitution $a=\cos\theta$,
\begin{equation}\label{g2def}
g(t)\stackrel{\eqref{one0p}}{=}
e^{\frac{\rho rst}{1-\rho^{2}}}\cdot \frac{\int_{-1}^{1}(1-a^{2})^{-1/2}\d a}{\int_{-1}^{1}(1-a^{2})^{-1/2}e^{\frac{\rho rsa}{1-\rho^{2}}}\d a}
=e^{\frac{\rho rst}{1-\rho^{2}}}\cdot \frac{\int_{0}^{\pi}\d\theta}{\int_{0}^{\pi}e^{\frac{\rho rs\cos\theta}{1-\rho^{2}}}\d\theta}
=\frac{e^{\frac{\rho rst}{1-\rho^{2}}}}{\frac{1}{\pi}\int_{0}^{\pi}e^{\frac{\rho rs\cos\theta}{1-\rho^{2}}}\d\theta}.
\end{equation}
%  a=cos\theta.  da=-sintheta dteta
% from theta=-pi/2 to theta=pi/2
\begin{equation}\label{four4}
\begin{aligned}
\int_{S^{1}}1_{\Omega}(x)U_{g}1_{\Omega}(x)\,\d\sigma(x)
&\stackrel{\eqref{g2def}}{=}\frac{\frac{1}{(2\pi)^{2}}\int_{S^{1}}\int_{S^{1}}1_{\Omega}(x)1_{\Omega}(y)e^{\frac{\rho rs\langle x,y\rangle}{1-\rho^{2}}}\d x\d y}{\frac{1}{2\pi}\int_{0}^{2\pi}e^{\frac{\rho rs\cos\theta}{1-\rho^{2}}}\d\theta}\\
&=\frac{\int_{a=0}^{a=2\pi}\int_{b=0}^{b=2\pi}1_{\Omega}(a)1_{\Omega}(b)e^{\frac{\rho rs\cos(a-b)}{1-\rho^{2}}}\,\d a\d b}{2\pi\int_{0}^{2\pi}e^{\frac{\rho rs\cos a}{1-\rho^{2}}}\d a}.
\end{aligned}
\end{equation}
Here and below we abuse notation slightly by identifying $S^{1}$ with the interval $[0,2\pi]$.

The spherical noise stability has a decomposition into Fourier series by the Funk-Hecke Formula \cite[Theorem 4.3]{hwang21} (see also \eqref{one6} for the definition of $\lambda_{d,2}^{r,s}$)
\begin{equation}\label{four7}
\int_{S^{1}}1_{\Omega}(x)U_{g}1_{\Omega}(x)\,\d\sigma(x)
=[\sigma(\Omega)]^{2}+\frac{1}{2\pi^{2}}\sum_{d=1}^{\infty}\lambda_{d,2}^{r,s}\Big(\Big(\int_{\Omega}\cos(xd)\,\d x\Big)^{2}+\Big(\int_{\Omega}\sin(xd)\,\d x\Big)^{2}\Big).
\end{equation}
More generally, for any measurable $\Omega,\Omega'\subset S^{1}$, we have e.g. by applying a polarization identity to \eqref{four7},
\begin{equation}\label{four7Z}
\begin{aligned}
&\int_{S^{1}}1_{\Omega}(x)U_{g}1_{\Omega'}(x)\,\d\sigma(x)
=\sigma(\Omega)\sigma(\Omega')\\
&\qquad+\frac{1}{2\pi^{2}}\sum_{d=1}^{\infty}\lambda_{d,2}^{r,s}\Big(\int_{\Omega}\cos(xd)\,\d x\int_{\Omega'}\cos(xd)\,\d x+\int_{\Omega}\sin(xd)\,\d x\int_{\Omega'}\sin(xd)\,\d x\Big).
\end{aligned}
\end{equation}
%  L2 normalized function (with respect to haar measure)   is   sqrt(2)cos(xd), since its integral squared is the integral of cos^2 / pi= cos^2/pi, which is 1 (with respect to haar measure)
%  so the coefficient^2 is the integral of (sqrt(2) cos(xd) / 2pi), then squared
%

Below, we will obtain some bounds for the constants $\lambda_{d,2}^{r,s}$ in e.g. \eqref{four9}, \eqref{two3z} and \eqref{two3zP}.  Roughly speaking, $\lambda_{d,2}^{r,s}$ behaves like $(\rho rs)^{d}$ for $r\cdot s$ near zero, and $\lambda_{d,2}^{r,s}$ is close to $1$ for large $r\cdot s$.  We note in passing that, for any $0<t\leq\pi$,
\begin{equation}\label{four8}
\int_{S^{1}}1_{[-t/2,t/2]}(x)U_{g}1_{[-t/2,t/2]}(x)\,\d\sigma(x)
\stackrel{\eqref{four7}}{=}\Big(\frac{t}{2\pi}\Big)^{2}+\frac{2}{\pi^{2}}\sum_{d=1}^{\infty}\lambda_{d,2}^{r,s}\frac{1}{d^{2}}\sin^{2}(td/2).
\end{equation}
% int_{-t/2}^{t/2} cos(xd) = (1/d)[\sin xd]_{-t/2}^{t/2} = (2/d)\sin(t/2)

\textbf{Notation: Rising Factorial}.  For any $x\in\R$ and for any integer $d\geq1$, we denote $(x)_{d}\colonequals \prod_{j=0}^{d-1}(x+j)$.

Let $C_{d}^{(\alpha)}\colon[-1,1]\to\R$ denote the index $\alpha$ degree $d$ Gegenbauer polynomial, which satisfies a Rodrigues formula
\cite[p. 303, 6.4.14]{andrews99}
$$(1-t^{2})^{\alpha-1/2}C_{d}^{(\alpha)}(t)=\frac{(-2)^{d}(\alpha)_{d}}{d!(d+2\alpha)_{d}}\frac{\d^{d}}{\d t^{d}}(1-t^{2})^{\alpha+d-1/2},\qquad\forall\,t\in[-1,1].$$
Letting $\alpha\colonequals\frac{n}{2}-1$, we have
\begin{equation}\label{one1z}
(1-t^{2})^{\frac{n}{2}-\frac{3}{2}}C_{d}^{(\frac{n}{2}-1)}(t)=\frac{(-2)^{d}\Big(\frac{n}{2}-1\Big)_{d}}{d!(d+n-2)_{d}}\frac{\d^{d}}{\d t^{d}}(1-t^{2})^{\frac{n}{2}+d-\frac{3}{2}},\qquad\forall\,t\in[-1,1].
\end{equation}
From \cite[p. 302]{andrews99},
\begin{equation}\label{one2z}
C_{d}^{(\frac{n}{2}-1)}(1)=\frac{(n-2)_{d}}{d!}.
\end{equation}
Then \cite[Corollary 4.6]{hwang21} defines
\begin{flalign*}
\lambda_{d,\sdimn}^{r,s}
&\colonequals\frac{\int_{-1}^{1}\frac{C_{d}^{(\frac{n}{2}-1)}(t)}{C_{d}^{(\frac{n}{2}-1)}(1)}(1-t^{2})^{\frac{n}{2}-\frac{3}{2}} g(t)\d t}{\int_{-1}^{1}(1-t^{2})^{\frac{n}{2}-\frac{3}{2}}\d t}
\stackrel{\eqref{one0p}}{=}\frac{\int_{-1}^{1}\frac{C_{d}^{(\frac{n}{2}-1)}(t)}{C_{d}^{(\frac{n}{2}-1)}(1)}(1-t^{2})^{\frac{n}{2}-\frac{3}{2}} e^{\frac{\rho rst}{1-\rho^{2}}}\d t}
{\int_{-1}^{1}(1-t^{2})^{\frac{n}{2}-\frac{3}{2}}e^{\frac{\rho rst}{1-\rho^{2}}}\d t}\\
&\stackrel{\eqref{one1z}\wedge\eqref{one2z}}{=}
\frac{(-2)^{d}(\frac{n}{2}-1)_{d}}{d!(d+n-2)_{d}}\frac{d!}{(n-2)_{d}}
\frac{\int_{-1}^{1} \Big[\frac{\d^{d}}{\d t^{d}}(1-t^{2})^{\frac{n}{2}+d-\frac{3}{2}}\Big]e^{\frac{\rho rst}{1-\rho^{2}}}\d t}
{\int_{-1}^{1}(1-t^{2})^{\frac{n}{2}-\frac{3}{2}}e^{\frac{\rho rst}{1-\rho^{2}}}\d t}.
\end{flalign*}
Integrating by parts $d$ times,
\begin{equation}\label{one6}
\lambda_{d,\sdimn}^{r,s}
=\Big(\frac{\rho rs}{1-\rho^{2}}\Big)^{d}
\frac{(-2)^{d}(\frac{n}{2}-1)_{d}}{(n-2)_{2d}}
\frac{(-1)^{d}\int_{-1}^{1} (1-t^{2})^{\frac{n}{2}+d-\frac{3}{2}}e^{\frac{\rho rst}{1-\rho^{2}}}\d t}
{\int_{-1}^{1}(1-t^{2})^{\frac{n}{2}-\frac{3}{2}}e^{\frac{\rho rst}{1-\rho^{2}}}\d t}.
\end{equation}

We now examine the ratio of integrals

\begin{equation}\label{two1z}
\frac{\int_{-1}^{1} (1-t^{2})^{\frac{n}{2}+d-\frac{3}{2}}e^{\frac{\rho rst}{1-\rho^{2}}}\d t}
{\int_{-1}^{1}(1-t^{2})^{\frac{n}{2}-\frac{3}{2}}e^{\frac{\rho rst}{1-\rho^{2}}}\d t},
\end{equation}
which is a ratio of modified Bessel functions of the first kind.  To recall their definition, first recall the definition of the Bessel function $J_{\alpha}$ of the first kind of order $\alpha\geq0$ \cite[p. 204]{andrews99}
$$J_{\alpha}(x)\colonequals\frac{1}{\sqrt{\pi}\Gamma(\alpha+1/2)}(x/2)^{\alpha}\int_{-1}^{1}e^{ixt}(1-t^{2})^{\alpha-1/2}\d t,\qquad\forall\,x\in\R.$$
The modified Bessel function $I_{\alpha}$ of the first kind of order $\alpha\geq0$ is then defined to be \cite[p. 222]{andrews99}
\begin{flalign*}
I_{\alpha}(x)
\colonequals i^{-\alpha}J_{\alpha}(ix)
&=\frac{1}{\sqrt{\pi}\Gamma(\alpha+1/2)}(x/2)^{\alpha}\int_{-1}^{1}e^{-xt}(1-t^{2})^{\alpha-1/2}\d t\\
&=\frac{1}{\sqrt{\pi}\Gamma(\alpha+1/2)}(x/2)^{\alpha}\int_{-1}^{1}e^{xt}(1-t^{2})^{\alpha-1/2}\d t,\qquad\forall\,x\in\R.
\end{flalign*}
$$
\frac{I_{\alpha+d}(a)}{I_{\alpha}(a)}
=\frac{(a/2)^{d}}{(\alpha+1/2)_{d}}\frac{\int_{-1}^{1}e^{at}(1-t^{2})^{\alpha+d-1/2}\d t}{\int_{-1}^{1}e^{at}(1-t^{2})^{\alpha-1/2}\d t},\qquad\forall\,a\in\R,\,\forall\,\alpha\geq0.
$$

So, setting $\alpha=(n/2)-1$ here, the ratio from \eqref{two1z} is equal to %\alpha = n/2
$$\Big(\frac{n-1}{2}\Big)_{d}\Big(\frac{2(1-\rho^{2})}{\rho rs}\Big)^{d}\cdot\frac{I_{n/2+d-1}(\rho rs/[1-\rho^{2}])}{I_{(n/2)-1}(\rho rs/[1-\rho^{2}])}.$$
Combining with \eqref{one6}, we have
\begin{equation}\label{four9}
\begin{aligned}
\lambda_{d,\sdimn}^{r,s}
&=\Big(\frac{\rho rs}{1-\rho^{2}}\Big)^{d}\frac{2^{d}(n/2 -1)_{d}}{(n-2)_{2d}}\Big(\frac{n-1}{2}\Big)_{d}\Big(\frac{2(1-\rho^{2})}{\rho rs}\Big)^{d}\frac{I_{n/2 +d-1}(\rho rs/[1-\rho^{2}])}{I_{(n/2) -1}(\rho rs/[1-\rho^{2}])}\\
&=\frac{I_{n/2 +d-1}(\rho rs/[1-\rho^{2}])}{I_{(n/2) -1}(\rho rs/[1-\rho^{2}])}
=\prod_{j=1}^{d}\frac{I_{(n/2) +j-1}(\rho rs/[1-\rho^{2}])}{I_{(n/2) +j-2}(\rho rs/[1-\rho^{2}])}.
\end{aligned}
\end{equation}

We have \cite[p. 241]{amos74}
\begin{equation}\label{two2z}
\frac{a}{\alpha+1+\sqrt{(\alpha+1)^{2}+a^{2}}}\leq\frac{I_{\alpha+1}(a)}{I_{\alpha}(a)}\leq\frac{a}{\alpha+\sqrt{\alpha^{2}+a^{2}}},\qquad\forall\,a,\alpha\geq0.
\end{equation}
It therefore follows from \eqref{four9} that
\begin{equation}\label{four10}
\lambda_{d,n}^{r,s}\geq\lambda_{d+1,n}^{r,s},\qquad\forall\,d\geq0,
\end{equation}
and
%use alpha=(n/2)-1
\begin{equation}\label{two3z}
\begin{aligned}
0&<\prod_{j=0}^{d-1}\frac{\rho rs/[1-\rho^{2}]}{n/2+j+\sqrt{(n/2+j)^{2}+[\rho rs/[1-\rho^{2}]]^{2}}}
\leq\lambda_{d,\sdimn}^{r,s}\\
&\quad\leq
\prod_{j=0}^{d-1}\frac{\rho rs/[1-\rho^{2}]}{(n/2+j-1)+\sqrt{(n/2+j-1)^{2}+[\rho rs/[1-\rho^{2}]]^{2}}}\leq1,\quad\forall\,r,s>0,\,\forall\,\rho\in(0,1).
\end{aligned}
\end{equation}

Also, from \cite[p. 242]{amos74}
\begin{equation}\label{two2zP}
\frac{a}{\alpha+1/2+\sqrt{(\alpha+3/2)^{2}+a^{2}}}\leq\frac{I_{\alpha+1}(a)}{I_{\alpha}(a)}\leq\frac{a}{\alpha+1/2+\sqrt{(\alpha+1/2)^{2}+a^{2}}},\qquad\forall\,a,\alpha\geq0.
\end{equation}
So, for all $r,s>0$ and for all $\rho\in(0,1)$,
\begin{equation}\label{two3zP}
\begin{aligned}
0&<\prod_{j=0}^{d-1}\frac{\rho rs/[1-\rho^{2}]}{n/2+j-1/2+\sqrt{(n/2+j+1/2)^{2}+[\rho rs/[1-\rho^{2}]]^{2}}}
\leq\lambda_{d,\sdimn}^{r,s}\\
&\quad\leq
\prod_{j=0}^{d-1}\frac{\rho rs/[1-\rho^{2}]}{(n/2+j-1/2)+\sqrt{(n/2+j-1/2)^{2}+[\rho rs/[1-\rho^{2}]]^{2}}}\leq1.
\end{aligned}
\end{equation}

Also, from \cite[p. 241]{amos74}
\begin{equation}\label{two4zP}
\frac{a}{1+\alpha+\sqrt{(\alpha+1)^{2}+a^{2}}}\leq\frac{I_{\alpha+1}(a)}{I_{\alpha}(a)}\leq\frac{a}{\alpha+\sqrt{(\alpha+2)^{2}+a^{2}}},\qquad\forall\,a,\alpha\geq0.
\end{equation}

\section{Derivative Estimates}\label{secposder}

As demonstrated in \eqref{four4}, the quantity $F(\theta)$ defined below is the mean subtracted spherical noise stability of an interval $[0,\theta]\subset\R/2\pi\Z$.  In this section we consider derivative properties of the mean subtracted spherical noise stability of a partition of the circle $S^{1}$, parameterized as $\R/2\pi\Z$.  The ultimate goal is control of the second term in \eqref{in1}.

\begin{lemma}\label{lemma5}
Denote
\begin{equation}\label{fdef}
F(\theta)\colonequals\frac{\int_{a=0}^{a=\theta}\int_{b=0}^{b=\theta}e^{\frac{\rho rs\cos(a-b)}{1-\rho^{2}}}\,\d a\d b}{2\pi\int_{0}^{2\pi}e^{\frac{\rho rs\cos a}{1-\rho^{2}}}\d a}-\Big(\frac{\theta}{2\pi}\Big)^{2}.
\end{equation}
Suppose $\theta_{1}>\theta_{2}>0$ and $\pi\leq\theta_{1}+\theta_{2}\leq2\pi$.  Then
$$
\frac{\d}{\d\theta}|_{\theta=0}\Big[F(\theta_{1}+\theta)+F(\theta_{2}-\theta)\Big]
<0.
$$
If additionally $\theta_{1}-\theta_{2}\leq\pi$, then
\begin{flalign*}
&\frac{\d}{\d\theta}|_{\theta=0}\Big[F(\theta_{1}+\theta)+F(\theta_{2}-\theta)\Big]\leq
\Big(-1+e^{-\frac{\rho rs\cos([\theta_{1}-\theta_{2}]/2)[\theta_{1}+\theta_{2}-\pi]/2}{1-\rho^{2}}}\Big)\frac{\theta_{1}-\theta_{2}}{2\pi^{2}}.
\end{flalign*}%
\end{lemma}

\begin{proof}
We have
\begin{equation}\label{four6}
\begin{aligned}
&\frac{\d}{\d\theta}|_{\theta=0}[F(\theta_{1}+\theta)+F(\theta_{2}-\theta)]
=F'(\theta_{1})-F'(\theta_{2})\\
&\qquad=2\frac{\int_{b=0}^{b=\theta_{1}}e^{\frac{\rho rs\cos(\theta_{1}-b)}{1-\rho^{2}}}\,\d b}{2\pi\int_{0}^{2\pi}e^{\frac{\rho rs\cos a}{1-\rho^{2}}}\d a}-\frac{\theta_{1}}{2\pi^{2}}
-2\frac{\int_{b=0}^{b=\theta_{2}}e^{\frac{\rho rs\cos(\theta_{2}-b)}{1-\rho^{2}}}\,\d b}{2\pi\int_{0}^{2\pi}e^{\frac{\rho rs\cos a}{1-\rho^{2}}}\d a}+\frac{\theta_{2}}{2\pi^{2}}\\
&\qquad=2\frac{\int_{b=0}^{b=\theta_{1}}e^{\frac{\rho rs\cos(b)}{1-\rho^{2}}}\,\d b}{2\pi\int_{0}^{2\pi}e^{\frac{\rho rs\cos a}{1-\rho^{2}}}\d a}-\frac{\theta_{1}}{2\pi^{2}}
-2\frac{\int_{b=0}^{b=\theta_{2}}e^{\frac{\rho rs\cos(b)}{1-\rho^{2}}}\,\d b}{2\pi\int_{0}^{2\pi}e^{\frac{\rho rs\cos a}{1-\rho^{2}}}\d a}+\frac{\theta_{2}}{2\pi^{2}}\\
&\qquad=2\frac{\int_{b=\theta_{2}}^{b=\theta_{1}}e^{\frac{\rho rs\cos(b)}{1-\rho^{2}}}\,\d b}{2\pi\int_{0}^{2\pi}e^{\frac{\rho rs\cos a}{1-\rho^{2}}}\d a}-\frac{\theta_{1}-\theta_{2}}{2\pi^{2}}.
\end{aligned}
\end{equation}
In the case $\theta_{1}-\theta_{2}\geq\pi$, \eqref{four6} is negative by the monotonicity of the cosine function:
\begin{equation}\label{two0}
2\frac{\int_{b=\theta_{2}}^{b=\theta_{1}}e^{\frac{\rho rs\cos(b)}{1-\rho^{2}}}\,\d b}{2\pi\int_{0}^{2\pi}e^{\frac{\rho rs\cos a}{1-\rho^{2}}}\d a}
\leq2\frac{\int_{b=0}^{b=\theta_{1}-\theta_{2}}e^{\frac{\rho rs\cos(b)}{1-\rho^{2}}}\,\d b}{2\pi\int_{0}^{2\pi}e^{\frac{\rho rs\cos a}{1-\rho^{2}}}\d a}
<\frac{\theta_{1}-\theta_{2}}{2\pi^{2}}.
\end{equation}%%
So, it remains to consider the case that $\theta_{1}-\theta_{2}<\pi$, i.e. that $[\theta_{1}-\theta_{2}]/2<\pi/2$.

Using $\cos(x-y)-\cos(x)=\int_{x-y}^{x}\sin(y)\,\d y\geq y\min(\sin(x),\sin(x-y))\geq y\min_{x-y\leq z\leq x}\sin(z)$, for all $\theta_{2}\leq b\leq\theta_{1}$, we have
\begin{flalign*}
e^{\frac{\rho rs\cos(b-[\theta_{1}+\theta_{2}-\pi]/2)}{1-\rho^{2}}}
-e^{\frac{\rho rs\cos(b)}{1-\rho^{2}}}
&=e^{\frac{\rho rs\cos(b)}{1-\rho^{2}}}\Big[e^{\frac{\rho rs(\cos(b-[\theta_{1}+\theta_{2}-\pi]/2)-\cos(b))}{1-\rho^{2}}}-1\Big]\\
&=e^{\frac{\rho rs\cos(b)}{1-\rho^{2}}}\Big[e^{\frac{\rho rs\int_{b-[\theta_{1}+\theta_{2}-\pi]/2}^{b}\sin(y)\,\d y}{1-\rho^{2}}}-1\Big]\\
&\geq e^{\frac{\rho rs\cos(b)}{1-\rho^{2}}}\Big[e^{\frac{\rho rs\cos([\theta_{1}-\theta_{2}]/2)\cdot[\theta_{1}+\theta_{2}-\pi]/2}{1-\rho^{2}}}-1\Big].
\end{flalign*}
That is,
$$
e^{\frac{\rho rs\cos(b)}{1-\rho^{2}}}
\leq e^{\frac{\rho rs\cos(b-[\theta_{1}+\theta_{2}-\pi]/2)}{1-\rho^{2}}}e^{-\frac{\rho rs\cos([\theta_{1}-\theta_{2}]/2)\cdot[\theta_{1}+\theta_{2}-\pi]/2}{1-\rho^{2}}},\qquad\forall\,\theta_{2}\leq b\leq\theta_{1}.
$$
Therefore, denoting $c\colonequals e^{-\frac{\rho rs\cos([\theta_{1}-\theta_{2}]/2)\cdot[\theta_{1}+\theta_{2}-\pi]/2}{1-\rho^{2}}}$,
\begin{equation}\label{four11}
\begin{aligned}
&\int_{b=\theta_{2}}^{b=\theta_{1}}e^{\frac{\rho rs\cos(b)}{1-\rho^{2}}}\,\d b
\leq c
\int_{b=\theta_{2}}^{b=\theta_{1}}e^{\frac{\rho rs\cos(b-[\theta_{1}+\theta_{2}-\pi]/2)}{1-\rho^{2}}}\,\d b
=c \int_{b=\theta_{2}-\frac{\theta_{1}+\theta_{2}-\pi}{2}}^{b=\theta_{1}-\frac{\theta_{1}+\theta_{2}-\pi}{2}}e^{\frac{\rho rs\cos(b)}{1-\rho^{2}}}\,\d b\\
&\qquad=c \int_{b=\frac{\theta_{2}-\theta_{1}}{2}+\frac{\pi}{2}}^{b=\frac{\theta_{1}-\theta_{2}}{2}+\frac{\pi}{2}}e^{\frac{\rho rs\cos(b)}{1-\rho^{2}}}\,\d b
=c \int_{b=\frac{\theta_{2}-\theta_{1}}{2}}^{b=\frac{\theta_{1}-\theta_{2}}{2}}e^{-\frac{\rho rs\sin(b)}{1-\rho^{2}}}\,\d b
=c\int_{b=\frac{\theta_{2}-\theta_{1}}{2}}^{b=\frac{\theta_{1}-\theta_{2}}{2}}e^{\frac{\rho rs\sin(b)}{1-\rho^{2}}}\,\d b.
\end{aligned}
\end{equation}
%  y=b-pi/2.    cos(b)= cos(y+pi/2)= sin(y) ?

The Lemma then follows from Lemma \ref{lemma6} below, since Lemma \ref{lemma6} implies that
$$\frac{2\int_{b=\frac{\theta_{2}-\theta_{1}}{2}}^{b=\frac{\theta_{1}-\theta_{2}}{2}}e^{\frac{\rho rs\sin(b)}{1-\rho^{2}}}\,\d b}{\int_{0}^{2\pi}e^{\frac{\rho rs\cos a}{1-\rho^{2}}}\d a}
<\frac{\theta_{1}-\theta_{2}}{\pi},$$
which concludes the first part of the proof by plugging into \eqref{four6}.

\end{proof}

\begin{lemma}\label{lemma6}
$$
\frac{2\int_{b=-t}^{b=t}e^{\frac{\rho rs\sin(b)}{1-\rho^{2}}}\,\d b}{\int_{0}^{2\pi}e^{\frac{\rho rs\cos a}{1-\rho^{2}}}\d a}<\frac{2t}{\pi}<0,\qquad\forall\,0<t<\pi/2.
$$
\end{lemma}

\begin{proof}
We must show that
$$
\frac{1}{2t}\int_{b=-t}^{b=t}e^{\frac{\rho rs\sin(b)}{1-\rho^{2}}}\,\d b<\frac{1}{2\pi}\int_{0}^{2\pi}e^{\frac{\rho rs\cos a}{1-\rho^{2}}}\d a
=\frac{1}{\pi}\int_{0}^{\pi}e^{\frac{\rho rs\cos a}{1-\rho^{2}}}\d a,\qquad\forall\,0<t<\pi/2.
$$
The quantity on the left is monotone increasing in $t$, and it is equal to the right when $t=\pi/2$.  The former statement follows since
$$
\frac{1}{2t}\int_{b=-t}^{b=t}e^{\frac{\rho rs\sin(b)}{1-\rho^{2}}}\,\d b
=\frac{1}{4t}\int_{b=-t}^{b=t}\Big(e^{\frac{\rho rs\sin(b)}{1-\rho^{2}}}+e^{-\frac{\rho rs\sin(b)}{1-\rho^{2}}}\Big)\,\d b,
$$
and the function $h(b)\colonequals e^{\frac{\rho rs\sin(b)}{1-\rho^{2}}}+e^{-\frac{\rho rs\sin(b)}{1-\rho^{2}}}$ is even and increasing for all $0<b<\pi/2$, so its average value on the interval $[-t,t]$ is increasing for all $0<t<\pi/2$, as can be seen by taking a derivative:
\begin{flalign*}
\frac{\d}{\d t}\Big(\frac{1}{t}\int_{b=-t}^{b=t}h(b)\,\d b\Big)
&=\frac{1}{t}2h(t)-\frac{1}{t^{2}}\int_{b=-t}^{b=t}h(b)\,\d b\\
&=\frac{1}{t}\Big(\frac{1}{t}\int_{b=-t}^{b=t}h(t)\d b-\frac{1}{t}\int_{b=-t}^{b=t}h(b)\,\d b\Big)
=\frac{1}{t^{2}}\int_{b=-t}^{b=t}(h(t)-h(b))\d b.
\end{flalign*}
\end{proof}

Integrating Lemma \ref{lemma5} gives the following stability estimate, showing that three intervals each of length $2\pi/3$ maximize the second term of \eqref{in1}, at least among partitions into three intervals.  (The case that $\theta_{i}\geq\pi$ for some $1\leq i\leq 3$ in Lemma \ref{lemma7} will be treated separately in Lemma \ref{lemma10} below.)

\begin{lemma}\label{lemma7}
Let $\theta_{1},\theta_{2},\theta_{3}\geq0$ with $\theta_{1}+\theta_{2}+\theta_{3}=2\pi$.  Define $F$ as in \eqref{fdef}.  Assume that $\theta_{i}\leq\pi$ for all $1\leq i\leq 3$.  Then
\begin{flalign*}
-3F(2\pi/3)+\sum_{i=1}^{3}F(\theta_{i})
&\leq
(.158)\Big(-1+\Big(1-\frac{3^{4/3}}{5}\Big) e^{-\frac{\rho r s}{1-\rho^{2}}\frac{\pi}{2}}\cdots \\
&\qquad+\frac{3^{4/3}}{5}e^{-\frac{\rho r s}{1-\rho^{2}}\frac{\pi}{6}}\Big)\sum_{1\leq i<j\leq 3}\Big(\frac{\theta_{i}}{2\pi}-1/3\Big)^{2}.
\end{flalign*}
\end{lemma}

\begin{proof}

Since $\theta_{i}\leq\pi$ for all $1\leq i\leq 3$ and $\theta_{1}+\theta_{2}+\theta_{3}=2\pi$, we have  $\abs{\theta_{i}-\theta_{j}}\leq\pi$ and $\theta_{i}+\theta_{j}\geq\pi$ for all $1\leq i<j\leq 3$.  We will use the latter inequalities frequently below.

We will replace $\theta_{1}\geq\theta_{2}$ with $\theta_{1}-t,\theta_{2}+t$, where $0\leq t\leq (\theta_{1}-\theta_{2})/2$.  We change the first two regions to have the same measure.  From Lemma \ref{lemma5}, after integrating in $t$, we will then get a factor of
\begin{flalign*}
&\int_{t=0}^{t=[\theta_{1}-\theta_{2}]/2}\Big(-1+e^{-\frac{\rho rs\cos((\theta_{1}-\theta_{2}-2t)/2)[\theta_{1}+\theta_{2}-\pi]/2}{1-\rho^{2}}}\Big)\frac{\theta_{1}-\theta_{2}-2t}{2\pi^{2}}\,\d t\\
&=\frac{1}{2}\int_{t=0}^{t=\theta_{1}-\theta_{2}}\Big(-1+e^{-\frac{\rho rs\cos(t/2)[\theta_{1}+\theta_{2}-\pi]/2}{1-\rho^{2}}}\Big)\frac{t}{2\pi^{2}}\,\d t.
\end{flalign*}
%  s=theta1 - theta2 - 2t.  ds=-2dt.  dt=-1/2 ds.

Now, the region has angles, $\{[\theta_{1}+\theta_{2}]/2,[\theta_{1}+\theta_{2}]/2,\theta_{3}\}$.  We will replace these angles with $\{[\theta_{1}+\theta_{2}]/2 -t,[\theta_{1}+\theta_{2}]/2 -t,\theta_{3}+2t\}$, where $0\leq t\leq [(\theta_{1}+\theta_{2})/2-\theta_{3}]/3$.  Then, change all three regions to have the same measure.  After integrating in $t$, we will get an additional factor from Lemma \ref{lemma5} of
%  [t1 +t2]/2 - t3  = [t1 -t3]/2 + [t2 - t3]/2
%  ans^2 = [[t1-t3]/2]^2 + [[t2-t3]/2]^2 + cross term
\begin{flalign*}
&2\int_{t=0}^{t=[(\theta_{1}+\theta_{2})/2 - \theta_{3}]/3}\Big(-1+e^{-\frac{\rho rs\cos\frac{1}{2}((\theta_{1}+\theta_{2})/2 - \theta_{3} - 3t)[\theta_{1}+\theta_{2}-t+2t+2\theta_{3} - 2\pi]/4]}{1-\rho^{2}}}\Big)\frac{(\theta_{1}+\theta_{2})/2 - \theta_{3} - 3t)}{2\cdot2\pi^{2}}\d t\\
&\qquad\qquad\qquad\qquad=\frac{2}{3}\int_{t=0}^{t=[\theta_{1}+\theta_{2}]/2 - \theta_{3}}\Big(-1+e^{-\frac{\rho rs\cos(t/2)[\theta_{3}+t]/4]}{1-\rho^{2}}}\Big)\frac{t}{2\pi^{2}}\d t.
\end{flalign*}
% s= theta1 + theta2 /2 - theta3  -3t
%  ds= -3dt
So, the total you get is, using $-1+e^{-\epsilon c}\leq (-1+e^{-\epsilon})c$, when $c=\cos(t/2)$
%Here we used $\int_{0}^{c}te^{bt}\,\d t=b^{-2}[1+e^{bc}(bc-1)]$
\begin{flalign*}
&\frac{1}{2}\int_{t=0}^{t=\theta_{1}-\theta_{2}}\Big(-1+e^{-\frac{\rho rs\cos(t/2)[\theta_{1}+\theta_{2}-\pi]/2}{1-\rho^{2}}}\Big)\frac{t}{2\pi^{2}}\d t\\
&\qquad\qquad\qquad+\frac{2}{3}\int_{t=0}^{t=[\theta_{1}+\theta_{2}]/2 - \theta_{3}}\Big(-1+e^{-\frac{\rho rs\cos(t/2)\theta_{3}/4}{1-\rho^{2}}}\Big)\frac{t}{2\pi^{2}}\d t\\
&\leq\frac{1}{2}\Big(-1+e^{-\frac{\rho rs[\theta_{1}+\theta_{2}-\pi]/2}{1-\rho^{2}}}\Big)\int_{t=0}^{t=\theta_{1}-\theta_{2}}\cos(t/2)\frac{t}{2\pi^{2}}\d t\\
&\qquad\qquad\qquad+\frac{2}{3}\Big(-1+e^{-\frac{\rho rs\theta_{3}/4}{1-\rho^{2}}}\Big)\int_{t=0}^{t=[\theta_{1}+\theta_{2}]/2 - \theta_{3}}\cos(t/2)\frac{t}{2\pi^{2}}\d t.
\end{flalign*}
Integrating and using $\int_{0}^{c}t\cos(t/2)\d t=2c\sin(c/2)+4(\cos(c/2)-1)$ and $\int_{0}^{c}\cos(t/2)\d t=2\sin(c/2)$.  We have $c^{2}/2 - c^{4}/34\geq\int_{0}^{c}t\cos(t/2)\d t\geq c^{2}/2 - c^{4}/32$ for all $0\leq c\leq \pi/2$.
\begin{flalign*}
&\frac{1}{2}\Big(-1+e^{-\frac{\rho rs[\theta_{1}+\theta_{2}-\pi]/2}{1-\rho^{2}}}\Big)\Big[\Big(\frac{\theta_{1}-\theta_{2}}{2\pi}\Big)^{2}\Big(\frac{1}{2}-\frac{1}{32}(\theta_{1}-\theta_{2})^{2}\Big)\\
&\quad+\frac{2}{3}\Big(-1+e^{-\frac{\rho rs\theta_{3}/4}{1-\rho^{2}}}\Big)\Big(\frac{[\theta_{1}-\theta_{3}]/2 + (\theta_{2}-\theta_{3})/2}{2\pi}\Big)^{2}\Big(\frac{1}{2}-\frac{1}{32}\Big([\theta_{1}-\theta_{3}]/2 + (\theta_{2}-\theta_{3})/2\Big)^{2}\Big)\\
&=\frac{1}{2}\Big(-1+e^{-\frac{\rho rs[\pi-\theta_{3}]/2}{1-\rho^{2}}}\Big)\Big[\Big(\frac{\theta_{1}-\theta_{2}}{2\pi}\Big)^{2}\Big(\frac{1}{2}-\frac{1}{32}(\theta_{1}-\theta_{2})^{2}\Big)\\
&\quad+\frac{2}{3}\Big(-1+e^{-\frac{\rho rs\theta_{3}/4}{1-\rho^{2}}}\Big)\Big(\frac{[\theta_{1}-\theta_{3}]/2 + (\theta_{2}-\theta_{3})/2}{2\pi}\Big)^{2}\Big(\frac{1}{2}-\frac{1}{32}\Big([\theta_{1}-\theta_{3}]/2 + (\theta_{2}-\theta_{3})/2\Big)^{2}\Big).
\end{flalign*}

Averaging over all permutations of $\theta_{1},\theta_{2},\theta_{3}$, we arrive at the function $(1/6)(e^{-\alpha[\pi-\theta_{3}]/2}+(2/3)e^{-\alpha \theta_{3}/4})$, which is minimized over all $0\leq\theta_{3}\leq \pi$ at the value $(1/6)3^{1/3}e^{-\alpha\pi/6}$ when $2\log 3/\pi\leq \alpha$, or $\theta_{3}=0$ gets the minimum value of $(1/6)((2/3)+e^{-\pi\alpha/2})$ when $\alpha<2\log 3/\pi$.  In summary,
\begin{flalign*}
-3F(2\pi/3)+\sum_{i=1}^{3}F(\theta_{i})
&\leq
\sum_{1\leq i<j\leq 3}\Big(-5/18+(1/6)((2/3)+e^{-\frac{\rho r s}{1-\rho^{2}}\frac{\pi}{2}})1_{\left\{\frac{\rho r s}{1-\rho^{2}}\leq \frac{2\log 3}{\pi}\right\}} \\
&\qquad+ (1/6)3^{1/3}e^{-\frac{\rho rs}{1-\rho^{2}}\frac{\pi}{6}}1_{\left\{\frac{\rho r s}{1-\rho^{2}}> \frac{2\log 3}{\pi}\right\}} \Big)\cdot(.19)\Big(\frac{\theta_{i}-\theta_{j}}{2\pi}\Big)^{2}.
\end{flalign*}
%Then $\theta_{i}-\theta_{j}\leq2\pi$, so $[\theta_{i}-\theta_{j}]/8\leq\pi/4$, so $\cos([\theta_{i}-\theta_{j}]/8)\geq.7$.
Finally, we use the elementary inequality
$$\sum_{1\leq i<j\leq 3}\Big(\frac{\theta_{i}-\theta_{j}}{2\pi}\Big)^{2}\geq 3\sum_{i=1}^{3}\Big(\frac{\theta_{i}}{2\pi}-1/3\Big)^{2},$$
(actually an equality) valid for all $\theta_{1},\theta_{2},\theta_{3}\geq0$ with $\sum_{i=1}^{3}\theta_{i}=2\pi$, to get
\begin{flalign*}
-3F(2\pi/3)+\sum_{i=1}^{3}F(\theta_{i})
&\leq
\sum_{1\leq i<j\leq 3}\Big(-5/18+(1/6)((2/3)+e^{-\frac{\rho r s}{1-\rho^{2}}\frac{\pi}{2}})1_{\left\{\frac{\rho r s}{1-\rho^{2}}\leq \frac{2\log 3}{\pi}\right\}} \\
&\qquad+ (1/6)3^{1/3}e^{-\frac{\rho rs}{1-\rho^{2}}\frac{\pi}{6}}1_{\left\{\frac{\rho r s}{1-\rho^{2}}> \frac{2\log 3}{\pi}\right\}} \Big)\cdot(.57)\sum_{i=1}^{3}\Big(\frac{\theta_{i}}{2\pi}-1/3\Big)^{2}.
\end{flalign*}
A slightly larger upper bound is
\begin{flalign*}
-3F(2\pi/3)+\sum_{i=1}^{3}F(\theta_{i})
&\leq
\frac{5}{18}\cdot(.57)\Big(-1+\Big(1-\frac{3^{4/3}}{5}\Big) e^{-\frac{\rho r s}{1-\rho^{2}}\frac{\pi}{2}} \\
&\qquad+\frac{3^{4/3}}{5}e^{-\frac{\rho r s}{1-\rho^{2}}\frac{\pi}{6}}\Big)\sum_{i=1}^{3}\Big(\frac{\theta_{i}}{2\pi}-1/3\Big)^{2}.
\end{flalign*}
%rho=.03;
%r=linspace(0,100,1000);
%plot( r,  .57*(-5/18   + (1/6)*((2/3) + exp(-rho* r.^2 *pi / (2*(1-rho^2)))).*( (rho* r.^2 / (1-rho^2)) < 2*log(3) / pi)  + (1/6)*(3^(1/3)) * exp(-rho* r.^2 *pi / (6*(1-rho^2))) .*( (rho* r.^2 / (1-rho^2)) > 2*log(3) / pi))    ,   r, (5/18) * .57 * (-1  +  (1- (3^(4/3) /5)) * exp(-rho* r.^2 *pi / (2*(1-rho^2)))  + (3^(4/3) /5)*exp(-rho* r.^2 * pi / (6*(1-rho^2))))) ;
%%axis([0 100 0 .14])
%legend('original','upper bound')
The proof is completed.
%
% boundary case:  x, 1-x, 0.  get (2x-1)^2 + x^2 + (1-x)^2  \geq?  (x-1/3)^2 + (x-2/3)^2
\end{proof}

\section{Stability Estimates}

In Section \ref{secposder}, we gave derivative estimates for the mean subtracted spherical noise stability \eqref{fdef} for intervals in $\R/2\pi\Z$.  However, these derivative estimates do not apply for intervals of length larger than $\pi$.  In this section, using more direct computations, we consider the mean subtracted spherical noise stability of partitions of $\R/2\pi\Z$ into three disjoint intervals, where one of the intervals has length larger than $\pi$.  Getting good constants in Lemma \ref{lemma10} below requires considering several different cases in a rather tedious manner.  We will eventually circumvent Lemma \ref{lemma10} with a numerical computation, but we still prove Lemma \ref{lemma10} to show that a rigorous (non-numerical) bound can be achieved, with a constant that is worse by a multiplicative factor of about $1/2$.

\begin{lemma}\label{lemma10}
Let $\theta_{1},\theta_{2},\theta_{3}\geq0$ with $\theta_{1}+\theta_{2}+\theta_{3}=2\pi$.  Denote $$F(\theta)\colonequals\frac{\int_{a=0}^{a=\theta}\int_{b=0}^{b=\theta}e^{\frac{\rho rs\cos(a-b)}{1-\rho^{2}}}\,\d a\d b}{2\pi\int_{0}^{2\pi}e^{\frac{\rho rs\cos a}{1-\rho^{2}}}\d a}-\Big(\frac{\theta}{2\pi}\Big)^{2}.$$
If $\theta_{1}\geq\pi$, then
$$
-3F(2\pi/3)+\sum_{i=1}^{3}F(\theta_{i})
\leq
-\frac{1}{\pi^{2}}\frac{13}{9}\lambda_{1,n}^{r,s}\sum_{i=1}^{3}\Big(\frac{\theta_{i}}{2\pi}-1/3\Big)^{2}.
$$
\end{lemma}

\begin{proof}
The function $F$ satisfies $F'(0)=F'(\pi)=0$, $F''(\theta)>0$ for $\theta>0$ near zero, and then $F''(\theta)<0$ for all larger $\theta<\pi$.  It follows that $F$ is strictly increasing for all $\theta\in[0,\pi]$.  Also, since $F''$ is monotone decreasing, the function $\theta\mapsto F(\theta_{1}-\theta)+F(\theta)$ is maximized at either the midpoint or endpoint of $\theta\in[0,\theta_{1}]$.  (Either this function has positive second derivative everywhere, or it is positive except for an interval containing the midpoint, where it is negative, and we know the midpoint is a critical point itself.)

\textbf{Case 1}.  If $5\pi/3\geq\theta_{1}\geq 4\pi/3$, then from the monotonicity properties of $F$, $\sum_{i=1}^{3}F(\theta_{i})$ is largest when $\theta_{1}=4\pi/3$.  To find the values of $\theta_{2},\theta_{3}$ that make $\sum_{i=1}^{3}F(\theta_{i})$ largest, we use Lemma \ref{lemma5}, (and $F(a)=F(2\pi-a)$ $\forall$ $0\leq a\leq 2\pi$, which follows by \eqref{four7}), to get
$$\sum_{i=1}^{3}F(\theta_{i})\leq F(4\pi/3)+\max(F(2\pi/3),2F(\pi/3))=F(2\pi/3)+\max(F(2\pi/3),2F(\pi/3)).$$  We then have
\begin{flalign*}
-3F(2\pi/3)+\sum_{i=1}^{3}F(\theta_{i})
&\leq \max(-F(2\pi/3),2(F(\pi/3)-F(2\pi/3)))\\
&=-\min(F(2\pi/3),2(F(\pi/3)-F(2\pi/3))).
\end{flalign*}
Using \eqref{four8}, we can estimate
\begin{equation}\label{three4}
F(2\pi/3)\geq\frac{3}{2\pi^{2}}\Big(\lambda_{1,n}^{r,s}+\frac{1}{4}\lambda_{2,n}^{r,s}+\frac{1}{16}\lambda_{4,n}^{r,s}+\frac{1}{25}\lambda_{5,n}^{r,s}\Big).
\end{equation}
\begin{equation}\label{three5}
2[F(2\pi/3)-F(\pi/3)]\geq\frac{2}{\pi^{2}}\Big(\frac{7}{9}\lambda_{1,n}^{r,s}+\frac{1}{25}\lambda_{5,n}^{r,s}\Big).
\end{equation}
Then
$$
\min(F(2\pi/3),2(F(\pi/3)-F(2\pi/3)))
\geq\frac{3}{2\pi^{2}}\Big(\lambda_{1,n}^{r,s}+\frac{1}{25}\lambda_{5,n}^{r,s}\Big)
$$
So,  %  2/9 + 5/18 = 9/18 = 1/2
\begin{equation}\label{three2}
-3F(2\pi/3)+\sum_{i=1}^{3}F(\theta_{i})\leq -\frac{3}{2\pi^{2}}\Big(\lambda_{1,n}^{r,s}+\frac{1}{25}\lambda_{5,n}^{r,s}\Big).
\end{equation}
%Meanwhile, an explicit calculation of $3f(2\pi/3)$ (as in \eqref{three1} below) gives
%$$3f(2\pi/3)\geq(9/[2\pi^{2}])(\lambda_{1,n}^{r,s}+\lambda_{2,n}^{r,s}/4).$$
Also, by assumption on $\theta_{1}$,
$$\sum_{i=1}^{3}\Big(\frac{\theta_{i}}{2\pi}-1/3\Big)^{2}\leq (1/2)^{2}+(1/6)^{2} +(1/3)^{2}=7/18.$$
In summary, combining with \eqref{three2},
\begin{equation}\label{three3}
-3F(2\pi/3)+\sum_{i=1}^{3}F(\theta_{i})
\leq-\frac{3}{2\pi^{2}}\Big(\lambda_{1,n}^{r,s}+\frac{1}{25}\lambda_{5,n}^{r,s}\Big)
\leq-\frac{18}{7}\frac{3}{2\pi^{2}}\Big(\lambda_{1,n}^{r,s}+\frac{1}{25}\lambda_{5,n}^{r,s}\Big)\sum_{i=1}^{3}\Big(\frac{\theta_{i}}{2\pi}-1/3\Big)^{2}.
\end{equation}

\textbf{Case 2}.  If $\theta_{1}\geq5\pi/3$, then by Lemma \ref{lemma5}, we have (using $F(a)=F(2\pi-a)$ for all $0\leq a\leq 2\pi$, which follows by \eqref{four7}),
$$\sum_{i=1}^{3}F(\theta_{i})\leq F(5\pi/3)+\max(F(\pi/3),2F(\pi/6))=F(\pi/3)+\max(F(\pi/3),2F(\pi/6)).$$  Using \eqref{four8}, we then have the estimate
\begin{equation}\label{three7}
\begin{aligned}
&-3F(2\pi/3)+\sum_{i=1}^{3}F(\theta_{i})
\leq -3F(2\pi/3)+F(\pi/3)+\max(F(\pi/3),2F(\pi/6))\\
&\quad\leq \max\Big(-F(2\pi/3)-2[F(2\pi/3)-F(\pi/3)] \, ,\\
&\qquad\qquad\qquad\qquad\qquad\qquad\qquad -F(2\pi/3)+F(\pi/3)-2[F(2\pi/3) - F(\pi/6)]\Big).
\end{aligned}
\end{equation}
Adding \eqref{three4} and \eqref{three5},
%%% for 2pi/3, get
%  3/2 3/2 0 3/2 3/2 0 ...
%     for pi/3, get
%  1/2  3/2  2  3/2 1/2 0 ...
% so difference multiplied by 2 is
%   2 0 -4 0 2 0 2 0 -4 0 2 0 ....
%  then add another 2pi/3, get
%  7/2 3/2 -4  3/2 7/2  0 7/2 3/2 -4 ...
%  7/2 + 3/2(1/4) - 4/9 =
\begin{equation}\label{three6}  %
-F(2\pi/3)-2[F(2\pi/3)-F(\pi/3)]\leq
-\frac{3}{\pi^{2}}\lambda_{1,n}^{r,s}.
\end{equation}
Using \eqref{four8}, we can estimate
\begin{flalign*}
&-\pi^{2}[F(2\pi/3) - F(\pi/6)]\\
&\qquad=-\frac{1+\sqrt{3}}{2}\frac{\lambda_{1,n}^{r,s}}{1}-\frac{\lambda_{2,n}^{r,s}}{4} +\frac{\lambda_{3,n}^{r,s}}{9}
-\frac{1-\sqrt{3}}{2}\frac{\lambda_{5,n}^{r,s}}{25} +2\frac{\lambda_{6,n}^{r,s}}{36}
-\frac{1-\sqrt{3}}{2}\frac{\lambda_{7,n}^{r,s}}{49} +\frac{\lambda_{9,n}^{r,s}}{81}+\cdots\\
&\qquad\leq-1.44\lambda_{1,n}^{r,s}.
\end{flalign*}
%%% for 2pi/3, get
%  3/2 3/2 0 3/2 3/2 0 ...
%     for pi/6, get
%  1-rt(3)/2  1/2  1  3/2  1+rt(3)/2  2 1+rt(3)/2  3/2 1 1/2 1-rt(3)/2 0 ...
% so difference is
%  (1+rt(3))/2  1  -1  0  1 -(rt3)/2  -2  1 - (rt3)/2  0  -1  1  (1+rt(3))/2  0 ...
%
% so get (1+sqrt(3))/2  +1/4  -1/9  + (1-sqrt(3)/2)/25 - 2/36  + (1-sqrt(3)/2)/49  -1/81
%
%1/162 + (sqrt(3)-1)/(4*49) + 1/36  + (sqrt(3)-1)/100  + 1/18  = .1006....

Adding this to \eqref{three5},
$$-F(2\pi/3)+F(\pi/3)-2[F(2\pi/3) - F(\pi/6)]
\leq -\frac{2.44}{\pi^{2}}\lambda_{1,n}^{r,s}.$$
Combining with \eqref{three6} in \eqref{three7},
%  2/9 + 5/18 = 9/18 = 1/2
$$-3F(2\pi/3)+\sum_{i=1}^{3}F(\theta_{i})\leq -\frac{2.44}{\pi^{2}}\lambda_{1,n}^{r,s}.$$%
Also, $\sum_{i=1}^{3}\Big(\frac{\theta_{i}}{2\pi}-1/3\Big)^{2}\leq (2/3)^{2}+ 2(1/3)^{2}=2/3$.  In summary,
\begin{equation}\label{three9}
-3F(2\pi/3)+\sum_{i=1}^{3}F(\theta_{i})
\leq-\frac{2.44}{\pi^{2}}\lambda_{1,n}^{r,s}
\leq-\frac{7.32}{2\pi^{2}}\lambda_{1,n}^{r,s}\sum_{i=1}^{3}\Big(\frac{\theta_{i}}{2\pi}-1/3\Big)^{2}.
\end{equation}

\textbf{Case 3}.   If $4\pi/3\geq\theta_{1}\geq \pi$ and $\theta_{2}-\theta_{3}\geq\pi/2$, then by Lemma \ref{lemma5}, we have (using $F(a)=F(2\pi-a)$ for all $0\leq a\leq 2\pi$, which follows by \eqref{four7}),
$$\sum_{i=1}^{3}F(\theta_{i})\leq \max(2F(\pi),\, F(\pi) + F(\pi/4)+F(3\pi/4)).$$
Using \eqref{four8} (or \eqref{three1} below), we can estimate
$$
-3F(2\pi/3) + 2F(\pi)\leq-\frac{1}{2\pi^{2}}\Big(\lambda_{1,n}^{r,s}+\lambda_{2,n}^{r,s}\Big).
$$
Similarly,
\begin{equation}\label{three11}
-3F(2\pi/3) + F(\pi)+F(\pi/4)+F(3\pi/4)\leq-\frac{1}{\pi^{2}}\frac{1}{2}\lambda_{1,n}^{r,s}.
\end{equation}
In more detail, the coefficient pattern here is $24$-periodic by \eqref{four8}, so that
\begin{flalign*}
&\pi^{2}\Big(3F(2\pi/3) -[ F(\pi)+F(\pi/4)+F(3\pi/4)]\Big)\\
&=\frac{1}{2}\frac{\lambda_{1,n}^{r,s}}{1}+\frac{7}{2}\frac{\lambda_{2,n}^{r,s}}{4}-4\frac{\lambda_{3,n}^{r,s}}{9}+\frac{5}{2}\frac{\lambda_{4,n}^{r,s}}{16}
+\frac{1}{2}\frac{\lambda_{5,n}^{r,s}}{25}-1\frac{\lambda_{6,n}^{r,s}}{36}+\frac{1}{2}\frac{\lambda_{7,n}^{r,s}}{49}+\frac{9}{2}\frac{\lambda_{8,n}^{r,s}}{64}\\
&\,\,-4\frac{\lambda_{9,n}^{r,s}}{81}+\frac{7}{2}\frac{\lambda_{10,n}^{r,s}}{100}+\frac{1}{2}\frac{\lambda_{11,n}^{r,s}}{121}-2\frac{\lambda_{12,n}^{r,s}}{144}
+\frac{1}{2}\frac{\lambda_{13,n}^{r,s}}{169}+\frac{7}{2}\frac{\lambda_{14,n}^{r,s}}{196}-4\frac{\lambda_{15,n}^{r,s}}{225}+\frac{9}{2}\frac{\lambda_{16,n}^{r,s}}{256}\\
&\,\,+\frac{1}{2}\frac{\lambda_{17,n}^{r,s}}{289}-1\frac{\lambda_{18,n}^{r,s}}{324}+\frac{1}{2}\frac{\lambda_{19,n}^{r,s}}{361}+\frac{5}{2}\frac{\lambda_{20,n}^{r,s}}{400}
-4\frac{\lambda_{21,n}^{r,s}}{441}+\frac{7}{2}\frac{\lambda_{22,n}^{r,s}}{484}+\frac{1}{2}\frac{\lambda_{23,n}^{r,s}}{529}+0\frac{\lambda_{24,n}^{r,s}}{576}+\cdots
\end{flalign*}
From the seventh term onwards, the sum of the coefficients is nonnegative by \eqref{four10}.  Then the sum of the first six terms is bounded as in \eqref{three11}, using \eqref{four10} again.% \snote{more detail}
%  9/2 9/2 0 9/2 9/2 0 ...
%  then
%  2 0 2 0                                    2 0 2 0 ....
% 1-rt(2)/2  1/2 1+ rt(2)/2  1       1+rt(2)/2   1/2  1-rt(2)/2   0
% 1+rt(2)/2  1/2 1- rt(2)/2  1       1-rt(2)/2   1/2  1+rt(2)/2   0
% pi + pi/4 + 3pi/4  get
%  4 1 4 2      4 1 4 0
% all together get 24 periodic patterm
%  1/2 7/2 -4   5/2             1/2 -1 1/2 9/2             -4 7/2 1/2 -2
%  1/2 7/2 -4   9/2             1/2 -1 1/2 5/2             -4 7/2 1/2 0
So,
$$
-3F(2\pi/3)+\sum_{i=1}^{3}F(\theta_{i})\leq-\frac{1}{\pi^{2}}\frac{1}{2}\lambda_{1,n}^{r,s}.
$$

Also, $\sum_{i=1}^{3}\left(\frac{\theta_{i}}{2\pi}-1/3\right)^{2}\leq 2(1/3)^{2}=2/9$.  In summary,
\begin{equation}\label{three10}
-3F(2\pi/3)+\sum_{i=1}^{3}F(\theta_{i})
\leq-\frac{1}{\pi^{2}}\frac{1}{2}\lambda_{1,n}^{r,s}.
\leq-\frac{9}{4\pi^{2}}\lambda_{1,n}^{r,s}\sum_{i=1}^{3}\Big(\frac{\theta_{i}}{2\pi}-1/3\Big)^{2}.
\end{equation}

\textbf{Case 4}.   If $7\pi/6\geq\theta_{1}\geq \pi$ and $\theta_{2}-\theta_{3}\leq\pi/2$, then arguing as before,

$$\sum_{i=1}^{3}F(\theta_{i})\leq \max(F(\pi) + 2F(\pi/2),\, F(\pi)+F(\pi/4)+F(3\pi/4) ).$$
The second term is bounded already in \eqref{three11}.  For the first term, we have  by \eqref{four8}
\begin{equation}\label{three14}
-3F(2\pi/3)+F(\pi)+F(\pi/4)+F(3\pi/4)
\leq -\frac{1}{\pi^{2}}\frac{13}{72}\lambda_{1,n}^{r,s}.
\end{equation}
%  9/2 9/2 0 9/2 9/2 0 ...
% 2 0 2 0 2 0 2 0 ....
% 1/2 1 1/2 0 1/2 1 1/2 0
% pi + 2 pi/2 get
%  4 4 4 0 4 4 4 0 4 4 4 0
% all together get 12 periodic patterm
%  1/2 1/2 -4   9/2     1/2 -4 1/2 9/2  -4 1/2 1/2 0

In more detail, the coefficient pattern here is $12$-periodic by \eqref{four8}, so that
\begin{flalign*}
&\pi^{2}\Big(3F(2\pi/3) -[ F(\pi)+ 2F(\pi/2)]\Big)\\
&\qquad\qquad=\frac{1}{2}\frac{\lambda_{1,n}^{r,s}}{1}+\frac{1}{2}\frac{\lambda_{2,n}^{r,s}}{4}-4\frac{\lambda_{3,n}^{r,s}}{9}+\frac{9}{2}\frac{\lambda_{4,n}^{r,s}}{16}
+\frac{1}{2}\frac{\lambda_{5,n}^{r,s}}{25} -4\frac{\lambda_{6,n}^{r,s}}{36} \\
&\qquad\qquad+\frac{1}{2}\frac{\lambda_{7,n}^{r,s}}{49}+\frac{9}{2}\frac{\lambda_{8,n}^{r,s}}{64} -4\frac{\lambda_{9,n}^{r,s}}{81}+\frac{1}{2}\frac{\lambda_{10,n}^{r,s}}{100}
+\frac{1}{2}\frac{\lambda_{11,n}^{r,s}}{121}+0\frac{\lambda_{12,n}^{r,s}}{144}+\cdots.
\end{flalign*}
From the fourth term onwards, the sum of the coefficients is nonnegative by\eqref{four10}.  Then the sum of the first three terms is bounded as in \eqref{three14}, using \eqref{four10} again.% \snote{more detail}
So, combining \eqref{three14} and \eqref{three11}
$$
-3F(2\pi/3)+\sum_{i=1}^{3}F(\theta_{i})\leq-\frac{1}{\pi^{2}}\frac{13}{72}\lambda_{1,n}^{r,s}.
$$

% max achieved when theta1 = 7pi/6, theta2 = pi/2, theta3 = (5/6 - 1/2) pi = pi/3.  L2^2 = 1/4^2  + 1/12^2  + 1/6^2 = 1/16 + 1/144 + 1/36  = 7/72
%  or when  theta1 = 7pi/6, theta2 = 5pi/12, theta3 = 5pi/12      L2^2 = 1/4^2  + 2*(1/8)^2 = 1/16 + 1/32 = 3/32 =
Also, $\sum_{i=1}^{3}\left(\frac{\theta_{i}}{2\pi}-1/3\right)^{2}\leq 2(1/4)^{2}=1/8$.  In summary,

\begin{equation}\label{three15}
-3F(2\pi/3)+\sum_{i=1}^{3}F(\theta_{i})
\leq-\frac{1}{\pi^{2}}\frac{13}{72}\lambda_{1,n}^{r,s}.
\leq-\frac{1}{\pi^{2}}\frac{13}{9}\lambda_{1,n}^{r,s}\sum_{i=1}^{3}\Big(\frac{\theta_{i}}{2\pi}-1/3\Big)^{2}.
\end{equation}

\textbf{Case 5}.   If $4\pi/3\geq\theta_{1}\geq 7\pi/6$ and $\theta_{2}-\theta_{3}\leq\pi/2$,

\begin{flalign*}
\sum_{i=1}^{3}F(\theta_{i})
&\leq \max(F(7\pi/6) + 2F(5\pi/12),\, F(7\pi/6)+F(\pi/6)+F(2\pi/3) )\\
&= \max(F(5\pi/6) + 2F(5\pi/12),\, F(5\pi/6)+F(\pi/6)+F(2\pi/3) ).
\end{flalign*}
For the first term, we have  by \eqref{four8}
\begin{equation}\label{three12}
-3F(2\pi/3)+F(5\pi/6) + 2F(5\pi/12)
\leq -\frac{1}{\pi^{2}}(.728)\lambda_{1,n}^{r,s}.
\end{equation}

In more detail, the coefficient pattern here is $24$-periodic by \eqref{four8}, so that
\begin{flalign*}
&\pi^{2}\Big(3F(2\pi/3)-[F(5\pi/6) + 2F(5\pi/12)]\Big)\\
&\qquad\qquad=1.1516\ldots\frac{\lambda_{1,n}^{r,s}}{1}+(2-\sqrt{3})\frac{\lambda_{2,n}^{r,s}}{4}-(3+\sqrt{2})\frac{\lambda_{3,n}^{r,s}}{9}+2\frac{\lambda_{4,n}^{r,s}}{16}\\
&\qquad\qquad+4.29\ldots\frac{\lambda_{5,n}^{r,s}}{25} -4\frac{\lambda_{6,n}^{r,s}}{36} +.434\ldots\frac{\lambda_{7,n}^{r,s}}{49}+0\frac{\lambda_{8,n}^{r,s}}{64}\\
&\qquad\qquad -(3-\sqrt{2})\frac{\lambda_{9,n}^{r,s}}{81}+(2-\sqrt{3})\frac{\lambda_{10,n}^{r,s}}{100}+.116\ldots\frac{\lambda_{11,n}^{r,s}}{121}-4\frac{\lambda_{12,n}^{r,s}}{144}\\
&\qquad\qquad+.116\ldots\frac{\lambda_{13,n}^{r,s}}{169} +(2-\sqrt{3})\frac{\lambda_{14,n}^{r,s}}{196} -(3-\sqrt{2})\frac{\lambda_{15,n}^{r,s}}{225} +0\frac{\lambda_{16,n}^{r,s}}{256}\\
&\qquad\qquad +.434\ldots\frac{\lambda_{17,n}^{r,s}}{289} -4\frac{\lambda_{18,n}^{r,s}}{324}  +4.29\ldots\frac{\lambda_{19,n}^{r,s}}{361}+2\frac{\lambda_{20,n}^{r,s}}{400}\\
&\qquad\qquad-(3+\sqrt{2})\frac{\lambda_{21,n}^{r,s}}{441} +(2-\sqrt{3})\frac{\lambda_{22,n}^{r,s}}{484}+ 1.1516\ldots\frac{\lambda_{23,n}^{r,s}}{529}+0\frac{\lambda_{24,n}^{r,s}}{576}+\cdots.
\end{flalign*}
From the nineteenth term onwards, the sum of the coefficients is nonnegative, so \eqref{four10} implies that these terms have a nonnegative sum (over a period of 24 terms, the sum of coefficients is $2$, and terms four through eighteen have a nonnegative sum.)

For the second term, we have  by \eqref{four8}
\begin{equation}\label{three16}
-3F(2\pi/3)+[F(5\pi/6)+F(\pi/6)+F(2\pi/3)]
\leq -\frac{1}{\pi^{2}}\lambda_{1,n}^{r,s}.
\end{equation}
%  %  (sin pi/12)^2 = 1/2 - sqrt(3)/4
%  = 2f(2pi/3) - f(5pi/6)  - f(pi/6)
%    3 3 0 3 3 0 3 3 0 3 3 0
%  .5-rt(3)/4   1/4  1/2   3/4         .5 +rt(3)/4  1   .5 + rt(3)/4  3/4           1/2   1/4    .5 - rt(3)/4   0
%  .5+rt(3)/4   1/4  1/2   3/4         .5 -rt(3)/4  1   .5 - rt(3)/4  3/4           1/2   1/4    .5 + rt(3)/4   0
%
%  1 2 -2 0 1 -4 1 0 -2 2 1 0
%
%  sum of all terms over 12 period is
%  6*4  - 2*(1 + 1/2 + 1 + 3/2 + 1 + 2 + 1 + 3/2 + 1 + 1/2 + 1) =0

In more detail, the coefficient pattern here is $12$-periodic by \eqref{four8}, so that
\begin{flalign*}
&\pi^{2}\Big(3F(2\pi/3)-[F(2\pi/3) +F(5\pi/6) + F(\pi/6)]\Big)\\
&\qquad\qquad=\frac{\lambda_{1,n}^{r,s}}{1}+2\frac{\lambda_{2,n}^{r,s}}{4}-2\frac{\lambda_{3,n}^{r,s}}{9}+0\frac{\lambda_{4,n}^{r,s}}{16}
+1\frac{\lambda_{5,n}^{r,s}}{25} -4\frac{\lambda_{6,n}^{r,s}}{36} \\
&\qquad\qquad+\frac{\lambda_{7,n}^{r,s}}{49}+0\frac{\lambda_{8,n}^{r,s}}{64}
-2\frac{\lambda_{9,n}^{r,s}}{81}+2\frac{\lambda_{10,n}^{r,s}}{100}+\frac{\lambda_{11,n}^{r,s}}{121}+0\frac{\lambda_{12,n}^{r,s}}{144}+\cdots.
\end{flalign*}
From the tenth term onwards, the sum of the coefficients is nonnegative, so \eqref{four10} implies that these terms have a nonnegative sum (over a period of 12 terms, the sum of coefficients is $0$, and terms two through nine have a nonnegative sum.)

So, combining \eqref{three12} and \eqref{three16},
$$
-3F(2\pi/3)+\sum_{i=1}^{3}F(\theta_{i})\leq-\frac{1}{\pi^{2}}(.728)\lambda_{1,n}^{r,s}.
$$

% max achieved when theta1 = 4pi/3, theta2 = pi/2, theta3 = (2/3 - 1/2) pi = pi/6.  L2^2 = 1/3^2  + 1/12^2  + 1/4^2 = 1/9 + 1/144 + 1/16  = 13/72
%  or when  theta1 = 4pi/3, theta2 = pi/3, theta3 = pi/3      L2^2 = 1/3^2  + 2*(1/6)^2 = 1/9 + 2/36 = 1/9 + 1/18 = 3/18 = 1/6
Also, $\sum_{i=1}^{3}\left(\frac{\theta_{i}}{2\pi}-1/3\right)^{2}\leq (1/3)^{2}+1/12^2 + 1/6^2=7/48$.  In summary,

\begin{equation}\label{three13}
-3F(2\pi/3)+\sum_{i=1}^{3}F(\theta_{i})
\leq-\frac{1}{\pi^{2}}(.728)\lambda_{1,n}^{r,s}.
\leq-\frac{1}{\pi^{2}}(.728)\frac{48}{7}\lambda_{1,n}^{r,s}\sum_{i=1}^{3}\Big(\frac{\theta_{i}}{2\pi}-1/3\Big)^{2}.
\end{equation}

\textbf{Combining Cases 1 through 5}.

Combining \eqref{three3}, \eqref{three9}, \eqref{three10}, \eqref{three15} and \eqref{three13}, we have, whenever $\theta_{1}\geq\pi$,
$$
-3F(2\pi/3)+\sum_{i=1}^{3}F(\theta_{i})
\leq-\frac{1}{\pi^{2}}\frac{13}{9}\lambda_{1,n}^{r,s}\sum_{i=1}^{3}\Big(\frac{\theta_{i}}{2\pi}-1/3\Big)^{2}.
$$

\end{proof}

%$(1/6)\exp(-\alpha\pi/2) + (1/2)\exp(-\alpha \pi/6)$

Lemma \ref{lemma7} and Lemma \ref{lemma10} and the following inequality imply Corollary \ref{cor1} below.
$$-\frac{x}{1/2+\sqrt{9/4+x^{2}}}\leq (3/4)\Big(-1+\Big(1-\frac{3^{4/3}}{5}\Big) e^{-x\frac{\pi}{2}}+\frac{3^{4/3}}{5}e^{-x\frac{\pi}{6}}\Big),\qquad\forall\,x>0.$$
%x=linspace(0,10,1000);
%plot (x, -x./(.5+sqrt((9/4)+x.^2))  , x, (3/4)*(-1 + (1-(3^(4/3) /5))*exp(-x*pi/2) +(3^(4/3) /5)*exp(-x*pi/6))  ); legend('lam','1-exp')
\begin{cor}\label{cor1}
Let $\theta_{1},\theta_{2},\theta_{3}\geq0$.  Define $F$ as in \eqref{fdef}.  Then
$$
-3F(2\pi/3)+\sum_{i=1}^{3}F(\theta_{i})
\leq
-\frac{1}{\pi^{2}}\frac{13}{9}\frac{3}{4}\Big(1-\Big(1-\frac{3^{4/3}}{5}\Big) e^{-\frac{\rho rs}{1-\rho^{2}}\frac{\pi}{2}}-\frac{3^{4/3}}{5}e^{-\frac{\rho rs}{1-\rho^{2}}\frac{\pi}{6}}\Big)\sum_{i=1}^{3}\Big(\frac{\theta_{i}}{2\pi}-1/3\Big)^{2}.
$$
\end{cor}

\section{Change of Measure}

The following inequality allows us to show that the first term in \eqref{in1} is smaller than the second one.  The proof amounts to an elementary truncated heat kernel bound, though with a change of measure using the $\lambda_{1,n}^{r,s}$ term from \eqref{one6}.

\begin{lemma}\label{lemma28}
Let $f\colon\R^{2}\to\Delta_{3}$ be a radial function (for any $r>0$, the function $f|_{rS^{1}}$ is constant).  Denote $\E_{\gamma} f\colonequals \int_{\R^{2}}f(x)\gamma_{2}(x)\,\d x$.  Let $\phi\colon\R^{2}\to(0,\infty)$.  Then
\begin{flalign*}
&\abs{\int_{\R^{2}}\langle f(x)-\E_{\gamma}f,\, T_{\rho}[f-\E_{\gamma}f](x)\rangle\gamma_{2}(x)\,\d x}\\
&\leq \int_{\R^{2}}\phi(x)\vnorm{f(x)-\E_{\gamma}f}^{2}\gamma_{2}(x)\,\d x \sum_{d\geq2\colon d\,\mathrm{even}}\rho^{d}\sum_{k\in(2\N)^{2}\colon \vnorm{k}_{1}=d}\int_{\R^{2}} \abs{\sqrt{k!}h_{k}(x)\frac{1}{\sqrt{\phi(x)}}}^{2}\gamma_{2}(x)\,\d x.
\end{flalign*}
In particular, if $\phi(x)=(1-e^{-\rho\vnorm{x}/2})$ for all $x\in\R^{2}$, and $0<\rho<1/7$, then
$$
\abs{\int_{\R^{2}}\langle f(x)-\E_{\gamma}f,\, T_{\rho}[f-\E_{\gamma}f](x)\rangle\gamma_{2}(x)\,\d x}
\leq 2.5(\rho+\rho^{2})\int_{\R^{2}}\phi(x)\vnorm{f(x)-\E_{\gamma}f}^{2}\gamma_{2}(x)\d x.
$$
\end{lemma}
\begin{proof}
Let $h_{0},h_{1},\ldots\colon\R\to\R$ be the Hermite polynomials with $h_{m}(x)=\sum_{k=0}^{\lfloor m/2\rfloor}\frac{x^{m-2k}(-1)^{k}2^{-k}}{k!(m-2k)!}$ $\forall$ $m\geq0$, $m\in\Z$, $\forall$ $x\in\R$.  It is well known \cite{heilman12} that $\{\sqrt{m!}h_{m}\}_{m\geq0}$ is an orthonormal basis of the Hilbert space of functions $\R\to\R$ equipped with the inner product $\langle g,h\rangle\colonequals\int_{\R}g(x)h(x)\gamma_{2}(x)\,\d x$.  For any $k\in\N^{2}$, define $k!\colonequals k_{1}!\cdot k_{2}!$, and let $\vnorm{k}_{1}\colonequals\abs{k_{1}}+\abs{k_{2}}$.  Let $\alpha_{k}\in\R$ to be determined later.  Since $f$ is radial, the Hermite-Fourier expansion of $f-\E_{\gamma}f$ is
\begin{flalign*}
&\int_{\R^{2}}\langle f(x)-\E_{\gamma}f,\, T_{\rho}[f-\E_{\gamma}f](x)\rangle\gamma_{2}(x)\,\d x\\
&=\sum_{d\geq2\colon d\,\mathrm{even}}\rho^{d}\sum_{k\in(2\N)^{2}\colon \vnorm{k}_{1}=d}\vnorm{\int_{\R^{2}} \sqrt{k!}h_{k}(x) (f(x)-\E_{\gamma}f)\gamma_{2}(x)\,\d x}^{2}\\
&=\sum_{d\geq2\colon d\,\mathrm{even}}\rho^{d}\sum_{k\in(2\N)^{2}\colon \vnorm{k}_{1}=d}\vnorm{\int_{\R^{2}} \sqrt{k!}[h_{k}(x)-\alpha_{k}] (f(x)-\E_{\gamma}f)\gamma_{2}(x)\,\d x}^{2}\\
&=\sum_{d\geq2\colon d\,\mathrm{even}}\rho^{d}\sum_{k\in(2\N)^{2}\colon \vnorm{k}_{1}=d}\vnorm{\int_{\R^{2}} \sqrt{k!}[h_{k}(x)-\alpha_{k}]\frac{1}{\sqrt{\phi(x)}} \sqrt{\phi(x)}(f(x)-\E_{\gamma}f)\gamma_{2}(x)\,\d x}^{2}\\
&\leq\sum_{\substack{d\geq2\colon\\ d\,\mathrm{even}}}\rho^{d}\sum_{\substack{k\in(2\N)^{2}\colon\\ \vnorm{k}_{1}=d}}\int_{\R^{2}} \abs{\sqrt{k!}[h_{k}(x)-\alpha_{k}]\frac{1}{\sqrt{\phi(x)}}}^{2}\gamma_{2}(x)\,\d x
\cdot\int_{\R^{2}}\phi(x)\vnorm{f(x)-\E_{\gamma}f}^{2}\gamma_{2}(x)\,\d x\\
&=\int_{\R^{2}}\phi(x)\vnorm{f(x)-\E_{\gamma}f}^{2}\gamma_{2}(x)\,\d x\cdot \sum_{\substack{d\geq2\colon\\ d\,\mathrm{even}}}\rho^{d}\sum_{\substack{k\in(2\N)^{2}\colon\\ \vnorm{k}_{1}=d}}\int_{\R^{2}} \abs{\sqrt{k!}[h_{k}(x)-\alpha_{k}]\frac{1}{\sqrt{\phi(x)}}}^{2}\gamma_{2}(x)\,\d x.
\end{flalign*}

To conclude, it remains to bound the last term.  We use the inequality
$$\frac{r}{1-e^{-\rho r/2}}\leq r+\frac{2}{\rho},\qquad\forall\,r>0,$$
which follows from the inequality $a\leq e^{a}-1$ for all $a>0$ (where $a=\rho r/2$), along with
$$\frac{r}{1-e^{-\rho r/2}}\leq r/1.7+\frac{2}{\rho},\qquad\forall\,0<r<1/\rho,$$
we obtain (using $G_{\rho}(x,y)=\gamma_{2}(x)\gamma_{2}(y)\sum_{d\geq0}\rho^{d}\sum_{k\in\N^{2}\colon\vnorm{k}_{1}=d}h_{k}(x)h_{k}(y)k!$, i.e.  the expansion of the Mehler kernel from \eqref{gdef} into Hermite polynomials), and choosing $\alpha_{k}\colonequals \alpha\cdot 1_{\{k=(2,0)\}}$,

\begin{flalign*}
&\sum_{d\geq2\colon d\,\mathrm{even}}\rho^{d}\sum_{k\in(2\N)^{2}\colon \vnorm{k}_{1}=d}\int_{\R^{2}} \abs{\sqrt{k!}[h_{k}(x)-\alpha_{k}]\frac{1}{\sqrt{\phi(x)}}}^{2}\gamma_{2}(x)\,\d x\\
&=\int_{\R^{2}}\frac{1}{\phi(x)}\Big[2\pi\frac{G_{\rho}(x,x)+G_{\rho}(x,-x)}{2e^{-\vnorm{x}^{2}/2}} -  \gamma_{2}(x)  -2\rho^{2}\alpha (x_{1}^{2}-1)\gamma_{2}(x)+\alpha^{2}\rho^{2}\gamma_{2}(x)\Big]\, \d x\\
&=\int_{\R^{2}}\frac{1}{\phi(x)}\Big[\frac{1}{2\pi}\frac{1}{1-\rho^{2}} e^{-\frac{\vnorm{x}^{2}}{1-\rho^{2}}}
\frac{e^{\frac{\rho\vnorm{x}^{2}}{1-\rho^{2}}}+e^{-\frac{\rho\vnorm{x}^{2}}{1-\rho^{2}}}}{2e^{-r^{2}/2}} - \gamma_{2}(x) -2\alpha\rho^{2}(x_{1}^{2}-1)\gamma_{2}(x)+\alpha^{2}\rho^{2}\gamma_{2}(x)\Big]\, \d x\\
&=\int_{r=0}^{\infty}((r/1.7)1_{r<1/\rho}+ r1_{r>1/\rho}+2/\rho)\frac{1}{2(1-\rho^{2})}\\
&\qquad\cdot\Big(-2(1-\rho^{2})(1-\alpha^{2}\rho^{2}+2\alpha\rho^{2}(r^{2}/2-1))e^{-r^{2}/2}+e^{-r^{2}\frac{1-\rho}{1-\rho^{2}}}e^{r^{2}/2}+e^{-r^{2}\frac{1+\rho}{1-\rho^{2}}}e^{r^{2}/2}\Big)\, \d r\\
%\end{flalign*}
%\begin{flalign*}
&\qquad=\int_{r=0}^{\infty}((r/1.7)1_{r<1/\rho}+ r1_{r>1/\rho}+2/\rho)\frac{1}{2(1-\rho^{2})}\\
&\qquad\qquad\cdot\Big(-2(1-\rho^{2})(1-\alpha^{2}\rho^{2}+2\alpha\rho^{2}(r^{2}/2-1))e^{-r^{2}/2}+e^{-r^{2}\left(\frac{1}{1+\rho}-\frac{1}{2}\right)}+e^{-r^{2}\left(\frac{1}{1-\rho}-\frac{1}{2}\right)}\Big)\, \d r\\
&\qquad=\frac{1}{2(1-\rho^{2})}\frac{1}{1.7}\Big(-2(1-\rho^{2})(1-\alpha^{2}\rho^{2})+\frac{1}{2[\frac{1}{1+\rho}-\frac{1}{2}]}+ \frac{1}{2[\frac{1}{1-\rho}-\frac{1}{2}]} \Big)\\
&\qquad\qquad+\frac{1}{2(1-\rho^{2})}\Big(1-\frac{1}{1.7}\Big)\Big(-2(1-\rho^{2})(1-\alpha^{2}\rho^{2}-\alpha)(-e^{-1/(2\rho^{2})})\\
&\qquad\qquad+\frac{1}{2[\frac{1}{1+\rho}-\frac{1}{2}]}(-e^{-\frac{1}{\rho^{2}}[\frac{1}{1+\rho}-\frac{1}{2}]})
+ \frac{1}{2[\frac{1}{1-\rho}-\frac{1}{2}]}(-e^{-\frac{1}{\rho^{2}}[\frac{1}{1-\rho}-\frac{1}{2}]}) \Big)\\
&\qquad\qquad+\frac{1}{\rho}\frac{1}{1-\rho^{2}}\sqrt{\frac{\pi}{2}}\Big(-2(1-\rho^{2})(1-\alpha^{2}\rho^{2}-\alpha\rho^{2})+\frac{1}{\sqrt{2[\frac{1}{1+\rho}-\frac{1}{2}]}}+\frac{1}{\sqrt{2[\frac{1}{1-\rho}-\frac{1}{2}]}}\Big).
\end{flalign*}%
%rho=linspace(.01,.99,1000);
%plot(rho, (1/1.7)*(1./(2*(1-rho.^2))).*(-2*(1- rho.^2)   + 1./(2* (-.5 +1./(1+rho))) + 1./(2* (-.5 +1./(1-rho)))   )     ...
%+(1- 1/1.7)*(1./(2*(1-rho.^2))).*(-2*(1- rho.^2).*(- exp( - ( 1./(1+rho)  - 1/2)./(rho.^2)) )   + (1./(2* (-.5 +1./(1+rho)))) .*(-exp(- (1./(1-rho)  - 1/2)./(rho.^2))) + (1./(2* (-.5 +1./(1-rho)))) .* (- exp(- (1 ./(1-rho)  -1/2)./ (rho.^2)))   )...
% +  (1./rho).*(1./((1-rho.^2))).*(-2*(1- rho.^2)   + 1./sqrt(2* (-.5 +1./(1+rho))) + 1./sqrt(2* (-.5 +1./(1-rho)))   )  , rho, 3*rho+ 3*rho.^2    )
%legend('orig','ub')
%axis([0 1 0 1])
%
%a=-1/2;
%rho=linspace(.01,.99,1000);
%plot(rho, (1/1.7)*(1./(2*(1-rho.^2))).*(-2*(1- rho.^2).*(1- (rho.^2).*a^2)   + 1./(2* (-.5 +1./(1+rho))) + 1./(2* (-.5 +1./(1-rho)))   )     ...
%+(1- 1/1.7)*(1./(2*(1-rho.^2))).*(-2*(1- rho.^2).*(1- (rho.^2).*a^2  - a ).*(- exp( - ( 1./(1+rho)  - 1/2)./(rho.^2)) )   + (1./(2* (-.5 +1./(1+rho)))) .*(-exp(- (1./(1-rho)  - 1/2)./(rho.^2))) + (1./(2* (-.5 +1./(1-rho)))) .* (- exp(- (1 ./(1-rho)  -1/2)./ (rho.^2)))   )...
% +  (1./rho).*(1./((1-rho.^2))).*(-2*(1- rho.^2).*(1- (rho.^2).*a^2 -(rho.^2)*a)   + 1./sqrt(2* (-.5 +1./(1+rho))) + 1./sqrt(2* (-.5 +1./(1-rho)))   )  , rho, 2.5*rho+ 2.5*rho.^2    )
%legend('orig','ub')
%axis([0 1/7 0 1/2])

When $\alpha=-1/2$ and $\rho<1/7$, this quantity is upper bounded by $2.5(\rho+\rho^{2})$.
\end{proof}

It would be interesting to find the optimal constant in Lemma \ref{lemma28}, since improving the constants here would also improve the constants in Theorem \ref{thm1}.

For the negative correlation case of Theorem \ref{thm1}, we require a bilinear version of Lemma \ref{lemma28} above.  The proof of Lemma \ref{lemma29} is similar to that of Lemma \ref{lemma28}, though the constants in Lemma \ref{lemma28} are slightly improved, hence our repetition of many of the same steps with different results.

%\snote{maybe keep critical point business until later, since}
\begin{lemma}\label{lemma29}
Let $\widetilde{f},\widetilde{g}\colon\R^{2}\to\Delta_{3}$.  Define
$$\phi(x)=\frac{\rho}{1-\rho^{2}}\vnorm{x}e^{-\frac{(1.1\rho\vnorm{x})^{2}}{1-\rho^{2}}-\frac{1.1\rho}{1-\rho^{2}}\vnorm{x}},\qquad\forall\,x\in\R^{2}.$$
If $0<\rho<.1$, then
\begin{flalign*}
&\sum_{d\geq2}\rho^{d}\sum_{k\in\N^{2}\colon \vnorm{k}_{1}=d}\abs{\Big\langle\int_{\R^{2}} \sqrt{k!}h_{k}(x) \widetilde{f}(x)\gamma_{2}(x)\,\d x,\,
\int_{\R^{2}} \sqrt{k!}h_{k}(y) \widetilde{g}(y)\gamma_{2}(y)\,\d y\Big\rangle}\\
&\qquad\leq \Big(5\rho+8\rho^{2}\Big)\cdot\frac{1}{2}\Big(\int_{\R^{2}}\phi(x)\vnormf{\widetilde{f}(x)}^{2}\gamma_{2}(x)\,\d x +\int_{\R^{2}}\phi(x)\vnormf{\widetilde{g}(x)}^{2}\gamma_{2}(x)\,\d x\Big).
\end{flalign*}
\end{lemma}
\begin{proof}

We will show that
\begin{equation}\label{five1}
\begin{aligned}
&\sum_{d\geq2}\rho^{d}\sum_{k\in\N^{2}\colon \vnorm{k}_{1}=d}\vnorm{\int_{\R^{2}} \sqrt{k!}h_{k}(x) \widetilde{f}(x)\gamma_{2}(x)\,\d x}^{2}\\
&\qquad\qquad\qquad\leq (5\rho+8\rho^{2})\int_{\R^{2}}\phi(x)\vnormf{\widetilde{f}(x)}^{2}\gamma_{2}(x)\d x.
\end{aligned}
\end{equation}
%\snote{how does it follows from cauchy-schwarz exactly?}
The conclusion of the Lemma then follows from this inequality and the Cauchy-Schwarz inequality.  We therefore prove \eqref{five1}.

Let $h_{0},h_{1},\ldots\colon\R\to\R$ be the Hermite polynomials with $h_{m}(x)=\sum_{k=0}^{\lfloor m/2\rfloor}\frac{x^{m-2k}(-1)^{k}2^{-k}}{k!(m-2k)!}$ for all integers $m\geq0$.  It is well known \cite{heilman12} that $\{\sqrt{m!}h_{m}\}_{m\geq0}$ is an orthonormal basis of the Hilbert space of functions $\R\to\R$ equipped with the inner product $\langle g,h\rangle\colonequals\int_{\R}g(x)h(x)\gamma_{2}(x)\,\d x$.  For any $k\in\N^{2}$, define $k!\colonequals k_{1}!\cdot k_{2}!$, and define $\vnorm{k}_{1}\colonequals\abs{k_{1}}+\abs{k_{2}}$.
\begin{flalign*}
&\sum_{d\geq2}\rho^{d}\sum_{k\in\N^{2}\colon \vnorm{k}_{1}=d}\vnorm{\int_{\R^{2}} \sqrt{k!}h_{k}(x) \widetilde{f}(x)\gamma_{2}(x)\,\d x}^{2}\\
&=\sum_{d\geq2}\rho^{d}\sum_{k\in\N^{2}\colon \vnorm{k}_{1}=d}\vnorm{\int_{\R^{2}} \sqrt{k!}h_{k}(x)\frac{1}{\sqrt{\phi(x)}} \sqrt{\phi(x)}\widetilde{f}(x)\gamma_{2}(x)\,\d x}^{2}\\
&\leq\sum_{d\geq2}\rho^{d}\sum_{k\in\N^{2}\colon \vnorm{k}_{1}=d}\int_{\R^{2}} \abs{\sqrt{k!}h_{k}(x)\frac{1}{\sqrt{\phi(x)}}}^{2}\gamma_{2}(x)\,\d x
\cdot\int_{\R^{2}}\phi(x)\vnormf{\widetilde{f}(x)}^{2}\gamma_{2}(x)\,\d x\\
&=\int_{\R^{2}}\phi(x)\vnormf{\widetilde{f}(x)}^{2}\gamma_{2}(x)\,\d x\cdot \sum_{d\geq2}\rho^{d}\sum_{k\in\N^{2}\colon \vnorm{k}_{1}=d}
\int_{\R^{2}} \abs{\sqrt{k!}h_{k}(x)\frac{1}{\sqrt{\phi(x)}}}^{2}\gamma_{2}(x)\,\d x.
\end{flalign*}

%Using the inequality
%$$\frac{r}{1-e^{-\rho r/2}}\leq r+\frac{2}{\rho},\qquad\forall\,r>0,$$
%which follows from the inequality $a\leq e^{a}-1$ for all $a>0$ (where $a=\rho r/2$), we obtain (
It remains to bound the final term.  Using the expansion of the Mehler kernel from \eqref{gdef} into Hermite polynomials, i.e. $G_{\rho}(x,y)=\gamma_{2}(x)\gamma_{2}(y)\sum_{d\geq0}\rho^{d}\sum_{k\in\N^{2}\colon\vnorm{k}_{1}=d}h_{k}(x)h_{k}(y)k!$,
%  r \leq (1-exp)(r+ 2/rho)   = r + 2/rho  - exp  * r  - exo * 2/rho
%    exp(..)  (r+ 2/rho) \leq  2/rho
%   r+ 2/rho \leq (2/rho) exp(r\rho /2)
%     rho r/2 + 1 \leq exp(r rho /2)
%   a +1 \leq exp(a)
%
%
%%$$\frac{r}{1-e^{-\rho r/2}}\leq r/1.7+\frac{2}{\rho},\qquad\forall\,0<r<1/\rho,$$
%%%  r \leq (1-exp)(r/1.7+ 2/rho)   = r/1.7 + 2/rho  - exp  * r/2  - exo * 2/rho
%%%    exp(..)  (r+ 2/rho) \leq  2/rho  - r(1-1/1.7)
%%%   r+ 2/rho \leq [(2/rho) - r(1-1/1.7) ] exp(r\rho /2)
%%%     rho r/2 + 1 \leq (1 - r rho (1-1/1.7)/2 ) exp (r rho /2)
%%%   a +1 \leq (1 - a(1-1/1.7)/2) exp(a)
%%%  true when a< 1/2   rho r <1

\begin{flalign*}
&\sum_{d\geq2}\rho^{d}\sum_{k\in\N^{2}\colon \vnorm{k}_{1}=d}\int_{\R^{2}} \abs{\sqrt{k!}h_{k}(x)\frac{1}{\sqrt{\phi(x)}}}^{2}\gamma_{2}(x)\,\d x\\
&=\int_{\R^{2}}\frac{1}{\phi(x)}\Big[2\pi\frac{G_{\rho}(x,x)}{e^{-\vnorm{x}^{2}/2}} -  \gamma_{2}(x)-\rho \vnorm{x}^{2}\gamma_{2}(x)\Big]\, \d x\\
&=\frac{1}{1-\rho^{2}}\int_{\R^{2}}\frac{1}{\phi(x)}\Big[\frac{1}{2\pi} e^{-\frac{\vnorm{x}^{2}(1-\rho - (1-\rho^{2})/2)}{1-\rho^{2}}}  - (1-\rho^{2})(1+\rho\vnorm{x}^{2})\gamma_{2}(x)\Big]\, \d x\\
&\leq\frac{1}{1-\rho^{2}}\int_{r=0}^{\infty}\frac{1-\rho^{2}}{\rho}\frac{r}{r}e^{-\frac{(1.1\rho r)^{2}}{1-\rho^{2}}-\frac{1.1\rho}{1-\rho^{2}}r}\Big[ e^{-\frac{r^{2}(1/2 - \rho  +\rho^{2}/2)}{1-\rho^{2}}}
  -  (1-\rho^{2})(1+\rho r^{2})e^{-r^{2}/2}\Big]\, \d r\\
  &=\frac{1}{\rho}\int_{r=0}^{\infty}e^{-\frac{(1.1\rho r)^{2}}{1-\rho^{2}}-\frac{1.1\rho}{1-\rho^{2}}r}\Big[ e^{-\frac{r^{2}(1/2 - \rho  +\rho^{2}/2)}{1-\rho^{2}}}
  -  (1-\rho^{2})(1+\rho r^{2})e^{-r^{2}/2}\Big]\, \d r\\
  &=\frac{1}{\rho}\Big(e^{\frac{(1.1\rho)^{2}}{4((1.1\rho)^{2}+ 1/2 - \rho + \rho^{2}/2)(1-\rho^{2})}}\sqrt{\frac{1-\rho^{2}}{(1.1\rho)^{2}+ 1/2 - \rho +\rho^{2}/2}}\int_{\frac{1.1\rho\sqrt{(1.1\rho)^{2}+ 1/2 - \rho + \rho^{2}/2}}{\sqrt{1-\rho^{2}}}}^{\infty}e^{-r^{2}}\,\d r\\
  &\qquad - (1-\rho^{2})e^{\frac{(1.1\rho)^{2}}{4((1.1\rho)^{2}+ 1/2 )(1-\rho^{2})}}\sqrt{\frac{1-\rho^{2}}{(1.1\rho)^{2}+ 1/2 }}\int_{\frac{1.1\rho\sqrt{(1.1\rho)^{2}+ 1/2 }}{\sqrt{1-\rho^{2}}}}^{\infty}e^{-r^{2}}\,\d r\\
%  e^{\frac{(1.1\rho)^{2}}{4((1.1\rho)^{2}+ 1/2 )(1-\rho^{2})}}\sqrt{\frac{1-\rho^{2}}{(1.1\rho)^{2}+ 1/2 }}\int_{\frac{1.1\rho\sqrt{(1.1\rho)^{2}+ 1/2 }}{\sqrt{1-\rho^{2}}}}^{\infty}e^{-r^{2}}\,\d r\\
  %\\
  &\qquad - (1-\rho^{2})\rho \Big[ \frac{2[ ((1.1 \rho)^2 + 1/2)     / (1-\rho^2)]+[1.1 \rho / (1-\rho^2)]^{2}}{4}e^{\frac{(1.1\rho)^{2}}{4((1.1\rho)^{2}+ 1/2 )(1-\rho^{2})}}\\
  &\qquad\qquad\cdot\Big(\frac{1-\rho^{2}}{(1.1\rho)^{2}+ 1/2 }\Big)^{5/2}\int_{\frac{1.1\rho\sqrt{(1.1\rho)^{2}+ 1/2 }}{\sqrt{1-\rho^{2}}}}^{\infty}e^{-r^{2}}\,\d r
   +\frac{-1.1 \rho / (1-\rho^2)}{2[((1.1 \rho)^2 + 1/2)     / (1-\rho^2)]^{2}}\Big]\Big).
%  e^{\frac{(1.1\rho)^{2}}{4((1.1\rho)^{2}+ 1/2 )(1-\rho^{2})}}\sqrt{\frac{1-\rho^{2}}{(1.1\rho)^{2}+ 1/2 }}\int_{\frac{1.1\rho\sqrt{(1.1\rho)^{2}+ 1/2 }}{\sqrt{1-\rho^{2}}}}^{\infty}e^{-r^{2}}\,\d r\\
 %  +
\end{flalign*}
% first term, a= (1.1 rho)^2 + (1/2 - rho+ rho^2/2)     / (1-rho^2)
%               b= -1.1 rho / (1-rho^2)
%
%
% \int e^{-ax^{2}+bx}dx
% = e^(b^2 / 4a) sqrt(pi)/[2\sqrt{a}] erfc(-b/2\sqrt{a})
% = e^{b^{2}/ 4a} a^{-1/2}\int_{-b/ 2sqrt(a)} e^{-y^{2}}dy

% \int x^2 e^{-ax^{2}+bx}dx
% = e^{b^{2}/ 4a} (2a+b^{2})(1/4)a^{-5/2}\int_{-b/ 2sqrt(a)} e^{-y^{2}}dy       +b/(2a^2)
% = ...

% erfc(x) = 2/sqrt(pi) \int_{x}^{\infty}e^-t^2  dt      %  t=y/sqrt(2).  dt = dy/sqrt(2).
%          = sqrt(2/pi) \int_{\sqrt{2}x}^{\infty}e^{-t^{2}/2}dt
%     so 1/sqrt(2pi)\int_{x}^{\infty}e^-t^2 /2 dt
%          =   (1/2)erfc(x/sqrt(2))
%
%    and int_x^infty e^-t^2 /2 dt   = sqrt(pi/2)erfc(x/sqrt(2))

When $\rho<.1$, this quantity is upper bounded by $5\rho+8\rho^{2}$.
%
%%%
%rho=linspace(.01,.2,1000);
%plot(rho,    (1./rho).*(    (  exp((1.1 *rho).^2 ./(4 *((1.1*rho).^2 +1/2 - rho + rho.^2 /2).*(1-rho.^2))) .* sqrt((1-rho.^2)./((1.1*rho).^2 + 1/2 - rho + rho.^2 /2)) .*(sqrt(pi)/2).*erfc(1.1*rho .* sqrt((1.1*rho).^2 + 1/2 - rho + rho.^2 /2)./sqrt(1-rho.^2)) )  - (1-rho.^2).*(  exp((1.1*rho).^2 ./(4*((1.1*rho).^2 +1/2 ).*(1-rho.^2))) .* sqrt((1-rho.^2)./((1.1*rho).^2 + 1/2 )) .*(sqrt(pi)/2).*erfc(1.1*rho .* sqrt((1.1*rho).^2 + 1/2 )./sqrt(1-rho.^2)) )  - (1-rho.^2).*rho.*(   .25*(2*(((1.1*rho).^2 +1/2)./(1-rho.^2))  + (1.1*rho./(1-rho.^2)).^2) .*exp((1.1*rho).^2 ./(4*((1.1*rho).^2 +1/2 ).*(1-rho.^2))) .* ((1-rho.^2)./((1.1*rho).^2 + 1/2)).^(5/2) .*(sqrt(pi)/2).*erfc(1.1*rho .* sqrt((1.1*rho).^2 + 1/2 )./sqrt(1-rho.^2))      - 1.1*rho./((1-rho.^2).*2.*( ((1.1*rho).^2 + 1/2)./(1-rho.^2)).^2))  ) , rho, rho.*5 + 8*rho.^2    )
%legend('orig','ub')
%axis([0 1 0 1])
%%

%
%%
%rho=linspace(.01,.99,1000);
%plot(rho,    1./(1-rho).^2  - (1+2*rho)   +(1./rho).*(1./(1-rho.^2)).*sqrt(2*pi).*(  sqrt(1-rho.^2)./(1-rho)   - (1+rho).*(1- rho.^2))   , rho, rho*3*sqrt(pi/2)+ 8*rho.^2    )
%legend('orig','ub')
%axis([0 1 0 1])
%%
\end{proof}

In the following Lemma, if $f\colon\R^{2}\to [0,1]$, we define $\mathrm{Proj}_{1}(f)\colon\R^{2}\to\R$ to be equal to the degree one projection of $f$ onto spherical harmonics of a given radius $\vnorm{x}$.  That is,

\begin{equation}\label{pjdef}
\begin{aligned}
\mathrm{Proj}_{1}(f)(x)
&\colonequals
2\Big(\E_{y\in \vnorm{x}S^{1}}f(y)\cos\Big(\frac{y_{1}}{\vnorm{x}}\Big)\Big)\cos\Big(\frac{x_{1}}{\vnorm{x}}\Big)\\
&\quad+2\Big(\E_{y\in \vnorm{x}S^{1}}f(y)\sin\Big(\frac{y_{2}}{\vnorm{x}}\Big)\Big)\sin\Big(\frac{x_{2}}{\vnorm{x}}\Big),\qquad\forall\,x=(x_{1},x_{2})\in\R^{2}\setminus\{0\}.
\end{aligned}
\end{equation}

%\snote{maybe keep critical point business until later, since}
\begin{lemma}\label{lemma29z}
Let $f,g\colon\R^{2}\to\Delta_{3}$.  Define
$$\phi(x)=\frac{\rho}{1-\rho^{2}}\vnorm{x}e^{-\frac{(1.1\rho\vnorm{x})^{2}}{1-\rho^{2}}-\frac{1.1\rho}{1-\rho^{2}}\vnorm{x}},\qquad\forall\,x\in\R^{2}.$$
If $0<\rho<1/36$, then
\begin{flalign*}
&\abs{\int_{\R^{2}}\langle \mathrm{Proj}_{1}(f),\, T_{\rho}[\mathrm{Proj}_{1}(g)](x)\rangle\gamma_{2}(x)\,\d x}\\
&\leq \Big(\frac{1}{2}\sqrt{\frac{\pi}{2}}+ 1.35\rho\Big)\frac{1}{2}\Big(\int_{\R^{2}}\phi(x)\vnormf{\mathrm{Proj}_{1}(f)(x)}^{2}\gamma_{2}(x)\,\d x +\int_{\R^{2}}\phi(x)\vnormf{\mathrm{Proj}_{1}(g)(x)}^{2}\gamma_{2}(x)\,\d x\Big).
\end{flalign*}
\end{lemma}
\begin{proof}
We will show that
\begin{equation}\label{five1z}
\begin{aligned}
&\abs{\int_{\R^{2}}\langle \mathrm{Proj}_{1}(f),\, T_{\rho}[\mathrm{Proj}_{1}(f)](x)\rangle\gamma_{2}(x)\,\d x}\\
&\qquad\qquad\qquad\leq \Big(\frac{1}{2}\sqrt{\frac{\pi}{2}}+ 1.35\rho\Big)\int_{\R^{2}}\phi(x)\vnorm{\mathrm{Proj}_{1}(f)(x)}^{2}\gamma_{2}(x)\d x.
\end{aligned}
\end{equation}
%\snote{how does it follows from cauchy-schwarz exactly?}
The conclusion of the Lemma then follows from this inequality and the Cauchy-Schwarz inequality.  We therefore prove \eqref{five1z}.  We use the Hermite polynomial notation from Lemma \ref{lemma29}.  Since $\mathrm{Proj}_{1}f$ is an odd function, its Hermite-Fourier expansion is

%Let $h_{0},h_{1},\ldots\colon\R\to\R$ be the Hermite polynomials with $h_{m}(x)=\sum_{k=0}^{\lfloor m/2\rfloor}\frac{x^{m-2k}(-1)^{k}2^{-k}}{k!(m-2k)!}$ $\forall$ $m\geq0$, $m\in\Z$, $\forall$ $x\in\R$.  It is well known \cite{heilman12} that $\{\sqrt{m!}h_{m}\}_{m\geq0}$ is an orthonormal basis of the Hilbert space of functions $\R\to\R$ equipped with the inner product $\langle g,h\rangle\colonequals\int_{\R}g(x)h(x)\gamma_{2}(x)\,\d x$.  For any $k\in\N^{2}$, define $k!\colonequals k_{1}!\cdot k_{2}!$, and let $\vnorm{k}_{1}\colonequals\abs{k_{1}}+\abs{k_{2}}$.  Let $\alpha_{k}\in\R$ to be determined later.

\begin{flalign*}
&\int_{\R^{2}}\langle \mathrm{Proj}_{1}(f),\, T_{\rho}[\mathrm{Proj}_{1}(f)](x)\rangle\gamma_{2}(x)\,\d x\\
&=\sum_{d\geq1\colon d\,\mathrm{odd}}\rho^{d}\sum_{k\in\N^{2}\colon \vnorm{k}_{1}=d}\vnorm{\int_{\R^{2}} \sqrt{k!}h_{k}(x) \mathrm{Proj}_{1}(f)\gamma_{2}(x)\,\d x}^{2}\\
&=\sum_{d\geq1\colon d\,\mathrm{odd}}\rho^{d}\sum_{k\in\N^{2}\colon \vnorm{k}_{1}=d}\vnorm{\int_{\R^{2}} \sqrt{k!}[h_{k}(x)]\frac{1}{\sqrt{\phi(x)}} \sqrt{\phi(x)}\mathrm{Proj}_{1}(f)\gamma_{2}(x)\,\d x}^{2}\\
&\leq\sum_{d\geq1\colon d\,\mathrm{odd}}\rho^{d}\sum_{\substack{k\in\N^{2}\colon\\ \vnorm{k}_{1}=d}}
\int_{\R^{2}} \abs{\sqrt{k!}h_{k}(x)\frac{1}{\sqrt{\phi(x)}}}^{2}\gamma_{2}(x)\,\d x
\cdot\int_{\R^{2}}\phi(x)\vnorm{\mathrm{Proj}_{1}(f)}^{2}\gamma_{2}(x)\,\d x\\
&=\int_{\R^{2}}\phi(x)\vnorm{\mathrm{Proj}_{1}(f)(x)}^{2}\gamma_{2}(x)\,\d x\cdot\sum_{d\geq1\colon d\,\mathrm{odd}}\rho^{d}\sum_{\substack{k\in\N^{2}\colon\\ \vnorm{k}_{1}=d}}
\int_{\R^{2}} \abs{\sqrt{k!}h_{k}(x)\frac{1}{\sqrt{\phi(x)}}}^{2}\gamma_{2}(x)\,\d x.
\end{flalign*}

To conclude, it remains to bound the last term.  Without loss of generality, we may restrict the sum over $k\in\N^{2}$ to $k\in\N^{2}\setminus\{(0,1)\}$ by rotation invariance of the Gaussian measure. We therefore denote
$$\alpha\colonequals\rho\int_{\R^{2}}x_{1}^{2}\frac{1-\rho}{\rho}\frac{1}{\vnorm{x}}e^{\frac{1.1\rho}{1-\rho^{2}}\vnorm{x} +\frac{(1.1\rho \vnorm{x})^{2}}{1-\rho^{2}}}\gamma_{2}(x)\,\d x,$$
and we estimate the last term minus $\alpha$.  Using the expansion of the Mehler kernel from \eqref{gdef} into Hermite polynomials, i.e. $G_{\rho}(x,y)=\gamma_{2}(x)\gamma_{2}(y)\sum_{d\geq0}\rho^{d}\sum_{k\in\N^{2}\colon\vnorm{k}_{1}=d}h_{k}(x)h_{k}(y)k!$,

\begin{flalign*}
&\sum_{d\geq1\colon d\,\mathrm{odd}}\rho^{d}\sum_{\substack{k\in\N^{2}\colon\\ \vnorm{k}_{1}=d}}\int_{\R^{2}} \abs{\sqrt{k!}h_{k}(x)\frac{1}{\sqrt{\phi(x)}}}^{2}\gamma_{2}(x)\,\d x\\
&=\int_{\R^{2}}\frac{1}{\phi(x)}\Big[2\pi\frac{G_{\rho}(x,x)-G_{\rho}(x,-x)}{2e^{-\vnorm{x}^{2}/2}}\Big]\d x
=\int_{\R^{2}}\frac{1}{\phi(x)}\Big[\frac{1}{2\pi}\frac{1}{1-\rho^{2}} e^{-\frac{\vnorm{x}^{2}}{1-\rho^{2}}}
\frac{e^{\frac{\rho\vnorm{x}^{2}}{1-\rho^{2}}}-e^{-\frac{\rho\vnorm{x}^{2}}{1-\rho^{2}}}}{2e^{-\vnorm{x}^{2}/2}}\Big]\, \d x\\
&=\int_{r=0}^{\infty}\frac{1-\rho^{2}}{\rho}\frac{1}{r}e^{\frac{1.1\rho}{1-\rho^{2}}r +\frac{(1.1\rho r)^{2}}{1-\rho^{2}}}\Big[\frac{1}{1-\rho^{2}} e^{-\frac{r^{2}}{1-\rho^{2}}}
\frac{e^{\frac{\rho r^{2}}{1-\rho^{2}}}-e^{-\frac{\rho r^{2}}{1-\rho^{2}}}}{2e^{-r^{2}/2}}\Big]\, r\d r\\
&=\frac{1}{2\rho}\int_{r=0}^{\infty}\Big[e^{-r^{2}\frac{1-\rho}{1-\rho^{2}}  + \frac{r^{2}}{2} +\frac{1.1\rho r+(1.1\rho)^{2}r^{2}}{1-\rho^{2}}}
-e^{-r^{2}\frac{1+\rho}{1-\rho^{2}} + \frac{r^{2}}{2} +\frac{1.1\rho r+(1.1\rho)^{2} r^{2}}{1-\rho^{2}}}\Big]\, \d r\\
%\end{flalign*}
%\begin{flalign*}
&=\frac{1}{2\rho}\int_{r=0}^{\infty}\Big[e^{-\frac{1}{2}r^{2}\frac{2-2\rho-(1-\rho^{2})-2.42\rho^{2}}{1-\rho^{2}}  +\frac{1.1\rho}{1-\rho^{2}}r}
-e^{-\frac{1}{2}r^{2}\frac{2+2\rho-(1-\rho^{2})-2.42\rho^{2}}{1-\rho^{2}} +\frac{1.1\rho}{1-\rho^{2}}r}\Big]\, \d r\\
&=\frac{1}{2\rho}\int_{r=0}^{\infty}\Big[e^{-\frac{1}{2}r^{2}\frac{1-2\rho-1.42\rho^{2}}{1-\rho^{2}}  +\frac{1.1\rho}{1-\rho^{2}}r}
-e^{-\frac{1}{2}r^{2}\frac{1+2\rho-1.42\rho^{2}}{1-\rho^{2}} +\frac{1.1\rho}{1-\rho^{2}}r}\Big]\, \d r\\
&=\frac{1}{2\rho}\int_{r=0}^{\infty}\Big[e^{-\frac{1}{2}\frac{1-2\rho-1.42\rho^{2}}{1-\rho^{2}}\left(r - \frac{1.1\rho}{1-2\rho-1.42\rho^{2}}\right)^{2} +\frac{(1.1\rho)^{2}}{2(1-2\rho-1.42\rho^{2})(1-\rho^{2})}}\\
&\qquad\qquad-e^{-\frac{1}{2}\frac{1+2\rho-1.42\rho^{2}}{1-\rho^{2}}\left(r - \frac{1.1\rho}{1+2\rho-1.42\rho^{2}}\right)^{2} +\frac{(1.1\rho)^{2}}{2(1+2\rho-1.42\rho^{2})(1-\rho^{2})}}\Big]\, \d r\\
&=\frac{1}{2\rho}e^{\frac{(1.1\rho)^{2}}{2(1-2\rho-1.42\rho^{2})(1-\rho^{2})}}\int_{r=-\frac{1.1\rho}{1-2\rho-1.42\rho^{2}}}^{\infty}e^{-\frac{1}{2}\frac{1-2\rho-1.42\rho^{2}}{1-\rho^{2}}r^{2} }\d r\\
&\qquad-\frac{1}{2\rho}e^{\frac{(1.1\rho)^{2}}{2(1+2\rho-1.42\rho^{2})(1-\rho^{2})}}\int_{r=-\frac{1.1\rho}{1+2\rho-1.42\rho^{2}}}^{\infty}e^{-\frac{1}{2}\frac{1+2\rho-1.42\rho^{2}}{1-\rho^{2}}r^{2} }\d r\\
&=\frac{1}{2\rho}e^{\frac{(1.1\rho)^{2}}{2(1-2\rho-1.42\rho^{2})(1-\rho^{2})}}\sqrt{\frac{1-\rho^{2}}{1-2\rho-1.42\rho^{2}}}
\int_{r=-\frac{1.1\rho}{\sqrt{1-2\rho-1.42\rho^{2}}\sqrt{1-\rho^{2}}}}^{\infty}e^{-\frac{1}{2}r^{2} }\d r\\
&\qquad-\frac{1}{2\rho}e^{\frac{(1.1\rho)^{2}}{2(1+2\rho-1.42\rho^{2})(1-\rho^{2})}}\sqrt{\frac{1-\rho^{2}}{1+2\rho-1.42\rho^{2}}}
\int_{r=-\frac{1.1\rho}{\sqrt{1+2\rho-1.42\rho^{2}}\sqrt{1-\rho^{2}}}}^{\infty}e^{-\frac{1}{2}r^{2} }\d r.
\end{flalign*}
When $0<\rho<1/36$, this quantity is bounded by $2.5\rho + \sqrt{\pi/2}$.  The $\alpha$ term is similarly bounded below by $(1/2)[\sqrt{\pi/2} + 2.3\rho]$, so our final bound is at most $(1/2)\sqrt{\pi/2}+ 1.35\rho$.

% erfc(x) = 2/sqrt(pi) \int_{x}^{\infty}e^-t^2  dt      %  t=y/sqrt(2).  dt = dy/sqrt(2).
%          = sqrt(2/pi) \int_{\sqrt{2}x}^{\infty}e^{-t^{2}/2}dt
%     so 1/sqrt(2pi)\int_{x}^{\infty}e^-t^2 /2 dt
%          =   (1/2)erfc(x/sqrt(2))
%
%    and int_x^infty e^-t^2 /2 dt   = sqrt(pi/2)erfc(x/sqrt(2))
%%%
%rho=linspace(.005,1/36,1000);
%plot(rho, (1./(2*rho)).* sqrt((1- rho.^2)./(1-2*rho - 1.42* rho.^2)).*exp(((1.21)*rho.^2)./(2.*(1-2*rho - 1.42* rho.^2).*(1-rho.^2))) .* sqrt(pi/2) .* erfc( (1/sqrt(2)) .* - 1.1*rho./sqrt((1-2*rho - 1.42*rho.^2).*(1-rho.^2)))-(1./(2*rho)).* sqrt((1- rho.^2)./(1+2*rho - 1.42* rho.^2)).*exp(((1.21)*rho.^2)./(2.*(1+2*rho - 1.42* rho.^2).*(1-rho.^2))) .* sqrt(pi/2) .* erfc( (1/sqrt(2)) .* - 1.1*rho./sqrt((1+2*rho - 1.42*rho.^2).*(1-rho.^2))) , rho, sqrt(pi/2) + 2.3*rho)
%axis([0 1/36 1 2])
%legend('orig','ub')

%%%%%%%%%
%%%%%%%%%
%function  integralest
%numpts=1000;
%rho=linspace(.0001 ,1/36,numpts);
%
%for i=1:numpts
%    plotme(i) = myint(rho(i));
%end
%figure
%lb=(1/2)*(sqrt(pi/2) + 2.2*rho) ;
%plot(rho,plotme,rho, lb);
%legend('orig','lb')
%figure;
%plot(rho, plotme ./lb );
%legend('orig/lb'); axis([0 100 0 2])
%    if  sum(plotme - lb <0)==0  % then the inequality is verified
%       fprintf('Verified\r');
%    end
%end
%
%function out=myint(rho)
%f= @(a) (1-rho^2) * (1/2) * (a.^2) .* exp( (-a.^2 /2) + (1.21* rho^2 * a.^2)/(1-rho^2)  + (1.1*rho*a)/(1-rho^2));
%out = integral(f,0,20);
%end

\end{proof}

\begin{lemma}\label{lastlem}
$$\int_{S^{1}}\vnorm{\mathrm{Proj}_{1}(f-h)}^{2}\d\sigma\leq .9555\sum_{i=1}^{3}\Big(\frac{\theta_{i}}{2\pi}-1/3\Big)^{2}.$$
\end{lemma}
Lemma \ref{lastlem} is verified with the following Matlab program, where the constant $.9555$ is an upper bound of the vlalue $2(4.715)/\pi^{2}$.
%4.715* 2 / pi^2
%
%%%%%%%
\begin{verbatim}
numpts=200;
x=linspace(0 ,2*pi ,numpts);
y=linspace(0 ,2*pi ,numpts);
xv=ones(numpts,1)*x;
yv= y' * ones(1,numpts);
zv=  ((sin(xv/2) - sqrt(3)/2).^2  + (sin(yv/2) - sqrt(3)/2).^2 ...
    + (sin((2*pi - xv - yv)/2) - sqrt(3)/2).^2  ).*(xv+yv<= 2*pi);
zz= ( ( (xv/(2*pi)) - 1/3).^2  +  ( (yv/(2*pi)) - 1/3).^2 +  ...
    ( ((2*pi - xv - yv)/(2*pi)) - 1/3).^2   ).*(xv+yv<= 2*pi);
hold on;
surf(x,y, zv - 4.715*zz);
if sum(sum( zv-4.715*zz >0))==0
    fprintf('Verified\r');  % verify zv< 4.7 zz
end
\end{verbatim}

\section{Circular Rearrangement}

In this section, we prove the key lemma, that the bilinear noise stability on the sphere $S^{1}$ (minus the measure of the set) is uniquely maximized for a partition into three congruent arcs.  After rearranging, this follows from Corollary \ref{cor1}.

In the following Lemma, we say $A_{1},A_{2},A_{3}\subset S^{1}$ is a \textbf{partition} of $S^{1}$ if $A_{i}\cap A_{j}=\emptyset$ for all $1\leq i<j\leq 3$ and $\cup_{i=1}^{3}A_{i}= S^{1}$.

\begin{lemma}[\embolden{Circular Rearrangement, Maximization}]\label{lemma3}
Let $0<\rho<1$.  Let $r,s>0$.  Let $A_{1},A_{2},A_{3}$ be a partition of $S^{1}$.  Let $B_{1},B_{2},B_{3}$ be a partition of $S^{1}$.  Let $\sigma$ denote the normalized (Haar) probability measure on $S^{1}$, and define $g=g_{\rho,r,s}$ by \eqref{g2def}.  Then
\begin{flalign*}
&\sum_{i=1}^{3}\int_{S^{1}}[1_{A_{i}}(x)-\sigma(A_{i})]U_{g}[1_{B_{i}}-\sigma(B_{i})](x)\,\d\sigma(x)\\
&\qquad\qquad\leq \sum_{i=1}^{3}\int_{S^{1}}[1_{D_{i}}(x)-\sigma(D_{i})]U_{g}[1_{D_{i}}-\sigma(D_{i})](x)\,\d\sigma(x)\\
&-\Big(1-\Big(1-\frac{3^{4/3}}{5}\Big) e^{-\frac{\rho rs}{1-\rho^{2}}\frac{\pi}{2}}-\frac{3^{4/3}}{5}e^{-\frac{\rho rs}{1-\rho^{2}}\frac{\pi}{6}}\Big)\frac{(.109)}{2}\sum_{i=1}^{3}\Big((\sigma(A_{i})-1/3)^{2} +(\sigma(B_{i})-1/3)^{2} \Big).
\end{flalign*}
where $D_{1},D_{2},D_{3}$ is a partition of $S^{1}$ into three congruent circular arcs (each with angle $2\pi/3$).
\end{lemma}
\begin{proof}
We first show that we may assume $A_{i}=B_{i}$ for all $1\leq i\leq 3$.  Since $\lambda_{d,2}^{r,s}>0$ for all $d\geq0$ by \eqref{two3z}, we can write $U_{g}=(\sqrt{U_{g}})^{2}$, where $\sqrt{U_{g}}$ is a positive definite operator.  Using then the Cauchy-Schwarz inequality and $ab\leq (a^{2}+b^{2})/2$ for all $a,b\in\R$, we have
\begin{flalign*}
&\sum_{i=1}^{3}\int_{S^{1}}[1_{A_{i}}(x)-\sigma(A_{i})]U_{g}[1_{B_{i}}-\sigma(B_{i})](x)\,\d\sigma(x)\\
&\qquad=\sum_{i=1}^{3}\int_{S^{1}}\sqrt{U_{g}}[1_{A_{i}}(x)-\sigma(A_{i})](x)\sqrt{U_{g}}[1_{B_{i}}-\sigma(B_{i})](x)\,\d\sigma(x)\\
&\qquad\leq\sum_{i=1}^{3}\Big(\int_{S^{1}}[\sqrt{U_{g}}[1_{A_{i}}(x)-\sigma(A_{i})](x)]^{2}\Big)^{1/2}\Big(\int_{S^{1}}[\sqrt{U_{g}}[1_{B_{i}}-\sigma(B_{i})](x)]^{2}\,\d\sigma(x)\Big)^{1/2}\\
&\qquad\leq\frac{1}{2}\sum_{i=1}^{3}\int_{S^{1}}[1_{A_{i}}(x)-\sigma(A_{i})]U_{g}[1_{A_{i}}-\sigma(A_{i})](x)\,\d\sigma(x)\\
&\qquad\qquad\qquad\qquad+\frac{1}{2}\sum_{i=1}^{3}\int_{S^{1}}[1_{B_{i}}(x)-\sigma(B_{i})]U_{g}[1_{B_{i}}-\sigma(B_{i})](x)\,\d\sigma(x).
\end{flalign*}

So, the general case of the Lemma follows from the special case that $A_{i}=B_{i}$ for all $1\leq i\leq 3$.  We now proceed with this assumption.  Fix $1\leq i\leq 3$.  Spherical rearrangement with Definition \ref{corsph} (using e.g. \cite[Theorem 2]{baernstein76} and that $t\mapsto e^{\rho rst/(1-\rho^{2})}$ is increasing in $t$) implies that
\begin{flalign*}
&\sum_{i=1}^{3}\int_{S^{1}}[1_{A_{i}}(x)-\sigma(A_{i})]U_{g}[1_{A_{i}}-\sigma(A_{i})](x)\,\d\sigma(x)\\
&\qquad\qquad\leq \sum_{i=1}^{3}\int_{S^{1}}[1_{H_{i}}(x)-\sigma(H_{i})]U_{g}[1_{H_{i}}-\sigma(H_{i})](x)\,\d\sigma(x).
\end{flalign*}
where $H_{i}\subset S^{1}$ is a spherical arc such that $\sigma(A_{i})=\sigma(H_{i})$.

Corollary \ref{cor1} therefore concludes the proof.
\end{proof}

\begin{remark}
We note in passing that three $120$ degree arcs have larger value of
\begin{equation}\label{one3}
\begin{aligned}
&\sum_{i=1}^{3}\int_{S^{1}}[1_{H_{i}}(x)-\sigma(H_{i})]U_{g}[1_{H_{i}}-\sigma(H_{i})](x)\,\d\sigma(x)\\
&\qquad\qquad\qquad=\sum_{i=1}^{3}\Big[\int_{S^{1}}1_{H_{i}}(x)U_{g}(1_{H_{i}})(x)\,\d\sigma(x)-(\sigma(H_{i}))^{2}\Big].
\end{aligned}
\end{equation}
then two $180$ degree arcs, i.e. we claim
\begin{equation}\label{three1}
\frac{9}{2\pi^{2}}\cdot\sum_{\substack{d\geq1\colon d\equiv1\,\mathrm{mod}\,3\,\,\mathrm{or}\\  \hphantom{}\quad d\equiv2\,\mathrm{mod}\,3}}\lambda_{d,2}^{r,s}\frac{1}{d^{2}}
>\frac{4}{\pi^{2}}\cdot\sum_{d\geq1\colon d\,\,\mathrm{odd}}\lambda_{d,2}^{r,s}\frac{1}{d^{2}}
,\qquad\forall\,0<\rho<1,\quad\forall\,r,s>0.
\end{equation}
This fact follows from Corollary \ref{cor1}, but the explicit computation is still enlightening (especially when $\rho<0$, in which case \eqref{three1} is false for $r\cdot s$ large.)

Since $\lambda_{d,2}^{r,s}\geq\lambda_{d+1}^{r,s}$ for all $d\geq0$ by \eqref{four10}, \eqref{three1} follows by a straightforward term-by-term monotonicity of the terms in \eqref{three1}.  That is, Denote $\N_{1,2}\colonequals\{d\geq1\colon d\equiv1\,\mathrm{mod}\,3\,\,\mathrm{or}\,\,\,d\equiv2\,\mathrm{mod}\,3\}$ and let $\N_{\mathrm{odd}}$ denote odd positive integers.  Define $J\colon \N_{1,2}\to\N_{\mathrm{odd}}$ as the order preserving map with $J(1)=1$ and recursively, $J(d)\colonequals\min\{y\in\N_{\mathrm{odd}}\colon y\notin\cup_{d'\in\N_{1,2}}\{J(d')\}\}$.  For example, $J(1)=1$, $J(2)=3$, $J(4)=5$, $J(5)=7$, $J(7)=9$, $J(8)=11$, and so on.  Since $\lambda_{d,2}^{r,s}\geq\lambda_{d+1}^{r,s}$ for all $d\geq0$ by \eqref{four10}, since $J$ is surjective and $J(d)\geq d$ for all $d\in\N_{1,2}$ each term in the right sum of \eqref{three1} has a corresponding larger or equal term on the left side of \eqref{three1}.  We conclude that \eqref{three1} holds.  The Lemma follows.
\end{remark}

\section{A Convolution Bound}

%\snote{something is still wrong here, $r=0$ doesn't give right answer}
\begin{lemma}\label{lemma8}
Let $r>0$.  Let $a\colonequals \rho r/(1-\rho^{2})$.  Let $\phi(x)\colonequals 1-e^{-a\vnorm{x}}$ for all $x\in\R^{2}$.  Then
$$
T_{\rho}\phi(r)\geq
1-e^{-\frac{\rho^{2}r^{2}}{2[1-\rho^{2}]}}\int_{t=0}^{t=\infty}\gamma_{1}(t)\,\d t
-e^{\frac{3}{2}\frac{\rho^{2}r^{2}}{1-\rho^{2}}}\int_{t=-\infty}^{t=-\frac{2\rho r}{\sqrt{1-\rho^{2}}}}\gamma_{1}(t)\,\d t.
$$
\end{lemma}
\begin{proof}
From Definition \ref{oudef} we have
$$T_{\rho}\phi(x)=\int_{\R^{2}}(1-e^{-a\vnormf{\rho x+y\sqrt{1-\rho^{2}}}})\gamma_{2}(y)\,\d y,\qquad\forall\,x\in\R^{2}.$$
When $y\in\R^{2}$ is fixed, if $z\in\R^{2}$ is perpendicular to $\rho x+y\sqrt{1-\rho^{2}}$, then $\vnormf{\rho x+(y+z)\sqrt{1-\rho^{2}}}\geq \vnormf{\rho x+y\sqrt{1-\rho^{2}}}$.  Therefore,
\begin{flalign*}
T_{\rho}\phi(x)
&\geq\int_{\R}(1-e^{-a\vnormf{\rho x + t\frac{x}{\vnorm{x}}\sqrt{1-\rho^{2}}}})\gamma_{1}(t)\,\d t
=\int_{\R}(1-e^{-a\abs{\rho\vnorm{x}+t\sqrt{1-\rho^{2}}}})\gamma_{1}(t)\,\d t\\
&=1-\int_{t=-\frac{\rho\vnorm{x}}{\sqrt{1-\rho^{2}}}}^{t=\infty}e^{-a\abs{\rho\vnorm{x}+t\sqrt{1-\rho^{2}}}}\gamma_{1}(t)\,\d t
-\int_{t=-\infty}^{t=-\frac{\rho\vnorm{x}}{\sqrt{1-\rho^{2}}}}e^{-a(\abs{\rho\vnorm{x}+t\sqrt{1-\rho^{2}}}}\gamma_{1}(t)\,\d t\\
&=1-\int_{t=-\frac{\rho\vnorm{x}}{\sqrt{1-\rho^{2}}}}^{t=\infty}e^{-a(\vnorm{x}\rho+t\sqrt{1-\rho^{2}})}\gamma_{1}(t)\,\d t
-\int_{t=-\infty}^{t=-\frac{\rho\vnorm{x}}{\sqrt{1-\rho^{2}}}}e^{a(\vnorm{x}\rho+t\sqrt{1-\rho^{2}})}\gamma_{1}(t)\,\d t\\
&=1-e^{-a\rho\vnorm{x}}\int_{t=-\frac{\rho\vnorm{x}}{\sqrt{1-\rho^{2}}}}^{t=\infty}e^{-at\sqrt{1-\rho^{2}}}\gamma_{1}(t)\,\d t
-e^{a\rho\vnorm{x}}\int_{t=-\infty}^{t=-\frac{\rho\vnorm{x}}{\sqrt{1-\rho^{2}}}}e^{at\sqrt{1-\rho^{2}}}\gamma_{1}(t)\,\d t.
\end{flalign*}
Then, using $e^{-t^{2}/2}e^{\lambda t}=e^{-(1/2)(t-\lambda)^{2}}e^{\lambda^{2}/2}$,
\begin{flalign*}
T_{\rho}\phi(x)
&\geq
1-e^{-a\rho\vnorm{x}+\frac{1}{2}a^{2}(1-\rho^{2})}\int_{t=-\frac{\rho\vnorm{x}}{\sqrt{1-\rho^{2}}}+a\sqrt{1-\rho^{2}}}^{t=\infty}\gamma_{1}(t)\,\d t\\
&\qquad\qquad
-e^{a\rho\vnorm{x}+\frac{1}{2}a^{2}(1-\rho^{2})}\int_{t=-\infty}^{t=-\frac{\rho\vnorm{x}}{\sqrt{1-\rho^{2}}}-a\sqrt{1-\rho^{2}}}\gamma_{1}(t)\,\d t.
\end{flalign*}
Plugging in the definition of $a=\rho r/(1-\rho^{2})$, we get
\begin{flalign*}
T_{\rho}\phi(x)
&\geq
1-e^{-\rho^{2}r\vnorm{x}/(1-\rho^{2})+\frac{1}{2}\rho^{2}r^{2}/(1-\rho^{2})}\int_{t=\frac{-\rho\vnorm{x}+\rho r}{\sqrt{1-\rho^{2}}}}^{t=\infty}\gamma_{1}(t)\,\d t\\
&\qquad\qquad
-e^{\rho^{2}r\vnorm{x}/(1-\rho^{2})+\frac{1}{2}\rho^{2}r^{2}/(1-\rho^{2})}\int_{t=-\infty}^{t=\frac{-\rho\vnorm{x}-\rho r}{\sqrt{1-\rho^{2}}}}\gamma_{1}(t)\,\d t\\
&\geq
1-e^{-\frac{\rho^{2}r}{1-\rho^{2}}(\vnorm{x}-r/2)}\int_{t=-\frac{\rho(\vnorm{x}-r)}{\sqrt{1-\rho^{2}}}}^{t=\infty}\gamma_{1}(t)\,\d t
-e^{\frac{\rho^{2}r}{1-\rho^{2}}(\vnorm{x}+r/2)}\int_{t=-\infty}^{t=-\frac{\rho(\vnorm{x}+r)}{\sqrt{1-\rho^{2}}}}\gamma_{1}(t)\,\d t.
\end{flalign*}
In particular, when $\vnorm{x}=r$, we get
\begin{flalign*}
T_{\rho}\phi(r)
&\geq
1-e^{-\frac{\rho^{2}r^{2}}{2[1-\rho^{2}]}}\int_{t=0}^{t=\infty}\gamma_{1}(t)\,\d t
-e^{\frac{3}{2}\frac{\rho^{2}r^{2}}{1-\rho^{2}}}\int_{t=-\infty}^{t=-\frac{2\rho r}{\sqrt{1-\rho^{2}}}}\gamma_{1}(t)\,\d t.
\end{flalign*}
\end{proof}

\begin{lemma}\label{convbd2}
Let $r>0$ and let $0<\rho<1$.  Let $f\colon\R^{2}\to\R$ be defined by
$$f(x)\colonequals \frac{\rho r\vnorm{x}}{1-\rho^{2}}e^{-\rho r\vnorm{x}/(1-\rho^{2})},\qquad\forall\,x\in\R^{2}.$$
For all $r>0$ and for all $0<\rho<.05$,
$$T_{\rho}f(r,0)\geq 1.2 \frac{\rho r}{1-\rho^{2}} e^{-\frac{(1.1\rho r)^{2}}{1-\rho^{2}} -1.1\frac{\rho r}{1-\rho^{2}}}.$$
\end{lemma}
\begin{verbatim}
function  convolveexp
%plot the 2d convolve of the function
%f(x) = (rho r|x|/ 1-rho^2) exp(- (rho r|x|/ 1-rho^2))
% since it is a radial function, suffices to do a one-dimensional plot
rho=.05;
numpts=1000;
r=linspace(0 ,10,numpts);

for i=1:numpts
    plotme(i) = comyfcn(r(i),0,rho);
end
lb = 1.2*(rho*r/(1-rho^2)) .* exp(-((1.1*rho*r).^2)/(1-rho^2) ...
        - 1.1*(rho*r / (1-rho^2)));
plot(r,plotme,r, lb );
legend('orig','lb')
    if  sum(plotme - lb <0)==0  % then the inequality is verified
       fprintf('Verified\r');
    end
end

function out=comyfcn(x,y,rho)
%2d convolve the function
%f(x) = (rho r|x|/ 1-rho^2) exp(- (rho r|x|/ 1-rho^2))
% T_rho f(z) = int f(rho z + ysqrt(1-rho^2))gamma(y)
% output is T_rho f(x,y)
r=sqrt(x.^2 + y.^2);
rhor = rho*r / (1-rho^2);
f= @(a,b) rhor*sqrt(a.^2 + b.^2) .* exp(- rhor * sqrt(a.^2 + b.^2));
myfun = @(a,b) (f(rho*x + a*sqrt(1-rho^2), rho*y + b*sqrt(1-rho^2)))...
    .*exp(-(a.^2 + b.^2)/2) ;
out = (1/(2*pi))* integral2(myfun, -10, 10, -10, 10);
end
\end{verbatim}

\section{Proof of Main Theorem: Positive Correlation}

\begin{proof}[Proof of Theorem \ref{thm1} when $\rho>0$]

The main result of \cite{heilman20d} implies that, since $k=3$, we can and will assume that $\Omega_{1},\Omega_{2},\Omega_{3}\subset\R^{2}$.  That is, the case $n=2$ is sufficient to prove the case $n>2$ in Theorem \ref{thm1}.

Denote $f\colon\R^{2}\to\{(1,0,0),(0,1,0),(0,0,1)\}$ by $(f(x))_{i}\colonequals 1_{\Omega_{i}}(x)$ for all $1\leq i\leq 3$.
Denote also $f_{r}\colon\R^{2}\to \Delta_{2}$ by $f_{r}(x)\colonequals f|_{rS^{1}}$, i.e. $f_{r}$ is $f$ restricted to the sphere $rS^{1}$, and denote $\E f_{r}$ as the average value of $f_{r}$ on $rS^{1}$ (with respect to the normalized Haar probability measure $\sigma$ on $S^{1}$.)  Using Definitions \ref{nsdef} and \ref{gausdef} we then write (recalling notation from Section \ref{enote})
\begin{equation}\label{one1}
\begin{aligned}
&\sum_{i=1}^{3}\int_{\R^{2}}1_{\Omega_{i}}(x)T_{\rho}1_{\Omega_{i}}(x) \gamma_{2}(x)\,\d x
=\E_{X\sim_{\rho}Y}\langle f(X),f(Y)\rangle
=\E_{R,S}\E_{(U,V)\sim N_{\rho}^{R,S}}\langle f_{R}(U),f_{S}(V)\rangle\\
&\qquad\qquad\qquad\qquad=\E_{R,S}\E_{(U,V)\sim N_{\rho}^{R,S}}\langle f_{R}(U)- \E f_{R}+\E f_{R},f_{S}(V)-\E f_{S}+\E f_{S}\rangle\\
&\qquad\qquad\qquad\qquad=\E_{R,S}\Big( \langle\E f_{R},\E f_{S}\rangle + \E_{(U,V)\sim N_{\rho}^{R,S}}\langle f_{R}(U)- \E f_{R},f_{S}(V)-\E f_{S}\rangle\Big).
\end{aligned}
\end{equation}

Let $\Theta_{1},\Theta_{2},\Theta_{3}\subset\R^{2}$ be three disjoint cones centered at the origin, each with cone angle $2\pi/3$.  Define then $h\colon\R^{2}\to\Delta_{3}$ by $h(x)\colonequals(1_{\Theta_{1}}(x),1_{\Theta_{3}}(x),1_{\Theta_{3}}(x))$ for all $x\in \R^{2}$.

We first upper bound the rightmost term in \eqref{one1}.  From Corollary \ref{cor1} we have, $\forall$ $r,s>0$,
\begin{equation}\label{two1}
\begin{aligned}
&\underset{(U,V)\sim N_{\rho}^{r,s}}{\E}\langle f_{r}(U) -\E f_{r},\,f_{s}(V)- \E f_{s}\rangle
\leq\underset{(U,V)\sim N_{\rho}^{r,s}}{\E}\langle h_{r}(U) -\E h_{r},\,h_{s}(V)- \E h_{s}\rangle\\
&-\Big(-1+\Big(1-\frac{3^{4/3}}{5}\Big) e^{-\frac{\rho rs}{1-\rho^{2}}\frac{\pi}{2}}+\frac{3^{4/3}}{5}e^{-\frac{\rho rs}{1-\rho^{2}}\frac{\pi}{6}}\Big)\\
&\qquad\qquad\qquad\qquad\qquad\cdot\frac{(.109)}{2}\sum_{i=1}^{3}\Big((\sigma(\Omega_{i}\cap r S^{1})-1/3)^{2} +(\sigma(\Omega_{i}\cap s S^{1})-1/3)^{2} \Big).
\end{aligned}
\end{equation}
Then, using Lemma \ref{lemma8}, we get
\begin{equation}\label{two6}
\begin{aligned}
&\underset{R,S}{\E}\underset{(U,V)\sim N_{\rho}^{R,S}}{\E}\langle f_{R}(U) -\E f_{R},\,f_{S}(V)- \E f_{S}\rangle
\leq\underset{R,S}{\E}\underset{(U,V)\sim N_{\rho}^{R,S}}{\E}\langle h_{R}(U) -\E h_{R},\,h_{S}(V)- \E h_{S}\rangle\\
&\,-\underset{R}{\E}\Big[1-\Big(1-\frac{3^{4/3}}{5}\Big)e^{-\frac{\rho^{2}R^{2}(\pi/2)^{2}}{2[1-\rho^{2}]}}\int_{t=0}^{t=\infty}\gamma_{1}(t)\,\d t
-\Big(1-\frac{3^{4/3}}{5}\Big)e^{\frac{3}{2}\frac{\rho^{2}R^{2}(\pi/2)^{2}}{1-\rho^{2}}}\int_{t=-\infty}^{t=-\frac{2\rho R(\pi/2)}{\sqrt{1-\rho^{2}}}}\gamma_{1}(t)\,\d t\\
&\,-\frac{3^{4/3}}{5}e^{-\frac{\rho^{2}R^{2}(\pi/6)^{2}}{2[1-\rho^{2}]}}\int_{t=0}^{t=\infty}\gamma_{1}(t)\,\d t
-\frac{3^{4/3}}{5}e^{\frac{3}{2}\frac{\rho^{2}R^{2}(\pi/6)^{2}}{1-\rho^{2}}}\int_{t=-\infty}^{t=-\frac{2\rho R(\pi/6)}{\sqrt{1-\rho^{2}}}}\gamma_{1}(t)\,\d t\Big]\\
&\qquad\cdot
(.109)\sum_{i=1}^{3}(\sigma(\Omega_{i}\cap R\cdot S^{1})-1/3)^{2}.
\end{aligned}
\end{equation}
%%%
%%  erfc(x)= (2/sqrt(pi))\int_x^infty exp(-t^2)dt   .   t=y/sqrt(2).  dt = dy/sqrt(2)
%%           = 2/sqrt(2pi)\int_{sqrt(2)x}^infty exp(-y^2/2)dy
%%  so normcdf(x) = .5*erfc(x/sqrt(2))

Using Matlab, we simplify this bound to $(1- e^{-\rho R/2})$, i.e.
\begin{verbatim}
rho=.1;
r=linspace(0,150,10000);
y=.109 - (.109)*(1- (3^(4/3))/5)*( (exp(-(rho^2)*(r.^2)*(pi/2)^2 ...
 /(2*(1-rho^2))))*(.5)  + .5*(exp((3/2)*(rho^2)*(r.^2)*(pi/2)^2 ...
  /(1-rho^2))).*erfc(  2*rho*r*(pi/2)/sqrt(2*(1-rho^2))))  ...
   -(.109)*((3^(4/3))/5)*( (exp(-(rho^2) *(r.^2)*(pi/6)^2 ...
   /(2*(1-rho^2))))*(.5) + .5*(exp((3/2)*(rho^2)*(r.^2)*(pi/6)^2 ...
  /(1-rho^2))).*erfc(  2*rho*r*(pi/6)/sqrt(2*(1-rho^2))));
plot(r, y, r,  .109*(1-exp(-r*rho/2))   );
axis([0 150 0 .109]);
legend('original','lower bound');
if sum((y - .109*(1-exp(-r*rho/2))<0))==0, fprintf('Verified\r'), end;
\end{verbatim}
\begin{equation}\label{two8}
\begin{aligned}
&\underset{R,S}{\E}\underset{(U,V)\sim N_{\rho}^{R,S}}{\E}\langle f_{R}(U) -\E f_{R},\,f_{S}(V)- \E f_{S}\rangle
\leq\underset{R,S}{\E}\underset{(U,V)\sim N_{\rho}^{R,S}}{\E}\langle h_{R}(U) -\E h_{R},\,h_{S}(V)- \E h_{S}\rangle\\
&\qquad - .109\cdot \underset{R}{\E}(1- e^{- R\rho/2})\sum_{i=1}^{3}(\sigma(\Omega_{i}\cap R\cdot S^{1})-1/3)^{2}.
\end{aligned}
\end{equation}

Or, rewriting using our notation for $f$,
\begin{equation}\label{two9}
\begin{aligned}
&\underset{R,S}{\E}\underset{(U,V)\sim N_{\rho}^{R,S}}{\E}\langle f_{R}(U) -\E f_{R},\,f_{S}(V)- \E f_{S}\rangle
\leq\underset{R,S}{\E}\underset{(U,V)\sim N_{\rho}^{R,S}}{\E}\langle h_{R}(U) -\E h_{R},\,h_{S}(V)- \E h_{S}\rangle\\
&\qquad - (.109)\cdot \underset{R}{\E}(1- e^{- R\rho/2})\sum_{i=1}^{3}\vnorm{f_{R}-\mathbf{1}/3}^{2}.
\end{aligned}
\end{equation}

Here we denote $\mathbf{1}\colonequals(1,1,1)\in\R^{3}$.  We now bound the $\langle\E f_{R},\E f_{S}\rangle$ term in \eqref{one1}.  Using the assumption that $\E_{R,S}f_{R}=\E_{R}f_{R}=\mathbf{1}/3$, we have
\begin{equation}\label{two2}
\begin{aligned}
\underset{R,S}{\E}\langle\E f_{R},\E f_{S}\rangle
&=\underset{R,S}{\E}\langle\E f_{R} - \mathbf{1}/3+\mathbf{1}/3,\E f_{S} - \mathbf{1}/3+\mathbf{1}/3\rangle\\
&=1/3 + \underset{R,S}{\E}\langle\E f_{R} - \mathbf{1}/3 ,\E f_{S} - \mathbf{1}/3 \rangle\\
&=1/3 + \underset{X\sim_{\rho} Y}{\E}\langle\E f_{\vnorm{X}} - \mathbf{1}/3 ,\E f_{\vnorm{Y}} - \mathbf{1}/3 \rangle.
%&\leq1/3 + \rho^{2}\underset{X\sim_{\rho} Y}{\E}\vnorm{\E f_{\vnorm{X}} - \mathbf{1}/3 }^{2}.
\end{aligned}
\end{equation}

Combining \eqref{two2} and \eqref{two9} into \eqref{one1}, we get

\begin{equation}
\begin{aligned}
&\sum_{i=1}^{3}\int_{\R^{2}}1_{\Omega_{i}}(x)T_{\rho}1_{\Omega_{i}}(x) \gamma_{2}(x)\,\d x
-\sum_{i=1}^{3}\int_{\R^{2}}1_{\Theta_{i}}(x)T_{\rho}1_{\Theta_{i}}(x) \gamma_{2}(x)\,\d x\\
&\qquad\leq \underset{X\sim_{\rho} Y}{\E}\langle\E f_{\vnorm{X}} - \mathbf{1}/3 ,\E f_{\vnorm{Y}} - \mathbf{1}/3 \rangle
\quad - (.109)\cdot \underset{R}{\E}(1- e^{- R\rho/2})\sum_{i=1}^{3}\vnorm{f_{R}-\mathbf{1}/3}^{2}.
\end{aligned}
\end{equation}

Lemma \ref{lemma28} with $\E_{\gamma}f=\mathbf{1}/3$ then bounds the penultimate term, yielding

\begin{equation}\label{two10}
\begin{aligned}
&\sum_{i=1}^{3}\int_{\R^{2}}1_{\Omega_{i}}(x)T_{\rho}1_{\Omega_{i}}(x) \gamma_{2}(x)\,\d x
-\sum_{i=1}^{3}\int_{\R^{2}}1_{\Theta_{i}}(x)T_{\rho}1_{\Theta_{i}}(x) \gamma_{2}(x)\,\d x\\
&\qquad\leq  (2.5(\rho+\rho^{2}) -.109)\underset{R}{\E}(1- e^{- R\rho/2})\sum_{i=1}^{3}\vnorm{f_{R}-\mathbf{1}/3}^{2}.
\end{aligned}
\end{equation}

When $0<\rho<.0418$, we have $(2.5\rho+2.5\rho^{2} -.109)<0$, so the right side of \eqref{two10} is nonpositive, with equality only when $f_{R}=\mathbf{1}/3$.  The proof is therefore concluded by \eqref{two10}, when $\rho$ satisfies $0<\rho<.0418$.
\end{proof}

\begin{remark}\label{rk1comp}
Using computer assistance, we can replace the constant $.109$ above with $.3$, so, with computer assistance, the main theorem holds for all $0<\rho$ such that $2.5(\rho+\rho^{2})<.3$, i.e. $0<\rho<\frac{1}{10}(\sqrt{37}-5)=.108\ldots$.  In particular, Theorem \ref{thm1} holds for all $0<\rho<1/10$.
\end{remark}
%(sqrt(85)-5)/20

\begin{verbatim}
function  surfplotfcn
    rho=.01;
numpts=100;
x=linspace(0 ,2*pi ,numpts);
y=linspace(0 ,2*pi ,numpts);
xv=ones(numpts,1)*x;
yv= y' * ones(1,numpts);

zv= (-3*myfun(2*pi/3,rho)+  myfun(xv,rho) + myfun(yv,rho) ...
    + myfun(2*pi-xv-yv,rho)).*(xv+yv<= 2*pi);

hold on;
pv=(1/3- xv/(2*pi)).^2  + (1/3- yv/(2*pi)).^2 + (2/3 - (xv+yv)/(2*pi)).^2;
surf(x,y, zv +.3*pv.*(xv+yv <=2*pi)*(besseli(1,rho))/besseli(0,rho));
%%% this surface being nonpositive demonstrates that a constant
%%%  .3 can be used in

end

function out=myfun(th,rho)
%output is mean subtracted noise stability of an interval [0,th]
% with positive correlation

k=30; % number of terms in expansion to use
out=zeros(size(th));
for i=1:k
    out=out + (1/i)^2 *(besseli(i,rho)/besseli(0,rho))*(sin(th*i/2)).^2;
end
out = out*2/pi^2;
end
\end{verbatim}

Likewise, we check the inequality \eqref{two8} with the constant $.3$ instead of $.109$.

\begin{verbatim}
rho=.04;
r=linspace(0,150,10000);
y=.3 - (.3)*(1- (3^(4/3))/5)*( (exp(-(rho^2)*(r.^2)*(pi/2)^2 ...
 /(2*(1-rho^2))))*(.5)  + .5*(exp((3/2)*(rho^2)*(r.^2)*(pi/2)^2 ...
  /(1-rho^2))).*erfc(  2*rho*r*(pi/2)/sqrt(2*(1-rho^2))))  ...
   -(.3)*((3^(4/3))/5)*( (exp(-(rho^2) *(r.^2)*(pi/6)^2 ...
   /(2*(1-rho^2))))*(.5) + .5*(exp((3/2)*(rho^2)*(r.^2)*(pi/6)^2 ...
  /(1-rho^2))).*erfc(  2*rho*r*(pi/6)/sqrt(2*(1-rho^2))));
plot(r, y, r,  .3*(1-exp(-r*rho/2))   );
axis([0 150 0 .3]);
legend('original','lower bound');
if sum((y - .109*(1-exp(-r*rho/2))<0))==0, fprintf('Verified\r'), end;
\end{verbatim}

\section{Proof of Main Theorem: Negative Correlation}

\begin{proof}[Proof of Theorem \ref{thm1} when $\rho<0$]

Let $0<\rho<1$.  In this proof, we will minimize the quantity
$$\sum_{i=1}^{3}\int_{\R^{n}}1_{\Omega_{i}}(x)T_{\rho}1_{\Omega_{i}'}(x)\gamma_{n}(x)\,\d x$$
over all (measurable) partitions $\Omega_{1},\Omega_{2},\Omega_{3}\subset\R^{n}$ and $\Omega_{1}',\Omega_{2}',\Omega_{3}'\subset\R^{n}$ (so e.g. $\cup_{i=1}^{3}\Omega_{i}=\R^{n}$ and $\cup_{i=1}^{3}\Omega_{i}'=\R^{n}$) subject to the constraint that
\begin{equation}\label{seven0}
\gamma_{n}(\Omega_{i})=\gamma_{n}(\Omega_{i}'),\qquad\forall\,1\leq i\leq 3.
\end{equation}

The main result of \cite{heilman20d} (as adapted in \cite{heilman22d}) implies that, since $k=3$, we can and will assume that $\Omega_{1},\Omega_{2},\Omega_{3},\Omega_{1}',\Omega_{2}',\Omega_{3}'\subset\R^{2}$. % That is, the case $n=2$ is sufficient to prove the case $n>2$ in Theorem \ref{thm1}

We will demonstrate that $\Omega_{i}=-\Omega_{i}'=\Theta_{i}$ for all $1\leq i\leq 3$, where $\Theta_{1},\Theta_{2},\Theta_{3}$ are three disjoint sectors (cones) each with cone angle $2\pi/3$ centered at the origin.  Since $T_{-\rho}f(x)=T_{\rho}f(-x)$, this result proves the negative correlation case of Theorem \ref{thm1}, i.e. the case $\rhocurn\leq\rho<0$ of Theorem \ref{thm1}.  Within the current proof, we assume that $0<\rho\leq\rhocurnabs$.

Denote $f,g\colon\R^{2}\to\{(1,0,0),(0,1,0),(0,0,1)\}$ by $(f(x))_{i}\colonequals 1_{\Omega_{i}}(x)$ and  $(g(x))_{i}\colonequals 1_{\Omega_{i}'}(x)$ for all $1\leq i\leq 3$.  Define also $h\colon\R^{2}\to\Delta_{3}$ by $h(x)\colonequals(1_{\Theta_{1}}(x),1_{\Theta_{3}}(x),1_{\Theta_{3}}(x))$ for all $x\in \R^{2}$.

Denote also $f_{r}\colon\R^{2}\to \Delta_{2}$ by $f_{r}(x)\colonequals f|_{rS^{1}}$, i.e. $f_{r}$ is $f$ restricted to the sphere $rS^{1}$, and denote $\E f$ as the average value of $f_{r}$ on $rS^{1}$ (with respect to the normalized Haar probability measure on $S^{1}$.)  Using Definitions \ref{nsdef} and \ref{gausdef}, and also \eqref{four7Z},  (recalling notation from Section \ref{enote})
\begin{equation}\label{seven1}
\begin{aligned}
&\sum_{i=1}^{3}\int_{\R^{2}}1_{\Omega_{i}}(x)T_{\rho}1_{\Omega_{i}'}(x) \gamma_{2}(x)\,\d x
=\E_{X\sim_{\rho}Y}\langle f(X),f(Y)\rangle
=\E_{R,S}\E_{(U,V)\sim N_{\rho}^{R,S}}\langle f_{R}(U),g_{S}(V)\rangle\\
&=\E_{R,S}\Big(\langle \E f_{R},\E g_{S}\rangle+ \lambda_{1,n}^{R,S}\frac{1}{2\pi^{2}}\Big\langle\int_{S^{1}}f_{R}(x)\cos(x)\,\d x,\, \int_{S^{1}}g_{S}(y)\cos(y)\,\d y\Big\rangle \\
&\qquad\qquad+ \lambda_{1,n}^{R,S}\frac{1}{2\pi^{2}}\Big\langle\int_{S^{1}}f_{R}(x)\sin(x)\,\d x,\,\int_{S^{1}}g_{S}(y)\sin(y)\,\d y\Big\rangle\\
&\qquad+\sum_{d=2}^{\infty}\lambda_{d,n}^{R,S}\frac{1}{2\pi^{2}}
\Big(\Big\langle\int_{S^{1}}f_{R}(x)\cos(xd)\,\d x,\, \int_{S^{1}}g_{S}(y)\cos(yd)\,\d y\Big\rangle\\
&\qquad\qquad+\Big\langle\int_{S^{1}}f_{R}(x)\sin(xd)\,\d x,\, \int_{S^{1}}g_{S}(y)\sin(yd)\,\d y\Big\rangle\Big).
\end{aligned}
\end{equation}
We introduce a notation for each of the forms on the right.  We rewrite \eqref{seven1} as
\begin{equation}\label{ten1}
Q(f,g)= Q_{0}(f,g)+Q_{1}(f,g)+Q_{2}(f,g).
\end{equation}
That is, $Q(f,g)\colonequals\E_{R,S}\E_{(U,V)\sim N_{\rho}^{R,S}}\langle f_{R}(U),g_{S}(V)\rangle$, $Q_{0}(f,g)\colonequals \E_{R,S}\langle \E f_{R},\E g_{S}\rangle$,
\begin{flalign*}
Q_{1}(f,g)
&\colonequals \E_{R,S}\Big(\lambda_{1,n}^{R,S}\frac{1}{2\pi^{2}}\Big\langle\int_{S^{1}}f_{R}(x)\cos(x)\,\d x,\, \int_{S^{1}}g_{S}(y)\cos(y)\,\d y\Big\rangle \\
&\qquad\qquad+ \lambda_{1,n}^{R,S}\frac{1}{2\pi^{2}}\Big\langle\int_{S^{1}}f_{R}(x)\sin(x)\,\d x,\,\int_{S^{1}}g_{S}(y)\sin(y)\,\d y\Big\rangle\Big),
\end{flalign*}
and $Q_{2}(f,g)$ denotes the remaining $d\geq2$ terms in \eqref{seven1}.

Denote $h_{-}(\cdot)\colonequals h(-\cdot)$.  Since $Q_{1}$ is bilinear, we have
\begin{flalign*}
Q_{1}(f,g)&= Q_{1}(f-h+h, g- h_{-}+ h_{-})\\
&= Q_{1}(f-h,g- h_{-}) + Q_{1}(f-h,h_{-}) + Q_{1}(h,g-h_{-}) + Q_{1}(h,h_{-})\\
&= Q_{1}(f-h,g- h_{-}) + Q_{1}(f,h_{-}) + Q_{1}(g,h) - Q_{1}(h,h_{-})\\
&= Q_{1}(f-h,g-h_{-}) + Q_{1}(f,h_{-}) + Q_{1}(g,h) - Q_{1}(h,h_{-}).
\end{flalign*}
And similarly for $Q_{0}$ and $Q_{2}$.  Using this identity in \eqref{ten1}, we therefore have
\begin{equation}\label{five0}
\begin{aligned}
&Q(f,g)- Q(h,h_{-})\\
&\stackrel{\eqref{ten1}}{=} \sum_{i=0}^{2}\Big(Q_{i}(f-h,g-h_{-}) + Q_{i}(f,h_{-}) - Q_{i}(h,h_{-})+ Q_{i}(g,h) - Q_{i}(h_{-},h)\Big)\\
&=\sum_{i=0}^{2}\Big(Q_{i}(f-h,g-h_{-}) + Q_{i}(f-h,h_{-}) + Q_{i}(g-h_{-},h) \Big).
\end{aligned}
\end{equation}
% first terms deal with via change of measure (with new phi from i=1 case)
% for next terms, i=1 term is largest
%  i=0 term is nonnegative
%  i=2 term has another change of measure?  i forget.  need to think again about the linear case (i.e. when g=h-).  was i thinking about subtracting mean terms or something?
%          in the linear case, since g=h-, we must have Ef = 1/3 1/3 1/3,  since we always constrain Ef=Eg.

\textbf{Plan of the Proof}.  We will bound the terms in \eqref{five0} separately.  We begin with the last two terms when $i=1,2$, which we bound from below.  Then, we will bound the first term in \eqref{five0} with $i=1$, demonstrating this term is smaller than the previously mentioned terms.  Finally, after subtracting a mean term, the remaining terms in \eqref{five0} will be shown to be small, i.e. $O(\rho)$ (whereas the previously mentioned terms are $O(1)$, as $\rho\to0$.)

\textbf{Step 1.}  We bound the last two $i=1,2$ terms from below in \eqref{five1}.  We numerically verify the following inequality with the subsequent Matlab program.
\begin{equation}\label{five2}
\begin{aligned}
&\sum_{i=1}^{2}\Big( Q_{i}(f-h,h_{-})  + Q_{i}(g-h_{-},h) \Big)\\
&\qquad\geq.3133\cdot\E_{R,S} \Big(\frac{\rho RS}{1-\rho^{2}}\cdot e^{- \frac{\rho RS}{1-\rho^{2}}}\Big)
\sum_{i=1}^{3}\Big((\sigma(\Omega_{i}\cap R S^{1})-1/3)^{2}+\sigma(\Omega_{i}'\cap S S^{1})-1/3)^{2}\Big).
\end{aligned}
\end{equation}
%  (2*pi)^2 / 125.3 ~ .31507

\begin{verbatim}
function  surfplotfcnnegcorlinear
    rho=10;
numpts=140;
x=linspace(0 ,2*pi ,numpts);
y=linspace(0 ,2*pi ,numpts);
xv=ones(numpts,1)*x;
yv= y' * ones(1,numpts);

%%%%  mean subtracted noise stability, linear
zv= (-3*nmyfun(2*pi/3,-rho)+nmyfun(xv,-rho) + nmyfun(yv,-rho) ...
    +nmyfun((2*pi - xv - yv),-rho)).*(xv+yv<= 2*pi);
surf(xv,yv,zv);
hold on;
con=rho*exp(-rho) / 125.3;
surf(xv,yv,(con)*( (xv- 2*pi/3).^2 + (yv- 2*pi/3).^2 +...
     (2*pi - xv - yv - 2*pi/3).^2).*(xv+yv<= 2*pi))
if  sum(sum( ( zv>0 & zv < (con)*( (xv- 2*pi/3).^2 +...
     (yv- 2*pi/3).^2 + (2*pi - xv - yv - 2*pi/3).^2).*(xv+yv<= 2*pi) )))==0
    fprintf('Verified\r')
end

end

function out=nmyfun(th,rho)
%output is bilinear mean subtracted noise stability of an interval [0,th]
% against an interval [-pi/3 pi/3]

k=30; % number of terms in expansion to use
out=zeros(size(th));
for i=1:k
    out=out + (1/i)^2 *(besseli(i,rho)/besseli(0,rho))...
        *(sin(th*i/2)).*sin(i*pi/3);
end
out = out*2/pi^2;
end
\end{verbatim}

Then, using Lemma \ref{convbd2} in \eqref{five2}, we have
\begin{equation}\label{five2.5}
\begin{aligned}
&\sum_{i=1}^{2}\Big( Q_{i}(f-h,h_{-})  + Q_{i}(g-h_{-},h) \Big)\\
&\qquad\geq.3133(1.2)\cdot\E_{R} \phi(R)
\sum_{i=1}^{3}\Big((\sigma(\Omega_{i}\cap R S^{1})-1/3)^{2}+\sigma(\Omega_{i}'\cap R S^{1})-1/3)^{2}\Big),
\end{aligned}
\end{equation}
where
$$\phi(r)\colonequals\frac{\rho}{1-\rho^{2}}re^{-\frac{(1.1\rho r)^{2}}{1-\rho^{2}}-\frac{1.1\rho}{1-\rho^{2}} r},\qquad\forall\,r>0.$$

\textbf{Step 2.}  We bound the first $i=1$ term in \eqref{five1}, using Lemmas \ref{lemma29z} and \ref{lastlem} to get
\begin{equation}\label{five3}
\begin{aligned}
&\abs{Q_{1}(f-h,g-h_{-})}\\
&\leq\frac{1}{2}(.9555)(\frac{1}{2}\sqrt{\frac{\pi}{2}}+1.35\rho)\cdot\E_{R} \phi(R)
\sum_{i=1}^{3}\Big((\sigma(\Omega_{i}\cap R S^{1})-1/3)^{2}+\sigma(\Omega_{i}'\cap R S^{1})-1/3)^{2}\Big).
\end{aligned}
\end{equation}

\textbf{Step 3.}  We bound the remaining terms in \eqref{five0}, i.e. the last $i=0$ term, and the first $i=0,2$ terms.

The last $i=0$ terms in \eqref{five0} are zero, since $h$ and $h_{-}$ have Haar measure $1/3$ assigned to each partition element, on a sphere of any radius centered at the origin.  The first $i=0$ term in \eqref{five0} is equal to a nonnegative term plus a term that is bounded by Lemma \ref{lemma29}.  Finally, the first $i=2$ term in \eqref{five0} is also bounded via Lemma \ref{lemma29}.  In fact, we can combine the first $i=0,2$ terms in \eqref{five0} and estimate them as a single term.  Using the notation of Lemma \ref{lemma29z} (i.e. \eqref{pjdef}), Let $\mathrm{Proj}(f)\colon\R^{2}\to\R^{2}$ denote $f$ minus its degree one projection onto spherical harmonics, i.e.
\begin{equation}\label{pdef}
\mathrm{Proj}(f)\stackrel{\eqref{pjdef}}{\colonequals} f-\mathrm{Proj}_{1}(f).
\end{equation}
We then have
\begin{equation}\label{five9}
\begin{aligned}
&\sum_{i=0,2}Q_{i}(f-h,g-h_{-})
=\E_{X\sim_{\rho}Y}\langle\mathrm{Proj}(f-h),\mathrm{Proj}(g-h_{-})\rangle\\
&=\vnorm{\E_{\gamma}(f-h)}^{2}+\sum_{d\geq2}\rho^{d}\sum_{k\in\N^{2}\colon \vnorm{k}_{1}=d}\\
&\qquad\Big\langle\int_{\R^{2}} \sqrt{k!}h_{k}(x) \mathrm{Proj}(f-h)(x)\gamma_{2}(x)\,\d x,\,\int_{\R^{2}} \sqrt{k!}h_{k}(x) \mathrm{Proj}(g-h_{-})(x)\gamma_{2}(x)\,\d x\Big\rangle.
\end{aligned}
\end{equation}

Applying Lemma \ref{lemma29} to \eqref{five9}, we then get

\begin{equation}\label{five8}
\begin{aligned}
&\abs{\sum_{i=0,2}Q_{i}(f-h,g-h_{-})\,\,- \vnorm{\E_{\gamma}(f-h)}^{2}}\\
&\leq(5\rho+8\rho^{2})\frac{1}{2}\cdot\E_{X\sim\gamma} \phi(\vnorm{X})
\Big(\vnorm{\mathrm{Proj}(f-h)(X)}^{2}+\vnorm{\mathrm{Proj}(g-h_{-})(X)}^{2}\Big).
\end{aligned}
\end{equation}
Since $\mathrm{Proj}$ is a contraction on each sphere centered at the origin, \eqref{five8} implies that

\begin{equation}\label{five4}
\begin{aligned}
&\abs{\sum_{i=0,2}Q_{i}(f-h,g-h_{-})\,\,- \vnorm{\E_{\gamma}(f-h)}^{2}}\\
&\leq(5\rho+8\rho^{2})\frac{1}{2}\cdot\E_{X\sim\gamma} \phi(\vnorm{X})
\Big(\vnorm{f(X)-h(X)}^{2}+\vnorm{g(X)-h_{-}(X)}^{2}\Big).
\end{aligned}
\end{equation}

Combining \eqref{five0}, \eqref{five2.5} and \eqref{five4},

\begin{flalign*}
&\sum_{i=1}^{3}\int_{\R^{2}}1_{\Omega_{i}}(x)T_{\rho}1_{\Omega_{i}'}(x) \gamma_{2}(x)\,\d x
-\sum_{i=1}^{3}\int_{\R^{2}}1_{\Theta_{i}}(x)T_{\rho}1_{\Theta_{i}'}(x) \gamma_{2}(x)\,\d x\\
&\stackrel{\eqref{seven1}\wedge\eqref{five0}}{\geq}\vnorm{\E_{\gamma}(f-h)}^{2}  \quad
+(.3133)(1.2)\E_{R} \phi(R)\cdot(\vnorm{\E f_{R}-\mathbf{1}/3}^{2}+\vnorm{\E g_{R}-\mathbf{1}/3}^{2})\\
&\qquad
-\frac{1}{2}(.9555)\Big(\frac{1}{2}\sqrt{\frac{\pi}{2}}+1.35\rho\Big)\cdot(\E_{R}\phi(R)(\vnorm{\E f_{R} - \E h_{R}}^{2} + \vnorm{\E g_{R} - \E h_{R}}^{2})\\
&\qquad
-\frac{1}{2}(5\rho+8\rho^{2})\cdot (\E_{X\sim\gamma}\phi(\vnorm{X})(\vnorm{ f(X) - h(X)}^{2} + \vnorm{g(X)-h_{-}(X)}^{2}).
\end{flalign*}
Spherical rearrangement with Definition \ref{corsph} (using e.g. \cite[Theorem 2]{baernstein76} and that $t\mapsto e^{\rho rst/(1-\rho^{2})}$ is increasing in $t$) implies that we may assume that $\Omega_{i}\cap r S^{1}$ and $\Omega_{i}'\cap r S^{1}$ are opposing circular arcs for all $r>0$.  (After this rearrangement, $\Omega_{i}$ and $\Omega_{j}$ might have an intersection with positive measure for some $i,j$, but we still have $\sum_{i=1}^{3}\sigma(\Omega_{i}\cap r S^{1})=1$ for all $r>0$.  Also, after this rearrangement, we may assume that $\Omega_{i}\cap r S^{1}$ and $\Theta_{i}\cap r S^{1}$ are circular arcs with the same center of mass in $rS^{1}$.)  So, if we write the last $\E_{X\sim\gamma}$ term in polar coordinates, fix $\vnorm{X}$, and first average in the angular direction, we find the average of $\vnorm{ f(X) - h(X)}^{2}$ on a sphere of radius $\vnorm{X}$ is equal to $\vnorm{\E f_{\vnorm{X}} - \E h_{\vnorm{X}}}^{2}$.  Also, $\E h_{R}=\mathbf{1}/3$ by definition of $h$, so that all of the terms involving $f-h$ and $g-h$ can be written in the same way, i.e.

\begin{equation}\label{eight1}
\begin{aligned}
&\sum_{i=1}^{3}\int_{\R^{2}}1_{\Omega_{i}}(x)T_{\rho}1_{\Omega_{i}'}(x) \gamma_{2}(x)\,\d x
-\sum_{i=1}^{3}\int_{\R^{2}}1_{\Theta_{i}}(x)T_{\rho}1_{\Theta_{i}'}(x) \gamma_{2}(x)\,\d x\geq\vnorm{\E_{\gamma}(f-h)}^{2}\\
&\qquad+\E_{R} \Big(\phi(R)\cdot(\vnorm{\E f_{R}-\mathbf{1}/3}^{2}+\vnorm{\E g_{R}-\mathbf{1}/3}^{2})\cdot\Big[.3759 - .3  - .645\rho -2.5\rho - 4\rho^{2}\Big]\Big).
\end{aligned}
\end{equation}
% .3133*1.2 ~ .3759
%.25*.9555*sqrt(pi/2) < .3
% .5* .9555 * 1.35 < .645
%
%solve .0759- 3.145*x - 4*x^2 =0
The right side is nonnegative for all $0<\rho<\rhocurnabs\ldots$, with equality only when $f=h$ and $g=-h$.

\end{proof}

\section{Proof of Unique Games Hardness}

\begin{proof}[Proof of Theorem \ref{thm4}]
From \cite[Theorem A.9]{isaksson11}: assuming the Unique Games Conjecture, for any $\epsilon>0$, it is NP-hard to approximate MAX-3-CUT within a multiplicative factor of $\beta_{3}+\epsilon$ where
$$
\beta_{3}
\colonequals
\lim_{n\to\infty}\inf_{-\frac{1}{2}\leq\rho\leq 1}\sup_{f\colon\R^{n}\to\Delta_{3}}\frac{3}{2}\frac{1-\sum_{i=1}^{3}\int_{\R^{n}}\langle f(x), T_{\rho}f(x)\rangle\,\gamma_{n}(x)\,\d x}{1-\rho}.
$$
The main result of \cite{heilman20d} (as adapted in \cite{heilman22d}) implies that, since $k=3$, the quantity $\beta_{3}$ does not depend on $n\geq4$, i.e. we can write
$$
\beta_{3}=\inf_{-\frac{1}{2}\leq\rho\leq 1}\sup_{f\colon\R^{4}\to\Delta_{3}}\frac{3}{2}\frac{1-\sum_{i=1}^{3}\int_{\R^{4}}\langle f(x), T_{\rho}f(x)\rangle\,\gamma_{4}(x)\,\d x}{1-\rho}.
$$
Also, \cite[Lemma A.4]{isaksson11} implies that the infimum is attained when $\rho\leq0$, i.e.
$$
\beta_{3}=\inf_{-\frac{1}{2}\leq\rho\leq 0}\sup_{f\colon\R^{4}\to\Delta_{3}}\frac{3}{2}\frac{1-\sum_{i=1}^{3}\int_{\R^{4}}\langle f(x), T_{\rho}f(x)\rangle\,\gamma_{4}(x)\,\d x}{1-\rho}.
$$
Since this is an infimum, we have an upper bound by taking an infimum over a smaller set:
$$
\beta_{3}\leq\inf_{\rhocurn\leq\rho\leq 1}\sup_{f\colon\R^{4}\to\Delta_{3}}\frac{3}{2}\frac{1-\sum_{i=1}^{3}\int_{\R^{4}}\langle f(x), T_{\rho}f(x)\rangle\,\gamma_{4}(x)\,\d x}{1-\rho}.
$$
Now our main result Theorem \ref{thm1} together with an explicit formula for the noise stability of the Plurality function from \cite{klerk04} implies that
\begin{flalign*}
\beta_{3}
&\leq\inf_{\rhocurn\leq\rho\leq 0}\sup_{f\colon\R^{2}\to\Delta_{3}}\frac{3}{2}\frac{1-\sum_{i=1}^{3}\int_{\R^{2}}\langle f(x), T_{\rho}f(x)\rangle\,\gamma_{2}(x)\,\d x}{1-\rho}\\
&=\inf_{\rhocurn\leq\rho\leq 0}\frac{3}{2}\cdot\frac{1-3((1/9)+ \frac{[\arccos(-\rho)]^{2} - [\arccos(\rho/2)]^{2}}{4\pi^{2}})}{1-\rho}\\
&=\frac{3}{2}\cdot \frac{1-3((1/9)+ \frac{[\arccos(\rhocurnabs)]^{2} - [\arccos((\rhocurn)/2)]^{2}}{4\pi^{2}})}{1-(\rhocurn)}
\approx .98937199597\ldots.
\end{flalign*}
The above function of $\rho$ is monotone.  It attains its minimum at the endpoint $\rho=\rhocurn$.
%rho=-1/36
%(3/2)*(1 - 3*(  (1/9)+ (((acos(-rho)).^2 - (acos(rho/2)).^2) / (4* pi^2))) )./(1-rho)
%rho=-.0234
%(3/2)*(1 - 3*(  (1/9)+ (((acos(-rho)).^2 - (acos(rho/2)).^2) / (4* pi^2))) )./(1-rho)

\end{proof}

\noindent
\textbf{Acknowledgement}.  We thank Elchanan Mossel for suggesting that Theorem \ref{thm4} should follow from Theorem \ref{thm1}.

\bibliographystyle{amsalpha}
\newcommand{\etalchar}[1]{$^{#1}$}
\def\polhk#1{\setbox0=\hbox{#1}{\ooalign{\hidewidth
  \lower1.5ex\hbox{`}\hidewidth\crcr\unhbox0}}} \def\cprime{$'$}
  \def\cprime{$'$}
\providecommand{\bysame}{\leavevmode\hbox to3em{\hrulefill}\thinspace}
\providecommand{\MR}{\relax\ifhmode\unskip\space\fi MR }
% \MRhref is called by the amsart/book/proc definition of \MR.
\providecommand{\MRhref}[2]{%
  \href{http://www.ams.org/mathscinet-getitem?mr=#1}{#2}
}
\providecommand{\href}[2]{#2}

\end{document}